\documentclass[11pt,a4paper]{article}

\usepackage[utf8]{inputenc}
\usepackage{amssymb}
\usepackage{amsmath}
\usepackage{mathtools}
\usepackage{amsfonts}
\usepackage{amsthm}
\usepackage{bbm}
\usepackage{bm}
\usepackage{hyperref}
\usepackage{physics}
\usepackage{graphicx}

\usepackage{color}
\definecolor{red}{rgb}{1,0,0}

\definecolor{gre}{rgb}{0,0.7,0}

\definecolor{blu}{rgb}{0,0,1}

\definecolor{bla}{rgb}{0,0,0}

\usepackage{enumerate}

\usepackage[dvips, left=2.4cm, right=2.4cm , top=2.4cm, bottom=3.4cm]{geometry}

\numberwithin{equation}{section}

\newcommand{\e}[0]{\ensuremath{\mathrm{e}}}

\makeatletter
\newcommand{\pushright}[1]{\ifmeasuring@#1\else\omit\hfill$\displaystyle#1$\fi\ignorespaces}
\newcommand{\pushleft}[1]{\ifmeasuring@#1\else\omit$\displaystyle#1$\hfill\fi\ignorespaces}
\makeatother

\newtheorem{theorem}{Theorem}[section]
\newtheorem{corollary}[theorem]{Corollary}
\newtheorem{lemma}[theorem]{Lemma}
\newtheorem{proposition}[theorem]{Proposition}

\theoremstyle{definition}

\theoremstyle{remark}

\def\EE{{\mathbb E}}

\def\HH{{\mathbb H}}

\def\PP{{\mathbb P}}
\def\RR{{\mathbb R}}
\def\TT{{\mathbb T}}

\def\ZZ{{\mathbb Z}}

\def\scrN{{\mathcal N}}

\def\scrS{{\mathcal S}}

\def\ii{{\mathrm i}}
\def\ee{{\mathrm e}}

\def\pmod{\bmod}

\usepackage{autonum}

\begin{document}

\title{Fine-scale distribution of roots of quadratic congruences\footnote{Research supported by EPSRC grant EP/S024948/1}}
\author{Jens Marklof and Matthew Welsh}
\date{School of Mathematics, University of Bristol, Bristol BS8 1UG, U.K.\\[10pt] 25 July 2022}

\maketitle

\begin{abstract}
  We establish limit laws for the distribution in small intervals of the roots of the quadratic congruence $\mu^2 \equiv D \bmod m$, with $D > 0$ square-free and $D\not\equiv 1 \bmod 4$. This is achieved by translating the problem to convergence of certain geodesic random line processes in the hyperbolic plane. This geometric interpretation allows us in particular to derive an explicit expression for the pair correlation density of the roots.
\end{abstract}

\tableofcontents

\section{Introduction}
\label{sec:introduction}

The object of this paper is the fine-scale distribution of solutions (the ``roots'') $\mu$ of the quadratic congruence $\mu^2\equiv D\pmod m$, where $D$ is a fixed non-zero integer and $m$ runs through the positive integers. For each $m$, we consider the set of normalised roots in $\TT=\RR/\ZZ$, 
\begin{equation}
Q_D(m) = \bigg\{ \frac{\mu}{m} \in \TT :  \mu^2 \equiv D \pmod m \bigg\} .
\end{equation}
We order the points in $Q_D(m)$ in a specified (but arbitrary) finite sequence, and then construct an infinite sequence $(\xi_j)_{j=1}^\infty$ in $\TT$ by first listing the points in $Q_D(1)$, then those in $Q_D(2)$, and so on. The first significant result in this setting is due to Hooley \cite{Hooley1963}, who proved uniform distribution modulo one if $D$ is not a perfect square. That is, for any $0\leq a < b\leq 1$ we have that
\begin{equation}
\lim_{N\to\infty} \frac{1}{N} \#\big\{ j\leq N :  \xi_j \in [a,b)+\ZZ \big\}  = b-a.
\end{equation}
Hooley later extended this result to general polynomial congruences of degree two or higher \cite{Hooley1964}.
His method has since been generalised and extended, see \cite{KowalskiSoundararajan2020} and \cite{Zehavi2020}.

Having established uniform distribution, it is natural to ask for finer measures that capture the pseudo-random properties of the sequence. The simplest of these is the pair (or two-point) correlation measure $R_{2,N}$, which for a finite interval $I$ is defined by
\begin{equation}\label{pairco}
R_{2,N}(I) = \frac{1}{N} \#\big\{ i, j\leq N :  i\neq j,\; \xi_i-\xi_j \in N^{-1} I +\ZZ \big\} .
\end{equation}
Our first principal result is the following. 

\begin{theorem}\label{thm1}
Assume $D>0$ is square-free and $D\not\equiv 1\pmod 4$. Then there is an even and continuous function $w_D:\RR\to\RR_{\geq 0}$, such that 
for every finite interval $I$ we have
\begin{equation}
\lim_{N\to\infty} R_{2,N}(I) = \int_I w_D(v) \dd v .
\end{equation}
\end{theorem}

In other words, the pair correlation measure $R_{2,N}$ converges vaguely to a limit with bounded continuous density $w_D$. Our geometric approach will allow us to derive an explicit formula for $w_D$; three illustrations are given in figures \ref{fig:paircorrelation1}--\ref{fig:paircorrelation3}. Compare this with the case of independent and uniformly distributed random variables in $\TT$ where the limit density is $w(v)=1$; or with the eigenphases of large random unitary matrices where the limiting pair correlation is $w(v)=1- (\sin(\pi v)/\pi v)^2$. 

The conditions that $D$ is squarefree and $D \not\equiv 1\pmod 4$ in Theorem \ref{thm1} ensure that the associated quadratic order $\mathbb{Z}[\sqrt{D}]$ is maximal, which will provide some simplifications in the proofs.

\begin{figure}[t]
  \centering
  \includegraphics[width=\textwidth]{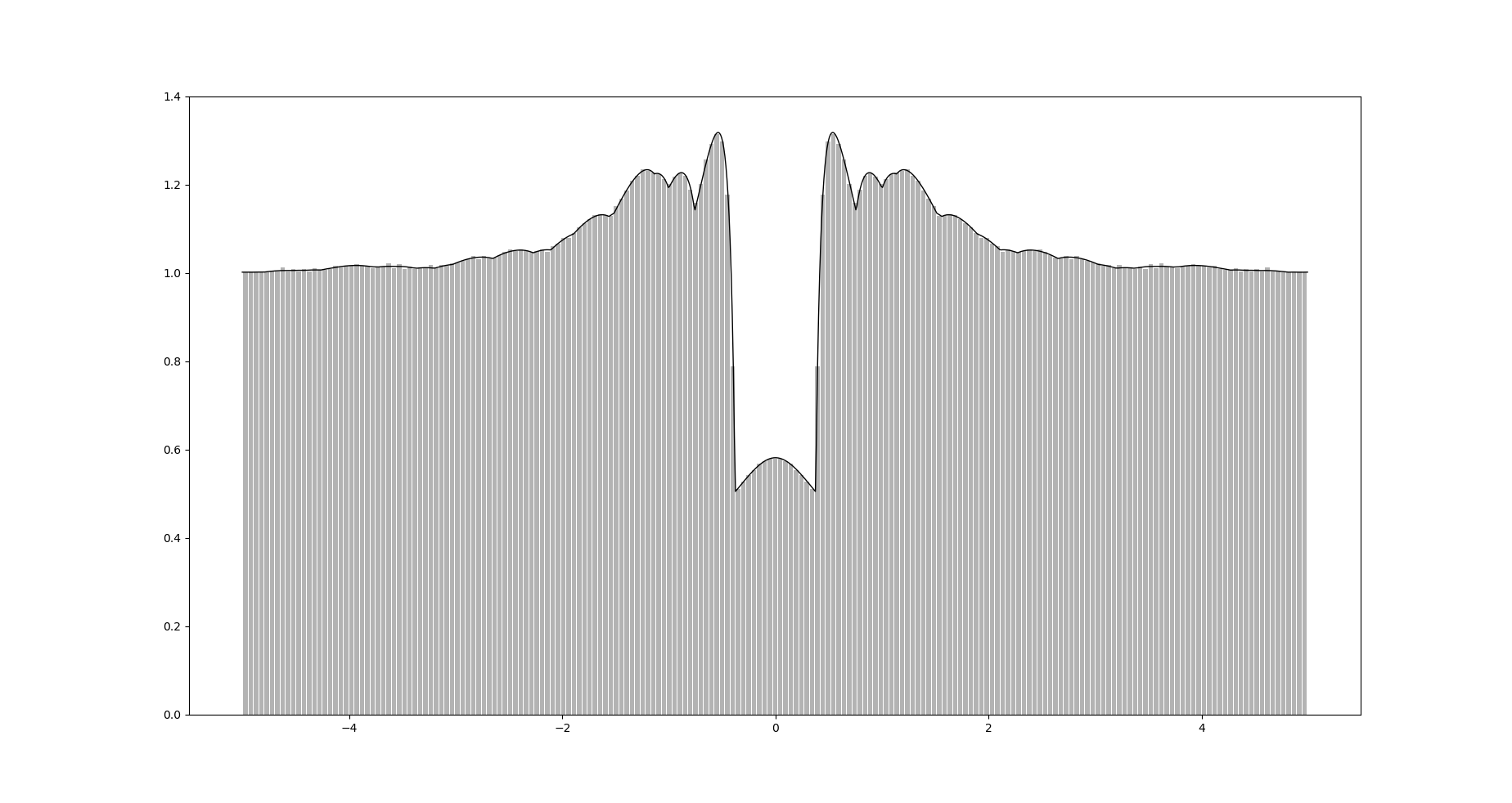}
  \caption{Pair correlation density $w_2$ compared to experiments with $N=10^6$. The ring of integers $\mathbb{Z}[\sqrt{2}]$ has both class and narrow class number $1$.}
  \label{fig:paircorrelation1}
\end{figure}

\begin{figure}[t]
  \centering
  \includegraphics[width=\textwidth]{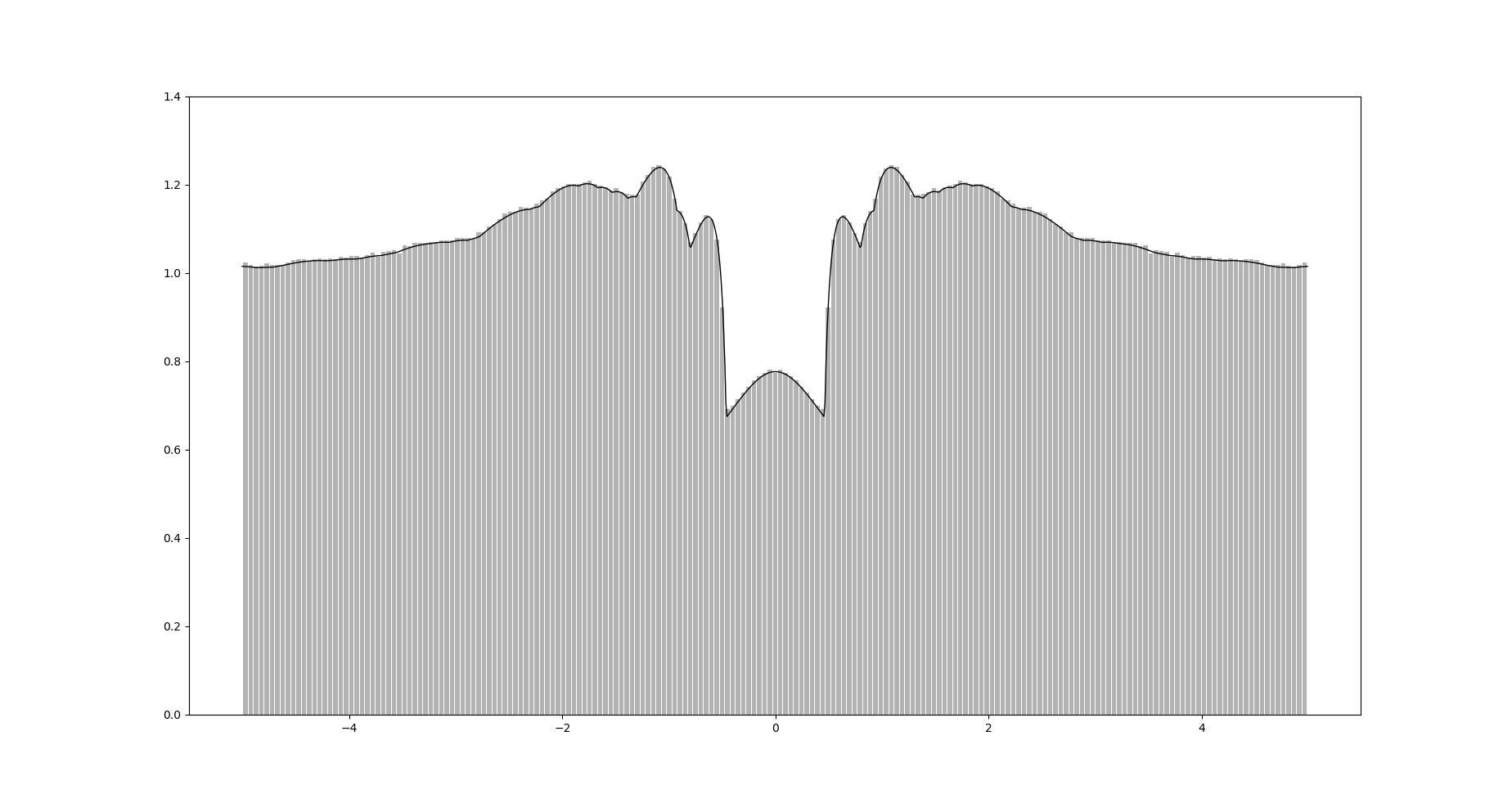}
  \caption{Pair correlation density $w_3$ compared to experiments with $N=10^6$. The ring of integers $\mathbb{Z}[\sqrt{3}]$ has class number $1$ and narrow class number $2$.}
  \label{fig:paircorrelation2}
\end{figure}

\begin{figure}[t]
  \centering
  \includegraphics[width=\textwidth]{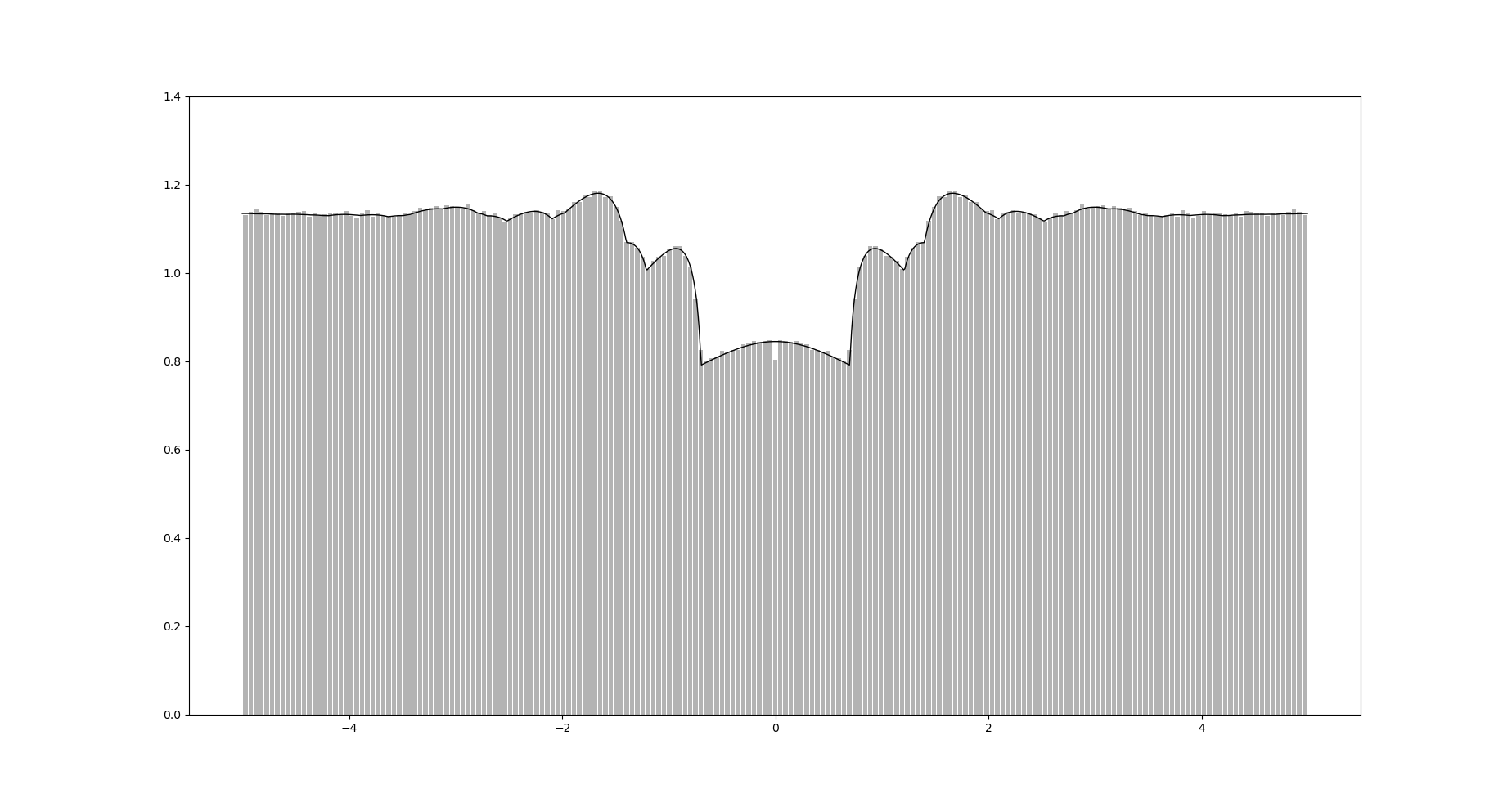}
  \caption{Pair correlation density $w_{10}$ compared to experiments with $N=10^6$. The ring of integers $\mathbb{Z}[\sqrt{10}]$ has both class and narrow class number $2$.}
  \label{fig:paircorrelation3}
\end{figure}

Our techniques are in fact strong enough to produce results for all higher-order correlation measures. Let $\xi$ be a random variable in $\TT$ which is distributed according to a Borel probability measure $\lambda$, and consider the random counting measure (random point process)
\begin{equation}
\Xi_{N,\lambda} = \sum_{j=1}^N \sum_{k\in\ZZ} \delta_{N(\xi_j-\xi+k)} 
\end{equation}
on $\RR$. Here $\delta_x$ denotes the Dirac mass at $x\in\RR$. For example, for a given interval $I\subset\RR$ and integer $k$, we have
\begin{equation}
\PP( \Xi_{N,\lambda}(I) =k ) =  \lambda( \{ x\in\TT : \scrN_I(x,N) = k \} )
\end{equation}
with 
\begin{equation}
\scrN_I(x,N) =\#\{ j \leq N : \xi_j \in x+N^{-1} I+\ZZ \} .
\end{equation}
The scaling of the interval by $N^{-1}$ ensures that we expect a finite number of points for typical $x$. The following theorem describes the precise distribution.

\begin{theorem}\label{thm2}
Assume $D>0$ is square-free and $D\not\equiv 1\pmod 4$. Then there exists a random point process $\Xi$ depending only on $D$ so that, for every Borel probability measure $\lambda$ on $\TT$ that is absolutely continuous with respect to the Lebesgue measure, we have convergence 
$\Xi_{N,\lambda}  \to \Xi$ in distribution as $N\to\infty$. 

Specifically, for all $k_1,\ldots,k_r\in\ZZ_{\geq 0}$ and finite intervals $I_1,\ldots,I_r\subset\RR$, we have that
\begin{equation}
\lim_{N\to\infty} 
\lambda\big(\big\{ x\in\TT : \scrN_{I_i}(x,N) = k_i \;\forall i \big\}\big) 
= \PP\big( \Xi(I_i)= k_i \; \forall i \big)
\end{equation}
and the limit is a continuous function of the endpoints of $I_i$.
\end{theorem}

We furthermore establish the uniform upper bound $\scrN_I(x,N)\ll \log N$, where the applied constant is independent of $x\in\TT$ and $N$. This generalises Fouvry and Iwaniec's bound \cite{FouvryIwaniec1997} in the case $D=-1$.

We provide an explicit description of  the limit process  $\Xi$ in terms of tops of random geodesics in the hyperbolic plane, as well as an alternative description as entry times for a certain Poincar\'e section of the horocycle flow. A striking feature is that $\Xi$ is independent of the choice of $\lambda$. Theorem \ref{thm2} implies the convergence of various popular fine-scale statistics of the sequence $(\xi_j)_{j=1}^\infty$, including the gap distribution. Furthermore, we will prove the convergence of moments and higher correlation densities.

The above theorems in fact remain true (with different limit distributions) if we further restrict the roots to $m \equiv 0 \pmod n$, and $\mu \equiv \nu \pmod n$, for fixed $n > 0$ and $\nu \pmod n$ such that $\nu^2 \equiv D \pmod n$.

The results of this paper also hold for negative $D$. Here the proofs reduce to existing results \cite{MarklofVinogradov2018,RisagerRudnick2009} on the fine-scale distribution of the real parts of hyperbolic lattice points, rather than the tops of geodesics. The studies of real parts of hyperbolic lattice points are closely related to the distribution of angles in hyperbolic lattices \cite{BocaPasolPopaZaharescu2014,BocaPopaZaharescu2014,KelmerKontorovich2015,MarklofVinogradov2018,RisagerSodergren2017}.
We also remark that results on counting the tops of geodesics can be found in \cite{ParkkonenPaulin2012, ParkkonenPaulin2016} as well as applications of a similar flavor.
Their methods also apply in a general setting of negative curvature, where the notion of `top' is naturally replaced by that of a common perpendicular between a fixed horosphere and geodesic.

\subsection{Previous work on roots}
\label{sec:background}

In this section we review previous techniques used to study the sequence of roots $\frac{\mu}{m}$ and the results produced.
As previously mentioned, Hooley \cite{Hooley1963} proved the equidistribution of this sequence.
In fact, he produced a power-saving bound on the Weyl sums
\begin{equation}
  \label{eq:weylsum}
  \sum_{m\leq M} \sum_{\mu^2 \equiv D (m)} e \left( \frac{h\mu}{m} \right), \qquad e(x):=\ee^{2\pi\ii x},
\end{equation}
for non-zero integer $h$.
These results are obtained by transforming (\ref{eq:weylsum}) into a sum of Kloosterman sums, which is then estimated by an application of the Weil bound.

In \cite{Iwaniec1978}, Iwaniec follows a similar strategy to bound the Weyl sum (\ref{eq:weylsum}) with $D = -1$ but with various additional arithmetic constraints on $m$ and $\mu$, most notably the restriction to $m \equiv 0 \pmod n$.
Iwaniec obtains enough uniformity in $n$ so that a sieve can be applied, resulting in a proof that $l^2 +1$ is infinitely often a product of at most two primes.

In light of the appearance of a sum of Kloosterman sums in these works, \cite{Hooley1963} and \cite{Iwaniec1978}, one could in retrospect already predict the relevance of the spectral theory of $\mathrm{SL}(2)$ automorphic forms.
However, the first application to analysing the sequence of $\frac{\mu}{m}$, found in the work of Bykovski\u{\i}, \cite{Bykovskii1984}, proceeds by a more direct method and only applies to negative $D$.
Roughly speaking, Bykovski\u{\i}'s method exploits the classical connection between the roots $\mu \mod m$ and the $\mathrm{SL}(2, \mathbb{Z})$ orbits of Heegner points in the Poincar\'e upper half-plane $\mathbb{H}$. 
For example, the images of $\ii \in \mathbb{H}$ under $\gamma \in \mathrm{SL}(2,\mathbb{Z})$ have the form
\begin{equation}
  \label{eq:imaginaryorbit}
  \gamma \ii = \frac{\mu}{m} + \frac{\ii}{m},
\end{equation}
where $\mu$ satisfies $\mu^2 \equiv -1 \pmod {m}$ (see section \ref{sec:negative} for more details).
In this way, Bykovski\u{\i} turns the Weyl sum (\ref{eq:weylsum}) into a finite sum of $\mathrm{SL}(2, \mathbb{Z})$-Poincar\'e series, which, after a spectral expansion, can be controlled using estimates for sums of Fourier coefficients of automorphic forms. These spectral techniques were extended substantially by Good \cite{Good1983}, {and similar problems were considered in a very general context by Parkkonen and Paulin \cite{ParkkonenPaulin2012}, \cite{ParkkonenPaulin2016},} although no reference to roots of quadratic congruences are made. 

In \cite{DukeFriedlanderIwaniec1995}, Duke, Friedlander, and Iwaniec extend Bykovski\u{\i}'s approach to apply to the Weyl sums (\ref{eq:weylsum}) restricted to $m \equiv 0 \pmod n$, analogously to Iwaniec's earlier extension of Hooley's work.
Here Duke, Friedlander, and Iwaniec obtain enough uniformity in $h$ and $n$ to apply a sieve method and conclude the spectacular result that the sequence $\frac{\mu}{m}$ with $m$ restricted to be prime is still equidistributed modulo one.

The first use of automorphic forms to understand the sequence of $\frac{\mu}{m}$ in the case with $D > 0$ came in the work of Hejhal, \cite{Hejhal1986}.
Hejhal approximates the Weyl sum (\ref{eq:weylsum}) by an integral of a fixed automorphic form (related to the automorphic Green's function) over some closed geodesics in $\mathrm{SL}(2,\mathbb{Z}) \backslash \mathbb{H}$.
The results obtained in \cite{Hejhal1986}  only apply for select $D$ and are not as strong as those previously mentioned, perhaps due to a lack of a clear geometric interpretation of the roots.
Nevertheless, in \cite{Toth1997} Toth manages to extend the results of Duke, Friedlander, and Iwaniec, \cite{DukeFriedlanderIwaniec1995} to the positive $D$ setting.
Here the geometric question is avoided by returning to the original methods of Hooley and Iwaniec, which transform the Weyl sum directly into a sum of Kloosterman sums.

It is also worth mentioning the approach to this equidistribution question outlined by Sarnak in \cite{Sarnak1990}.
Sarnak suggests relating the Weyl sum (\ref{eq:weylsum}) to the Fourier coefficients of certain half integral weight automorphic forms by means of Salie sums.
As shown by Duke, Friedlander, and Iwaniec in \cite{DukeFriedlanderIwaniec2012}, this kind of method is capable of producing savings in the discriminant of the polynomial $F$ itself.
None of the previously mentioned methods, nor the new method we develop below, seem to be capable of this.

Other than equidistribution, the only statistics about the sequence $\frac{\mu}{m}$ that appear in the literature, as far as we are aware, are upper bounds for the number that can lie in very short intervals.
For $D = -1$, Fouvry and Iwaniec \cite{FouvryIwaniec1997} found that the number of $\frac{\mu}{m}$ with $m \leq M$ in an interval of length $\frac{1}{M}$ is $O(\log M)$. Our approach reproduces this bound for general $D$, cf.\ lemma \ref{lemma:FIbound}.
In \cite{FouvryIwaniec1997} and in Friedlander and Iwaniec's breakthrough \cite{FriedlanderIwaniec1998}, this bound was a key input into a sieve method for producing primes of the form $l^2 + p^2$ and $l^2 + k^4$, respectively.

\subsection{Outline of paper}
\label{sec:outline}

Let $\Gamma$ be a discrete subgroup of the group of isometries $G$ acting on the complex upper half-plane $\HH$, and let $\bm{c}_1,\ldots, \bm{c}_h$ be a finite collection of geodesics in $\HH$ that project to closed geodesics on the hyperbolic surface $\Gamma\backslash\HH$. In {\bf section \ref{sec:geodesic}} we study the distribution of the geodesic lines $g^{-1}\gamma \bm{c}_l$ where $\gamma$ runs through $\Gamma$ and $l=1,\ldots,h$. This becomes a geodesic random line process if $g$ is taken to be a random isometry with respect to a suitable probability measure. In {\bf section \ref{sec:convergence}} we prove limit theorems for sequences of random geodesics that are randomly translated, followed by the proof of convergence of moments in  {\bf section \ref{sec:moments}}. Key ingredients of the proofs are the equidistribution of long horocycles, a Siegel-style mean value formula and no-escape-of-mass estimates to control the moments. A convenient parametrisation of a geodesic is given by its ``top'', i.e., the point on the geodesics with largest imaginary part. In {\bf section \ref{sec:1dpp}} we project the tops of the random geodesics to the real line, which produces one-dimensional point processes. The convergence of these processes to a limit will eventually lead to the proof of theorem \ref{thm1} for the roots of quadratic congruences. 

The objective of {\bf section \ref{sec:discrete}} is to prepare the ground for geodesic line process conditioned so that at least one geodesic has a top in the imaginary axis. To this end, we introduce in {\bf section \ref{sec:surface}} a Poincar\'e section of the horocycle flow, and prove various equidistribution results for the intersection points of horocycles with the section, with an interesting variant in {\bf section \ref{sec:intersection}} that considers the intersection points reduced to a closed geodesic.

These equidistribution results are then applied in {\bf section \ref{sec:conditioned}} to establish limit theorems for the conditioned geodesic line processes. In analogy with section \ref{sec:geodesic} we prove convergence in distribution in {\bf section \ref{sec:convergence2}}, work out out a highly non-trivial Siegel-type volume formula in {\bf section \ref{section:intensity23}}, and prove convergence of moments in {\bf section \ref{sec:moments2}}. Projecting the tops to the real line as in section \ref{sec:1dpp} yields a point process $\Xi^0$ which is distributed according the Palm distribution of $\Xi$.\ The corresponding limit theorems are stated in {\bf section \ref{sec:palm}}. We give an alternative description of the processes $\Xi$ and $\Xi^0$ in terms of the entry and return times for the horocycle flow in {\bf section \ref{sec:entry}}.

The final {\bf section \ref{sec:roots}} of this paper expresses the roots $\mu$ of quadratic congruences in terms of tops of geodesics and hence makes the connection with the distribution results established thus far. This is detailed in {\bf section \ref{sec:roots2}} for the setting described in the introduction, where the relevant discrete subgroup $\Gamma$ is the modular group $\mathrm{SL}(2, \mathbb{Z})$. In {\bf section \ref{sec:congruence}} we extend the setting to allow additional restrictions $\mu \equiv \nu \pmod n$, which leads us to the congruence subgroups $\Gamma=\Gamma_0(n)$. In {\bf section \ref{sec:negative}} we indicate how our results can be modified in the case of negative $D$, leading to random hyperbolic lattices, rather than geodesic line processes.

{\bf Acknowledgement.}
We would like to thank Zonglin Li for his helpful comments, particularly on section \ref{sec:roots}, and the referees for their careful reading and suggestions.

\section{Geodesic line processes}
\label{sec:geodesic}

We recall that $G = \mathrm{SL}(2, \mathbb{R})$ acts on the complex upper half-plane $\mathbb{H}=\{ x + \ii y : y > 0\}$ by fractional linear transformations and that $\mathrm{PSL}(2, \mathbb{R})=G/\{\pm I\}$ is identified with the unit tangent bundle of $\mathbb{H}$ via
\begin{equation}
  \label{eq:GT1H}
  g =
  \begin{pmatrix}
    a & b \\
    c & d
  \end{pmatrix}
  \mapsto \left( \frac{a\ii + b}{c\ii + d}, \frac{\ii}{(c\ii + d)^2} \right),
\end{equation}
where we have written the unit tangent bundle of $\mathbb{H}$ as $(z, w)$ with $z\in \mathbb{H}$ and $w \in \mathbb{C}$, $|w| = \mathrm{Im}(z)$.
This identification is nicely expressed in terms of the Iwasawa decomposition of $G$: if
\begin{equation}
  \label{eq:Iwasawa}
  g =
  \begin{pmatrix}
    1 & x \\
    0 & 1
  \end{pmatrix}
  \begin{pmatrix}
    y^{\frac{1}{2}} & 0 \\
    0 & y^{-\frac{1}{2}}
  \end{pmatrix}
  \begin{pmatrix}
    \cos \frac{\theta}{2} & -\sin \frac{\theta}{2} \\
    \sin \frac{\theta}{2} & \cos \frac{\theta}{2}
  \end{pmatrix}
  ,
\end{equation}
then $\pm g$ is identified with $(x + \ii y, \ii\ee^{-\ii\theta} y)$.
We note that multiplying $g$ by $-I$ corresponds to changing $\theta$ in (\ref{eq:Iwasawa}) to $\theta + 2\pi$, and so does not change the point $(x + \ii y, \ii \ee^{-\ii\theta} y)$.
In what follows we work with $G$ directly rather than $G / \{\pm I\}$, and we only emphasise the difference when there is a chance of confusion.

Given $g\in G$, we set
\begin{equation}
  \label{eq:cdef}
  \bm{c}=\bm{c}(g) = \left\{ g
  \begin{pmatrix}
    t & 0 \\
    0 & t^{-1}
  \end{pmatrix}
  \ii : t > 0 \right\} \subset \mathbb{H},
\end{equation}
which defines, via the identification of $G/\{\pm I\}$ with the unit tangent bundle of $\mathbb{H}$, an oriented geodesic in $\mathbb{H}$.
Every geodesic has two limit points on the boundary $\partial\mathbb{H} = \mathbb{R} \cup \{\infty\}$. 
We say that $\bm{c}$ has positive orientation if both endpoints are in $\mathbb{R}$ and the endpoint corresponding to $t \to 0$ is less than the endpoint corresponding to $t \to \infty$. This means that $\bm{c}$ is a semicircle traced from left to right. (Of course we could also assign an orientation to geodesics with one endpoint at $\infty$, but it is convenient not to include them in this definition.)  For a positively oriented geodesic $\bm{c}$ we define $z_{\bm{c}} \in \mathbb{H}$, the top of the geodesic $\bm{c}$, to be the point in $\mathbb{H}$ on $\bm{c}$ with largest imaginary part.
That is, the point on $\bm{c}$ closest (with respect to the hyperbolic metric) to a horocyclic neighborhood of the boundary point $\infty \in \partial \mathbb{H}$.
We see that for every $z\in\mathbb{H}$ there is a unique positively oriented geodesic $\bm{c}$ such that $z=z_{\bm{c}}$. (Note that this would fail for geodesics with endpoint $\infty$). The aim is now study an infinite set of geodesics in $\mathbb{H}$ generated by a discrete subgroup $\Gamma$.

Let $\Gamma$ be a lattice in $G$ which is not co-compact, and we assume (without loss of generality) that $\Gamma$ is scaled so that $\Gamma \backslash \mathbb{H}$ has a cusp of width one at $\infty$, and also that $-I\in\Gamma$.
This means that 
\begin{equation}
  \label{eq:Gammainfty}
  \Gamma_\infty =
  \Gamma\cap  \left\{
  \pm \begin{pmatrix}
    1 & x \\
    0 & 1
  \end{pmatrix}
  : x \in \mathbb{R}\right\}
  = \left\{
  \pm \begin{pmatrix}
    1 & k \\
    0 & 1
  \end{pmatrix}
  : k \in \mathbb{Z}\right\}
\end{equation}
is a subgroup of $\Gamma$; it is the stabiliser in $\Gamma$ of $\infty \in \partial \mathbb{H}$.

In our applications to roots of quadratic congruences, we will work with the modular group $\Gamma = \mathrm{SL}(2,\mathbb{Z})$ and, more generally, the congruence subgroups 
\begin{equation}
  \label{eq:Gamma0ndef}
  \Gamma_0(n) = \left\{
  \begin{pmatrix}
    a & b \\
    c & d
  \end{pmatrix}
  \in \mathrm{SL}(2,\mathbb{Z}) : c \equiv 0 \pmod n \right\}.
\end{equation}

Consider the finite collection of positively oriented geodesics $\bm{c}_1=\bm{c}(g_1),\ldots, \bm{c}_h=\bm{c}(g_h)$ with $g_1, \dots, g_h \in G$ so that no two $\bm{c}_l$ are $\Gamma$-equivalent (i.e. $\Gamma \bm{c}_{l'}\neq \Gamma \bm{c}_l$ for $l'\neq l$) and that each $\bm{c}_l$ projects to a closed geodesic in $\Gamma \backslash \mathbb{H}$. 
This last assumption implies that the stabiliser $\Gamma_{\bm{c}_l}$ of $\bm{c}_l$ in $\Gamma$, 
\begin{equation}
  \label{eq:geodesicstabilizer}
  \Gamma_{\bm{c}_l} = \Gamma \cap \left\{ \pm g_l
    \begin{pmatrix}
      t & 0 \\
      0 & t^{-1}
    \end{pmatrix}
  g_l^{-1} : t > 0 \right\},
\end{equation}
is conjugate to the infinite (projectively) cyclic group 
\begin{equation}
  \label{eq:fundamentalunit}
  g_l^{-1} \Gamma_{\bm{c}_l} g_l =  \left\langle
    \pm \begin{pmatrix}
      \varepsilon_l & 0 \\
      0 & \varepsilon_l^{-1}
    \end{pmatrix}
  \right\rangle
\end{equation}
for suitable $\varepsilon_l > 1$.

Let us now turn to the distribution of the infinite collection of geodesics
\begin{equation}
 \bigcup_{l=1}^h  \bigcup_{\gamma\in\Gamma/\Gamma_{\bm{c}_l}}  \gamma \bm{c}_l  
\end{equation}
by studying their tops, 
\begin{equation}
\scrS = \bigcup_{l=1}^h  \bigcup_{\gamma\in\Gamma/\Gamma_{\bm{c}_l}} z_{\gamma \bm{c}_l}.
\end{equation}
Note that in general $g z_{\bm{c}_l} \neq z_{g \bm{c}_l}$. However we have $\gamma z_{\bm{c}_l} = z_{\gamma \bm{c}_l}$ for $\gamma\in\Gamma_\infty$. Thus for such $\gamma$ we have $\gamma\scrS = \scrS$, i.e. $\scrS$ is invariant under translations by 1, and we can therefore reduce the problem to $\scrS$ modulo one, with $\gamma\in\Gamma_\infty \backslash \Gamma / \Gamma_{\bm{c}_l}$.
Specifically, we are interested in the fine-scale statistics of the real parts of geodesic tops with imaginary part in the interval $[ \alpha y, \beta y)$,
\begin{equation}\label{def:Xab}
X_{\alpha,\beta}(y) = \biguplus_{l=1}^h X_{\alpha,\beta}^l(y) 
\end{equation}
with 
\begin{equation}\label{def:Xabl}
X_{\alpha,\beta}^l(y) = \Big\{ \Re(z_{\gamma \bm{c}_l}) \bmod 1 : \gamma\in \Gamma_\infty \backslash \Gamma / \Gamma_{\bm{c}_l},\; \alpha y \leq \Im (z_{\gamma \bm{c}_l}) < \beta y\Big\} \subset \TT,
\end{equation}
in the limit $y\to 0$. Here $X_{\alpha,\beta}(y)$, $X_{\alpha,\beta}^l(y)$ are defined as multisets, i.e., we list points with their multiplicity. To quantify their distribution, define the number of points in small intervals by
\begin{equation}
\label{eq:NIdef}
\scrN_{I,\alpha,\beta}(x,y) = \# \Big(X_{\alpha,\beta}(y) \cap (x + y I +\ZZ)\Big) ,
\end{equation}
where $I\subset \mathbb{R}$ is a fixed finite interval, and $1\leq \alpha < \beta \leq \infty$. With
\begin{equation}
  \label{eq:Bdef}
  B_{I, \alpha, \beta} = \{ u + \mathrm{i}v \in \mathbb{H} : u \in I, \alpha \leq v < \beta\},
\end{equation}
we have
\begin{equation}
\scrN_{I,\alpha,\beta}(x,y) = \sum_{l=1}^h \# \Big\{ \gamma\in\Gamma_\infty \backslash \Gamma / \Gamma_{\bm{c}_l} : 
z_{\gamma \bm{c}_l} \in x+ y B_{I, \alpha, \beta} +\ZZ \Big\} .
\end{equation}
In fact
\begin{equation}
\begin{split}
\scrN_{I,\alpha,\beta}(x,y) & = \sum_{l=1}^h \# \Big\{ \gamma\in \Gamma / \Gamma_{\bm{c}_l} : 
z_{\gamma \bm{c}_l} \in x+y B_{I, \alpha, \beta} \Big\} \\
& = \sum_{l=1}^h \# \Big\{ \gamma\in \Gamma / \Gamma_{\bm{c}_l} : 
z_{\gamma \bm{c}_l} \in n(x) a(y) B_{I, \alpha, \beta} \Big\} 
\end{split}
\end{equation}
with
\begin{equation}
  \label{eq:andef}
  n(x) =
  \begin{pmatrix}
    1 & x \\
    0 & 1
  \end{pmatrix}
  ,\quad a(y) =
  \begin{pmatrix}
    y^{\frac{1}{2}} & 0 \\
    0 & y^{-\frac{1}{2}}
  \end{pmatrix}
  .
\end{equation}
Define for $g\in G$ and bounded $B\subset\HH$ the counting function
\begin{equation}
  \label{eq:NBdef1}
  \scrN_B(g) = \sum_{l=1}^h \# \{ \gamma \in \Gamma / \Gamma_{\bm{c}_l} : z_{g^{-1}\gamma \bm{c}_l} \in B \}.
\end{equation}
Using the fact that $g z_{\bm{c}} = z_{g\bm{c}}$ for $g=n(x) a(y)$ we see that
\begin{equation}\label{eq:keyid}
 \scrN_{I, \alpha, \beta}(x,y) =  \scrN_B(n(x)a(y))  .  
\end{equation}
where $B=B_{I, \alpha, \beta}$ is given by (\ref{eq:Bdef}). We furthermore note that $\scrN_B(\gamma g)=\scrN_B(g)$ for all $\gamma\in\Gamma$.
This motivates the definition of the following random line processes in $\HH$.

Define 
\begin{equation}
\Theta_{y,\lambda} = \sum_{l=1}^h \sum_{\gamma \in \Gamma / \Gamma_{\bm{c}_l}} \delta_{z_{(n(\xi)a(y))^{-1}\gamma \bm{c}_l}}
\end{equation}
with the random variable $\xi$ distributed according to the Borel probability measure $\lambda$ on $\TT$. 

We also define
\begin{equation}
\Theta = \sum_{l=1}^h \sum_{\gamma \in \Gamma / \Gamma_{\bm{c}_l}} \delta_{z_{g^{-1}\gamma \bm{c}_l}}
\end{equation}
with the random element $g$ distributed with respect to Haar probability measure $\mu_\Gamma$ on $\Gamma\backslash G$ ($\mu_\Gamma$ is the Haar measure of $G$ normalised so that $\mu_\Gamma(\Gamma\backslash G)=1$).
Since $\Theta_{y,\lambda}$ and $\Theta$ describe the distribution of random geodesics in $\HH$, we refer to them as {\em geodesic line processes}. Of course, in the present parametrisation they can also be viewed as point processes describing the location of the geodesic tops in $\HH$.

\subsection{Convergence of geodesic line processes}
\label{sec:convergence}

We will now investigate the convergence of our random line processes, first in distribution, then (in the next section) the convergence of its moments.

\begin{theorem}\label{theorem:convergence}
For every Borel probability measure $\lambda$ on $\TT$ that is absolutely continuous with respect to the Lebesgue measure, we have convergence 
$\Theta_{y,\lambda}  \to \Theta$ in distribution as $y\to 0$. 

In particular, for all $k_1,\ldots,k_r\in\ZZ_{\geq 0}$, finite intervals $I_i$ and $1\leq\alpha_i<\beta_i\leq\infty$, we have that
\begin{equation}
\lim_{y\to 0} \lambda\big(\big\{ x\in\TT : \scrN_{I_i,\alpha_i,\beta_i}(x,y) = k_i \;\forall i \big\}\big) 
= \PP\big( \Theta(B_{I_i,\alpha_i,\beta_i})= k_i \; \forall i \big) 
\end{equation}
and the limit is a continuous function of $\alpha_i$, $\beta_i$ and the endpoints of $I_i$.
\end{theorem}

As an example of theorem \ref{theorem:convergence}, let us consider the special case of the void distribution, where $r=1$ and $k_1=0$. Here, for $1\leq\alpha<\beta\leq\infty$ we have that for any finite interval $I$
\begin{equation}
\lim_{y\to 0} \lambda\big(\big\{ x\in\TT : \scrN_{I,\alpha,\beta}(x,y) = 0 \big\}\big) 
= \PP\big( \Theta(B_{I,\alpha,\beta})= 0 \big)  ,
\end{equation}
with the formula
\begin{equation}
\PP\big( \Theta(B)= 0 \big) 
= \int_{\Gamma\backslash G} \bigg( \prod_{l=1}^h \prod_{\substack{\gamma \in \Gamma / \Gamma_{\bm{c}_l}}} \Big( 1- \chi_{B}( z_{g^{-1}\gamma \bm{c}_l} ) \Big)\bigg) \dd\mu_\Gamma(g),
\end{equation}
where $\chi_B$ denotes the characteristic function of the set $B = B_{I, \alpha, \beta}$.

Our proof of theorem \ref{theorem:convergence} 
follows similar lines as the arguments in section 4 of \cite{MarklofVinogradov2018}.
The main difference here is that the points $z_{\gamma \bm{c}_l}$ do not form a $\Gamma$-orbit in $\mathbb{H}$.
Both here and in \cite{MarklofVinogradov2018} the main ingredient for the proof is the equidistribution of long horocycles.
The precise statement we use here can be found in theorem 5.6 of \cite{MarklofStrombergsson2010}, which is based on ideas going back to \cite{Margulis2004} and \cite{EskinMcMullen1993}.


We first need to compute the average of the functions $\scrN_B(g)$, which gives the intensity measure $\EE\Theta$ of the process $\Theta$. The analogue of this calculation in \cite{MarklofVinogradov2018} is just a straightforward application of the unfolding technique, but here it is more involved due to the fact that the $z_{\gamma \bm{c}}$ are not an orbit of a point.
This obstacle is overcome by an appropriate change of variables from the Iwasawa coordinates (\ref{eq:Iwasawa}) which has a notably simple Jacobian as seen in the proof below.
Set 
\begin{equation}\label{def:eta}
\kappa_\Gamma = \frac{\ell}{2\pi\,\mathrm{vol}_{\mathbb{H}}(\Gamma \backslash \mathbb{H})} ,\qquad \ell = 2\sum_{l=1}^h \log \varepsilon_l ,
\end{equation}
where $2\log \varepsilon_l$ and $\ell$ represent the individual and total lengths of the geodesics $\bm{c}_1,\ldots,\bm{c}_h$ in $\Gamma\backslash\HH$.

\begin{proposition}
  \label{lemma:Faverage}
  For any Borel set $B\subset\HH$, 
  \begin{equation}
    \label{eq:Faverage}
    \EE\Theta(B) = \int_{\Gamma \backslash G} \scrN_{B}(g) \dd\mu_\Gamma(g) =  \kappa_\Gamma\, \mathrm{vol}_{\mathbb{H}}(B). 
  \end{equation}
\end{proposition}

\begin{proof}
  We have
  \begin{equation}
    \label{eq:FBdef1}
    \scrN_B(g) = \sum_{l=1}^h\sum_{\gamma \in \Gamma / \Gamma_{\bm{c}_l}} \chi_{S_B}(g^{-1}\gamma g_l),
  \end{equation}
  where $\chi_{S_B}$ is the characteristic function of the set 
  \begin{equation}
    \label{eq:SBdef}
    S_B = \{ g \in G : z_{g \bm{c}_0} \in B \}
  \end{equation}
  with
  \begin{equation}
    \label{eq:c0def}
    \bm{c}_0 = \left\{
      \begin{pmatrix}
        t & 0 \\
        0 & t^{-1}
      \end{pmatrix} \ii
      : t > 0 \right\} 
  \end{equation}
  defining an upward oriented geodesic. 
  Unfolding shows that
  \begin{equation}
    \label{eq:unfolded}
    \int_{\Gamma \backslash G} \scrN_B(g) \dd \mu_\Gamma(g) = \sum_{l=1}^h \int_{\Gamma_{\bm{c}_l} \backslash G} \chi_{S_B}(g^{-1} g_l)\dd \mu_\Gamma(g). 
  \end{equation}
  Changing variables $g \gets g_l g$, the region of integration changes from $\Gamma_{\bm{c}_l} \backslash G$ to
  \begin{equation}
    \label{eq:newintregion}
    (g_l^{-1} \Gamma_{\bm{c}_l} g_l) \backslash G = \left\langle
      \pm\begin{pmatrix}
        \varepsilon_l & 0 \\
        0 & \varepsilon_l^{-1}
      \end{pmatrix}
      \right\rangle \backslash G,
  \end{equation}
  and so the integral becomes
  \begin{equation}
    \label{eq:Faverage1}
    \sum_{l=1}^h \int_{\mathcal{F}_l} \chi_{S_B}(g^{-1}) \dd \mu_\Gamma(g),
  \end{equation}
  where $\mathcal{F}_l$ is a fundamental domain for the action of $\left\langle
      \pm\begin{pmatrix}
        \varepsilon_l & 0 \\
        0 & \varepsilon_l^{-1}
      \end{pmatrix}
      \right\rangle $ on $G$.

To compute this integral, we introduce a coordinate chart on $G$ in which both $S_B$ and a fundamental domain $\mathcal{F}_l$ are easy to describe. 
The region of $G$ where this chart applies is the open set of $g$ such that $g^{-1}\bm{c}_0$ is positively oriented. 
Clearly $S_B$ is a subset of this region, as we only define $z_{\bm{c}}$ for positively oriented geodesics, and in what follows we only consider the subset of $\mathcal{F}_l$ in this region.

The choice of coordinates we use is explicitly 
\begin{equation}
  \label{eq:coordinates}
  g^{-1} = n(u)a(v) k(\tfrac{\pi}{2}) a(s)^{-1},
\end{equation}
where
\begin{equation}
  \label{eq:kdef}
  k(\theta) =
  \begin{pmatrix}
    \cos \frac{\theta}{2} & - \sin \frac{\theta}{2} \\
    \sin \frac{\theta}{2} & \cos \frac{\theta}{2}
  \end{pmatrix}
\end{equation}
and $u, v, s \in \mathbb{R}$ with $v, s > 0$. 
Since $k(\tfrac{\pi}{2}) a(s)^{-1} \bm{c}_0$ is the positively oriented geodesic with top $i$, we see that in these coordinates $z_{g^{-1} \bm{c}_0} = u + \mathrm{i}v$.
Moreover, we observe that in this coordinate chart the action of $
\pm\begin{pmatrix}
  \varepsilon_l & 0 \\
  0 & \varepsilon_l^{-1}
\end{pmatrix}
$ is simply to multiply $s$ by $\varepsilon_l^2$,
and so
\begin{equation}
  \label{eq:fundamentaldomaindef}
  \mathcal{F}_l = \{ a(s) k(-\tfrac{\pi}{2}) a(v^{-1}) n(u)^{-1} : 1 \leq s < \varepsilon_l^2, v > 0, u \in \mathbb{R}\}
\end{equation}
is a valid region for use in the integration (\ref{eq:Faverage1}). 

We now express the Haar measure on $G$ with respect to the coordinate chart (\ref{eq:coordinates}). 
To this end, we write the matrix on the right in (\ref{eq:coordinates}) in Iwasawa coordinates. 
We have
\begin{equation}
  \label{eq:coordinates1}
  g^{-1} = 
  \begin{pmatrix}
    \frac{1}{\sqrt{2sv}}(u + v) & \frac{\sqrt{s}}{\sqrt{2v}}(u - v) \\
    \frac{1}{\sqrt{2sv}} & \frac{\sqrt{s}}{\sqrt{2v}}
  \end{pmatrix}
  ,
\end{equation}
and so we find that the Iwasawa coordinates (\ref{eq:Iwasawa}) of $g^{-1}$ are
\begin{flalign}
  \label{eq:Iwasawa1}
  x & = \left( \frac{1}{2sv}(u + v) + \frac{s}{2v}(u - v) \right)\left( \frac{1}{2sv} + \frac{s}{2sv} \right)^{-1} = u - \frac{s^2 - 1}{s^2 + 1} v \\ \nonumber 
  y & = \left(\frac{1}{2sv} + \frac{s}{2v}\right)^{-1} = \frac{2sv}{s^2 + 1} \\ \nonumber
  \theta & = 2 \arctan (s^{-1}). 
\end{flalign}
We note that the Jacobian matrix of this change of coordinates, 
\begin{equation}
  \label{eq:Jacobian}
  J = 
  \begin{pmatrix}
    \pdv{x}{u} & \pdv{x}{v} & \pdv{x}{s} \\
    \pdv{y}{u} & \pdv{y}{v} & \pdv{y}{s} \\
    \pdv{\theta}{u} & \pdv{\theta}{v} & \pdv{\theta}{s}
  \end{pmatrix}
  ,
\end{equation}
is upper triangular, and so we only need to compute the diagonal entries.
We have
\begin{equation}
  \label{eq:Jdiagonalentries}
  \pdv{x}{u} = 1, \quad \pdv{y}{v} = \frac{2s}{s^2 + 1},\quad \pdv{\theta}{s} = - \frac{2}{s^2 + 1},
\end{equation}
and so
\begin{equation}
  \label{eq:newHaar}
  \dd x \frac{\dd y}{y^2} \dd \theta = \dd u  \frac{\dd v}{v^2} \frac{\dd s}{s}. 
\end{equation}
Including the normalizing factor that makes $\mu_\Gamma$ a probability measure on $\Gamma \backslash G$, we have that the integral (\ref{eq:Faverage1}) is
\begin{equation}
  \label{eq:Faverage2}
  \frac{1}{2 \pi \mathrm{vol}_{\mathbb{H}}(\Gamma \backslash \mathbb{H})} \int_1^{\varepsilon_l^2} \underset{u + \mathrm{i}v \in B}{\int\int} \dd u \frac{\dd v}{v^2} \frac{\dd s}{s},
\end{equation}
and (\ref{eq:Faverage}) follows. 
\end{proof}

Alhtough not strictly necessary for our purposes, we note that we can extend proposition \ref{lemma:Faverage} to the following mean value formula.

\begin{proposition}
For every measurable $f:\HH \to \RR_{\geq 0}$ (or $f\in L^1(\HH,\mathrm{vol}_\HH)$),
\begin{equation}
\int_{\Gamma\backslash G} \bigg(\sum_{l=1}^h \sum_{\gamma \in \Gamma / \Gamma_{\bm{c}_l}} f(z_{g^{-1}\gamma \bm{c}_l})\,\bigg) \dd\mu_\Gamma(g)
= \kappa_\Gamma \int_\HH f(z) \dd\mathrm{vol}_\HH(z).
\end{equation}
\end{proposition}

\begin{proof}
  As $\mathrm{vol}_{\mathbb{H}}$ is a Borel measure, we may reduce to the case when $f$ is the characteristic function of a Borel set, which is exactly the content of proposition \ref{lemma:Faverage}.
\end{proof}

Following the technique in \cite{MarklofVinogradov2018}, we derive from proposition \ref{lemma:Faverage} the following estimates.

\begin{lemma}
  \label{lemma:boundaryestimate}
  We have
  \begin{equation}
    \label{eq:boundaryestimate}
    \mu_\Gamma\left( \{ g \in \Gamma \backslash G : \scrN_B(g) \geq 1 \} \right) \leq \kappa_\Gamma\, \mathrm{vol}_{\mathbb{H}}(B).
  \end{equation}
  In particular, if $\mathrm{vol}_{\mathbb{H}}(\partial B) = 0$, then
  \begin{equation}
    \label{eq:boundaryestimate1}
    \mu_\Gamma \left( \partial \{ g \in \Gamma \backslash G : \scrN_B(g) = k \} \right) = 0.
  \end{equation}
\end{lemma}

\begin{proof}
  Markov's inequality implies
  \begin{equation}
    \label{eq:markov}
    \mu_\Gamma \left( \{ g \in \Gamma \backslash G : \scrN_B(g) \geq 1 \} \right) \leq \int_{\Gamma \backslash G} \scrN_B(g) \dd \mu_\Gamma(g),
  \end{equation}
  and so proposition \ref{lemma:Faverage} implies (\ref{eq:boundaryestimate}).
  We now observe that $g' \in \partial \{ g \in \Gamma \backslash G: \scrN_B(g) = k \}$ implies $\scrN_{\partial B}(g') \geq 1$, and so (\ref{eq:boundaryestimate1}) follows from (\ref{eq:boundaryestimate}).
\end{proof}

\begin{proof}[Proof of theorem \ref{theorem:convergence}]
    We now have all the ingredients to apply the equidistribution of long horocycles, theorem 5.6 of \cite{MarklofStrombergsson2010}.
This theorem states in particular that, for $\lambda$ absolutely continuous with respect to Lebesgue measure, if $S \subset \Gamma \backslash G$ has boundary of $\mu_\Gamma$-measure $0$, then
\begin{equation}
  \label{eq:bigtheoremrecall}
  \lim_{y\to 0} \int_{\mathbb{T}} \chi_S( n(x)a(y) ) \dd \lambda(x) = \mu_\Gamma(S),
\end{equation}
where $\chi_S$ is the indicator function of $S$.
Now (\ref{eq:boundaryestimate1}) shows that the relevant set $S$ for our theorem \ref{theorem:convergence} does indeed have boundary of $\mu_\Gamma$-measure zero. This furthermore implies that the limit distribution is continuous as stated.
\end{proof}

\subsection{Moments}\label{sec:moments}

We now turn to the moments of the counting functions $\scrN_{I, \alpha, \beta}(x,y)$.
For $y > 0$, $\bm{s} = (s_1, \dots, s_r) \in \mathbb{C}^r$, $\mathcal{I} = I_1 \times \cdots \times I_r$, $I_j \subset \mathbb{R}$ a finite interval, and $\bm{\alpha} = (\alpha_1, \dots, \alpha_r)$, $\bm{\beta} = ( \beta_1, \dots, \beta_r)$ satisfying $1\leq \alpha_j < \beta_j \leq \infty$, 
define the probabilities (cf.~theorem \ref{theorem:convergence})
\begin{equation}
  \label{eq:Pdef}
  P_{\mathcal{I}, \bm{\alpha}, \bm{\beta}}^\lambda(\bm{k}, y) = \lambda( \{ x \in \mathbb{T} : \scrN_{I_j, \alpha_j, \beta_j}(x, y) = k_j \;\forall  j\}),
\end{equation}
\begin{equation}
P_{\mathcal{I}, \bm{\alpha}, \bm{\beta}}(\bm{k}) 
= \PP\big( \Theta(B_{I_i,\alpha_i,\beta_i})= k_i \; \forall i \big),
\end{equation}
where $\bm{k} = (k_1, \dots, k_r)$ with $k_j$ non-negative integers;
the moments
\begin{equation}
  \label{eq:Momentsdef}
  M_{\mathcal{I}, \bm{\beta}, \bm{\alpha}}^\lambda(\bm{\kappa}, y) = \int_{\mathbb{T}} \prod_{1\leq j \leq r} \scrN_{I_j, \alpha_j, \beta_j}(x, y)^{\kappa_j} \dd \lambda (x) = \underset{k_1, \dots, k_r \geq 0}{\sum} k_1^{\kappa_1} \cdots k_r^{\kappa_r} P_{\mathcal{I}, \bm{\alpha}, \bm{\beta}}^\lambda(\bm{k},y) ,
\end{equation}
and
\begin{equation}
  \label{eq:Mdef1}
  M_{\mathcal{I}, \bm{\alpha}, \bm{\beta}}(\bm{\kappa}) = \int_{\Gamma \backslash G} \prod_{1\leq j \leq r} \scrN_{I_j, \alpha_j, \beta_j}(g)^{\kappa_j} \dd \mu_\Gamma(g) = \underset{k_1, \dots, k_r \geq 0}{\sum} k_1^{\kappa_1} \cdots k_r^{\kappa_r} P_{\mathcal{I}, \bm{\alpha}, \bm{\beta}}(\bm{k}),; 
\end{equation}
where $\bm{\kappa} = (\kappa_1, \dots, \kappa_r) \in \mathbb{R}^r$ with $\kappa_j \geq 0$;
and the moment generating functions 
\begin{equation}
  \label{eq:Gdef}
   \begin{split}
  G_{\mathcal{I}, \bm{\alpha}, \bm{\beta}}^\lambda(\bm{s}, y) & = \int_{\mathbb{T}} \exp( \sum_{1\leq j \leq r} s_j \scrN_{I_j, \alpha_j, \beta_j}(x, y)) \dd \lambda(x) \\
  & = \underset{k_1, \dots, k_r \geq 0}{\sum} \exp( \sum_{1\leq j\leq r} k_j s_j ) P^\lambda_{\mathcal{I}, \bm{\alpha}, \bm{\beta}}(\bm{k}, y) ,
  \end{split}
\end{equation}
and
\begin{equation}
  \label{eq:Gdef1}
  \begin{split}
  G_{\mathcal{I}, \bm{\alpha}, \bm{\beta}}(\bm{s}) & = \int_{\Gamma \backslash G} \exp( \sum_{1\leq j\leq r} s_j \scrN_{I_j, \alpha_j, \beta_j}(g)) \dd \mu_\Gamma(g) \\
  & = \underset{k_1, \dots, k_r \geq 0}{\sum} \exp( \sum_{1\leq j \leq r} k_j s_j) P_{\mathcal{I}, \bm{\alpha}, \bm{\beta}}(\bm{k}).
  \end{split}
\end{equation}

\begin{theorem}
  \label{theorem:moments}
  Let $\lambda$ be a probability measure on $\mathbb{T}$ with bounded density (with respect to Lebesgue measure), and $\mathcal{I}$, $\bm{\alpha}$, $\bm{\beta}$ as above.
  Then there is a constant $c_0 > 0$ such that for $\bm{s} = (s_1, \dots, s_r) \in \mathbb{C}^r$ satisfying $ \sum_j \max\{ \mathrm{Re}(s_j), 0 \} < c_0 $ the function $G_{\mathcal{I}, \bm{\alpha}, \bm{\beta}}(\bm{s})$ is analytic and we have
  \begin{equation}
    \label{eq:momentconvergence}
    \lim_{y \to 0} G^\lambda_{\mathcal{I}, \bm{\alpha}, \bm{\beta}}(\bm{s}, y) = G_{\mathcal{I}, \bm{\alpha}, \bm{\beta}}(\bm{s}).
  \end{equation}
\end{theorem}

By a standard argument, theorem \ref{theorem:moments} implies the convergence as $y \to 0$ of the mixed moments $M^\lambda_{\mathcal{I}, \bm{\alpha}, \bm{\beta}}(\bm{\kappa}, y)$.

\begin{corollary}
  \label{corollary:moments}
  Let $\lambda$ be a probability measure on $\mathbb{T}$ with bounded density, and $\mathcal{I}$, $\bm{\alpha}$, $\bm{\beta}$ as above.
  Then for all $\bm{\kappa} = (\kappa_1, \dots, \kappa_r) \in \mathbb{R}^r$ with $\kappa_j \geq 0$, we have $M_{\mathcal{I}, \bm{\alpha}, \bm{\beta}}(\bm{\kappa})$ is finite and
  \begin{equation}
    \label{eq:momentconvergence1}
    \lim_{y \to \infty} M^\lambda_{\mathcal{I}, \bm{\alpha}, \bm{\beta}}(\bm{\kappa}, y) = M_{\mathcal{I}, \bm{\alpha}, \bm{\beta}}(\bm{\kappa}). 
  \end{equation}
\end{corollary}

We also record the following corollary, which is simply an application of corollary \ref{corollary:moments} to $r = 1$ and $\kappa = 1$ together with the calculation in proposition \ref{lemma:Faverage}.

For $0\leq a< b\leq 1$ and $y>0$, define the counting function for the distribution of the real parts of $z_{\gamma \bm{c}_l}$ modulo one,
\begin{equation}
N_{\alpha,\beta}(a,b,y) =  \# \Big( X_{\alpha,\beta}(y) \cap ( [a, b]+\ZZ) \Big) .
\end{equation}

\begin{corollary}[Equidistribution modulo one]
  \label{corollary:totalcount}
  For a finite interval $I \subset \mathbb{R}$, $1\leq \alpha < \beta \leq \infty$ and probability measure $\lambda$ with bounded density, we have
  \begin{equation}
    \label{eq:totalcount}
    \lim_{y\to 0} \int_{\mathbb{T}} \scrN_{I, \alpha, \beta} (x, y) \dd \lambda(x) = \kappa_\Gamma \bigg(\frac1\alpha- \frac1\beta\bigg) \mathrm{length}(I)
  \end{equation}
  and, for $0\leq a< b\leq 1$,
  \begin{equation}
    \label{eq:equidistribution}
    \lim_{y\to 0}  y N_{\alpha,\beta}(a,b,y)
    = \kappa_\Gamma \bigg(\frac1\alpha- \frac1\beta\bigg) (b-a) . 
  \end{equation}
\end{corollary}

\begin{proof}
  As mentioned above, (\ref{eq:totalcount}) follows from corollary \ref{corollary:moments} and proposition \ref{lemma:Faverage}.
  Equation (\ref{eq:equidistribution}) follows from the observation that given any $\epsilon>0$, we have for $0<y<\epsilon$ and $I=[-\tfrac12,\tfrac12]$
  \begin{equation}
\int_{a+\epsilon}^{b-\epsilon} \scrN_{I, \alpha, \beta} (x, y) \dd x 
\leq
y N_{\alpha,\beta}(a,b,y)
\leq
\int_{a-\epsilon}^{b+\epsilon} \scrN_{I, \alpha, \beta} (x, y) \dd x .
\end{equation}
Indeed, from (\ref{eq:NIdef}), the left side (resp. right side) counts each point in $\xi \in X_{\alpha, \beta}(y)$ with $\xi \in [a, b] + \mathbb{Z}$ with weight at most (resp. at least) $y \mathrm{length}(I)$.
In view of \eqref{eq:totalcount} (choose $\lambda$ to be the uniform measure on the intervals $[a + \epsilon, b - \epsilon]$ and $[a - \epsilon, b + \epsilon]$, respectively), the limits of the left and right hand side exist for all $\epsilon>0$. Taking $\epsilon\to 0$, we obtain the right hand side of \eqref{eq:equidistribution}. 
\end{proof}

Our proof of theorem \ref{theorem:moments} follows the strategy of the proof of theorem 8 in \cite{MarklofVinogradov2018}.
As in the proof of theorem \ref{theorem:convergence}, our setting is somewhat complicated by the fact that our points $z_{\gamma \bm{c}_l}$ do not form a $\Gamma$-orbit of a point in $\mathbb{H}$.
However there is a point in $w_l \in \mathbb{H}$ such that for any of our points $z_{\gamma \bm{c}_l}$, there is a $\gamma' \in \Gamma$ such that $z_{\gamma \bm{c}_l}$ is a bounded distance from $\gamma' w_l$ (and in fact $\gamma' \Gamma_{\bm{c}_l} = \gamma \Gamma_{\bm{c}_l}$).
This fact gives us sufficient control over the points $z_{\gamma \bm{c}_l}$ for the purposes of proving theorem \ref{theorem:moments}, as manifested in lemmas \ref{lemma:FIbound} and \ref{lemma:measurebound}.

\begin{lemma}
  \label{lemma:FIbound}
 For every $\delta>1$, there is a constant $C_\delta$ such that or any finite interval $I$, $0\leq\alpha<\beta\leq \delta\alpha<\infty$, we have that
 \begin{enumerate}[{\rm (i)}]
\item  for $x\in\TT$, $y>0$,
    \begin{equation}
    \label{eq:xybounds}
    \scrN_{I,\alpha,\beta}(x,y) < C_\delta  \bigg(1+\frac{\mathrm{length}(I)}{\alpha}\bigg);
  \end{equation}
\item for $g\in G$ and $B = B_{I, \alpha, \beta}$ as in (\ref{eq:Bdef}),
    \begin{equation}
    \label{eq:xybounds1}
    \scrN_B(g) < C_\delta  \bigg(1+\frac{\mathrm{length}(I)}{\alpha}\bigg) ;
  \end{equation}
\item for $x\in\TT$, $0<y\leq \e^{-1}$,
    \begin{equation}
    \label{eq:xybounds2}
    \scrN_{I,1,\infty}(x,y) < C_2 \big( 1+ \mathrm{length}(I) \big) \log y^{-1} ;
  \end{equation}
\item for $0\leq a<b\leq 1$, $0<y\leq 1$,
  \begin{equation}
    \label{eq:ybound}
    N_{1,\infty}(a,b,y)  <  4C_2\, y^{-1} .
   \end{equation}
\end{enumerate}
(The constant $C_\delta$ only depends on $\delta$, $\Gamma$ and the geodesics $\bm{c}_l$.)
 \end{lemma}

\begin{proof}
Consider the number of hyperbolic lattice points in $B_{I,\alpha,\beta}$,
\begin{equation}
\widetilde\scrN_{I,\alpha,\beta}(g)= \sum_{l=1}^h \#\{ \gamma\in \Gamma: g^{-1} \gamma g_l \ii\in B_{I,\alpha,\beta}  \}.
\end{equation}
The set $g^{-1} \Gamma g_l$ is uniformly discrete, and the minimum distance between any two points is independent of $g$, since $G$ acts by isometries. Thus there is a constant $C_\delta$ such that $\widetilde\scrN_{[0,1],1,\delta}(g)\leq C_\delta$ for all $g\in G$. Note that 
\begin{equation}
\widetilde\scrN_{[t,t+\alpha],\alpha,\beta}(g)\leq \widetilde\scrN_{[t,t+\alpha],\alpha,\delta\alpha}(g) =\widetilde\scrN_{[0,1],1,\delta}(g n(t) a(\alpha))  \leq C_\delta.
\end{equation}
This bound is evidently uniform in $t,\alpha,\beta,g$. A covering argument shows that 
\begin{equation}
\widetilde\scrN_{I,\alpha,\beta}(g) \leq C_\delta \bigg(1+\frac{\mathrm{length}(I)}{\alpha}\bigg).
\end{equation}
We now show that this bound implies \eqref{eq:xybounds1}. Since $z_{\gamma \bm{c}_l}$ depends only on the coset $\gamma \Gamma_{\bm{c}_l}$, we claim that we can choose $\gamma$ appropriately so that $\gamma g_l \ii$ is at most a bounded hyperbolic distance from $z_{\gamma \bm{c}_l}$.
  Indeed, $\gamma g_l \ii$ is on the geodesic $\gamma \bm{c}_l$, and $\gamma \Gamma_l\gamma^{-1}$ is generated by the hyperbolic motion along this geodesic moving points a distance of $2 \log \varepsilon_l$.
  Hence $\gamma$ can be chosen so that $\gamma g_l i$ is a distance at most $\log \varepsilon_l$ from $z_{\gamma\bm{c}_l}$. Therefore \eqref{eq:xybounds1} follows by adjusting $C_\delta$ with a constant depending on $\varepsilon_l$.
Next, \eqref{eq:xybounds} follows trivially from \eqref{eq:xybounds1} in view of \eqref{eq:keyid}.

Again using the fact that there is a $\gamma$ such that $\gamma g_l \ii$ is a bounded distance from $z_{\gamma \bm{c}_l}$, it follows that there is a constant $Y > 0$ depending only on $g_l$ and $\Gamma$ such that $\mathrm{Im}(z_{\gamma \bm{c}_l}) \leq Y$ for all $\gamma$ (recall that since $\Gamma \backslash \mathbb{H}$ has a cusp at $\infty$, the maximum height of any orbit is bounded). We then have for $y<Y$ (the bound for $y\geq Y$ is obvious)
\begin{equation}
\begin{split}
\scrN_{I,1,\infty}(x,y)  & = \sum_{0\leq j \leq \log_2(Y/y)} \scrN_{I,2^j,2^{j+1}}(x,y) \\
&\leq \sum_{0\leq j \leq \log_2(Y/y)} C_2  \big(1+\mathrm{length}(I)\big) ,
\end{split}
\end{equation}
using \eqref{eq:xybounds}. This establishes \eqref{eq:xybounds2}.

The final bound \eqref{eq:ybound} follows from a similar argument. For the upper bound we may take $[a,b]=[0,1]$. Then using periodicity in $x$, we have
\begin{equation}
    \label{eq:ybound12344}
    \begin{split}
     N_{1,\infty}(0,1,y) & = \sum_{j=0}^\infty N_{1,2}(0,1,2^j y) \\
    & = y^{-1} \sum_{j=0}^\infty 2^{-j} \int_0^1  \scrN_{[0,1],1,2}(x,2^j y) \dd x ,
    \end{split}
  \end{equation}
which in view of \eqref{eq:xybounds} is bounded by $4C_2\, y^{-1}$. This proves \eqref{eq:ybound}.
\end{proof}

\begin{lemma}
  \label{lemma:measurebound}
  Let $\lambda$ be a measure with bounded density, $\mathcal{I}$, $\bm{\alpha}$, $\bm{\beta}$ as above, and $\bm{k} = (k_1, \dots, k_r)$ with $k_j \geq 0$ integers.
  Then there exists $c_0 > 0$ such that
  \begin{equation}
    \label{eq:measurebound}
    \sup_{y>0} P^\lambda_{\mathcal{I}, \bm{\alpha}, \bm{\beta}}(\bm{k}, y) 
    \leq \sup_{y>0}  \lambda\big(\big\{ x\in\TT : \scrN_{I_i,\alpha_i,\beta_i}(x,y) \geq k_i \;\forall i \big\}\big) 
    \ll \exp( - c_0 \max_{1\leq j\leq r} k_j)  
  \end{equation}
  and
 \begin{equation}
    \label{eq:measurebound1}
    P_{\mathcal{I}, \bm{\alpha}, \bm{\beta}}(\bm{k}) 
    \leq \PP\big( \Theta(B_{I_i,\alpha_i,\beta_i})\geq  k_i \; \forall i \big)  \ll \exp( -c_0 \max_{1\leq j \leq r} k_j ). 
  \end{equation}
\end{lemma}

\begin{proof}
  The inequalities ``$\leq$'' are trivial. As to the upper bounds ``$\ll$'', it is sufficient to establish them for $[\alpha_i,\beta_i)=[1,\infty)$. Let $j$ be so that $k_j = \max_i k_i$. 
    We may assume that $k_j \geq 1$, since (\ref{eq:measurebound}) is otherwise evident.
  Now recalling that
  \begin{equation}
    \label{eq:FIrecall}
    \scrN_{I_j, 1, \infty}(x, y) = \sum_{l=1}^h \# \{ \gamma \in \Gamma / \Gamma_{\bm{c}_l} : \mathrm{Re}(z_{\gamma \bm{c}_l}) \in x + y I_j, \mathrm{Im}(z_{\gamma \bm{c}_l}) \geq y \}
  \end{equation}
  and that by lemma \ref{lemma:FIbound}
  \begin{equation}
    \label{eq:Ytruncate}
    \sum_{l=1}^h \# \{ \gamma \in \Gamma / \Gamma_{\bm{c}_l} : \mathrm{Re}(z_{\gamma \bm{c}_l}) \in x + y I_j, y \leq \mathrm{Im}(z_{\gamma \bm{c}_l}) \leq Yy \} \ll \log Y
  \end{equation}
  for $Y > 1$ (note that this is not the same $Y$ that appears in the proof of lemma \ref{lemma:FIbound}), we conclude that for $x$ such that $\scrN_{I_j, 1,\infty}(x, y) \geq k_j$,
  \begin{equation}
    \label{eq:Ytruncate1}
    \sum_{l=1}^h \# \{ \gamma \in \Gamma / \Gamma_{\bm{c}_l} : \mathrm{Re}(z_{\gamma \bm{c}_l}) \in x + y I_j, \mathrm{Im}(z_{\gamma \bm{c}_l}) \geq C y \exp( c_0 k_j) \} \geq 1
  \end{equation}
  for some positive $C, c_0$ independent of $x$ and $y$.
  Setting $Y = C \exp( c_0 k_j)$ we have that
  \begin{equation}
    \label{eq:onejbound1}
    \{ x \in \mathbb{T} : \scrN_{I_j, 1, \infty}(x, y) \geq k_j \} \subset \{ x\in \mathbb{T} : \scrN_{ I_j, Y, \infty}(x, y) \geq 1 \}. 
  \end{equation}
  By Markov's inequality, the $\lambda$ measure of the set on the right of (\ref{eq:onejbound1}) is at most
  \begin{equation} \begin{split}
    \label{eq:markovbound}
    & \int_{\mathbb{T}} \scrN_{I_j, Y, \infty} (x, y) \dd \lambda(x) \ll \int_{\mathbb{T}} \scrN_{I_j, Y, \infty} (x, y) \dd x \\
    & \ll y\; \mathrm{length}(I_j) \sum_{l=1}^h \# \{ \gamma \in \Gamma / \Gamma_{\bm{c}_j} : 0\leq \mathrm{Re}(z_{\gamma \bm{c}_l}) \leq 1, \mathrm{Im}(z_{\gamma \bm{c}_l}) \geq yY \},
  \end{split} \end{equation}
  where the first inequality follows from $\lambda$ having bounded density.
  It follows from lemma \ref{lemma:FIbound} that the quantity on the right of (\ref{eq:markovbound}) is
  \begin{equation}
    \label{eq:goodbound}
    \ll Y^{-1} \ll \exp( - c_0 k_j)
  \end{equation}
  with the implied constant independent of $y$ (but depending on $I_j$, the geodesics $\bm{c}_l$, etc.).
  This establishes (\ref{eq:measurebound}), and (\ref{eq:measurebound1}) follows immediately from (\ref{eq:measurebound}) and theorem \ref{theorem:convergence}. 
\end{proof}

\begin{proof}[Proof of theorem \ref{theorem:moments}]
  Having lemma \ref{lemma:measurebound}, the proof of theorem \ref{theorem:moments} is easy to finish.
  To verify that $G_{\mathcal{I}, \bm{\alpha}, \bm{\beta}}(\bm{s})$ is analytic in the set of $\bm{s} = (s_1, \dots, s_r) \in \mathbb{C}^r$ satisfying $\sum_j \max\{ \mathrm{Re}(s_j), 0\} < c_0$, we start by arranging
  \begin{equation}
    \label{eq:GsI}
    G_{\mathcal{I}, \bm{\alpha}, \bm{\beta}}(\bm{s}) = \sum_{k \geq 0} \underset{ \substack{ k_1, \dots, k_r \geq 0 \\ \max_j k_j = k}}{\sum} \exp( \sum_{1\leq j \leq r} k_j s_j) P_{\mathcal{I}, \bm{\alpha}, \bm{\beta}}(\bm{k}).
  \end{equation}
  Applying (\ref{eq:measurebound1}) we find that
  \begin{equation} \begin{split}
    \label{eq:tailbound}
    & \sum_{k \geq K} \underset{ \substack{ k_1, \dots, k_r \geq 0 \\ \max_j k_j = k}}{\sum} \bigg| \exp ( \sum_{1\leq j \leq r} k_j s_j ) P_{\mathcal{I}, \bm{\alpha}, \bm{\beta}}(\bm{k}) \bigg| \\
    & \ll \sum_{k \geq K} (k+1)^r \exp( -\left( c_0 - \sum_{1\leq j \leq r} \max\{ \mathrm{Re}(s_j), 0\}\right) k ).
  \end{split} \end{equation}
  This bound clearly goes to $0$ uniformly as $K \to \infty$ for $\bm{s}$ in compact subsets of the region $\sum_j \max\{ \mathrm{Re}(s_j), 0\} < c_0$, and so $G(\bm{s}, \mathcal{I})$ is analytic in this region.

  To verify that
  \begin{equation}
    \label{eq:Gylimit}
    \lim_{y\to 0} G_{\mathcal{I}, \bm{\alpha}, \bm{\beta}}(\bm{s}, y) = G_{\mathcal{I}, \bm{\alpha}, \bm{\beta}}(\bm{s}),
  \end{equation}
  we recall that
  \begin{equation}
    \label{eq:Gyrecall}
    G_{\mathcal{I}, \bm{\alpha}, \bm{\beta}}(\bm{s}, y) = \underset{ k_1, \dots, k_r \geq 0 }{\sum} \exp( \sum_{1\leq j \leq r} k_j s_j ) P^\lambda_{\mathcal{I}, \bm{\alpha}, \bm{\beta}}(\bm{k}, y).
  \end{equation}
  Proceeding as above, we arrange this sum according to $k = \max_j k_j$, and split the sum into $k < K$ and $k \geq K$.

  For the terms $k < K$, we apply theorem \ref{theorem:convergence} to obtain
  \begin{equation} \begin{split}
    \label{eq:smallkterms}
    &\lim_{y\to \infty} \sum_{0\leq k < K} \underset{ \substack{ k_1, \dots, k_r \geq 0 \\ \max_j k_j = k}}{\sum} \exp( \sum_{1\leq j \leq r} k_js_j) P^\lambda_{\mathcal{I}, \bm{\alpha}, \bm{\beta}}(\bm{k}, y) \\
    &= \sum_{0\leq k < K} \underset{ \substack{ k_1, \dots, k_r \geq 0 \\ \max_j k_j = k}}{\sum} \exp( \sum_{1\leq j \leq r} k_js_j) P_{\mathcal{I}, \bm{\alpha}, \bm{\beta}}(\bm{k}). 
  \end{split} \end{equation}

Using (\ref{eq:tailbound}), we see that this is equal to $G_{\mathcal{I}, \bm{\alpha}, \bm{\beta}}(\bm{s})$ up to an error that goes to $0$ as $K \to \infty$ since we are assuming $\sum_j \max\{ \mathrm{Re}(s_j), 0\} < c_0$. 
For the terms $k \geq K$, we apply (\ref{eq:measurebound}) to obtain
\begin{equation} \begin{split}
  \label{eq:largekterms}
  &\sum_{k \geq K} \underset{ \substack{ k_1, \dots, k_r \geq 0 \\ \max_j k_j = k}}{\sum} | \exp( \sum_{1\leq j \leq r} k_js_j)P^\lambda_{\mathcal{I}, \bm{\alpha}, \bm{\beta}}(\bm{k})|  \\
  &\ll \sum_{k \geq K} (k+1)^r  \exp( - \left( c_0 - \sum_{1\leq j \leq r} \max \{ \mathrm{Re}(s_j), 0\} \right) k).
\end{split} \end{equation}
As before, this goes to $0$ as $K \to \infty$ since $\sum_j \max\{ \mathrm{Re}(s_j), 0\} < c_0$, and so we obtain (\ref{eq:Gylimit}).
\end{proof}

\subsection{A one-dimensional point process}
\label{sec:1dpp}

We will now restate theorem \ref{theorem:convergence} (in the case $\alpha=1$, $\beta=\infty$) in a form closer to the setting of the introduction. The family of multisets 
\begin{equation}\label{def:X1oo}
X_{1,\infty}(y) = \biguplus_{l=1}^h \big\{ \Re(z_{\gamma \bm{c}_l}) \bmod 1 : \gamma\in\Gamma_\infty \backslash \Gamma / \Gamma_{\bm{c}_l},\; \Im (z_{\gamma \bm{c}}) \geq y\big\} \subset \TT,
\end{equation}
forms a nested subsequence as $y\to 0$. We list its elements (with multiplicity) as the sequence $(\xi_j)_{j=1}^\infty$, where we fix any ordering such that if $\xi_i\in X_{1,\infty}(y')$, $\xi_j\in X_{1,\infty}(y)$ with $i<j$, then $y'\geq y$. So in particular, the elements in $X_{1,\infty}(y)$ are listed as $\xi_1,\ldots,\xi_N$ with $N=N(y)=N_{1,\infty}(0,1,y)$. As in the introduction, we are interested in the number $\scrN_{I}(x,N)=\scrN_{I_i,1,\infty}(x,N)$ of points in the interval $x+N^{-1} I+\ZZ$. We define the corresponding random point process
\begin{equation}
\Xi_{N,\lambda} = \sum_{j=1}^N \sum_{k\in\ZZ} \delta_{N(\xi_j-\xi+k)} 
\end{equation}
on $\RR$, with $\xi\in\TT$ distributed according to $\lambda$.
Define the one-dimensional point process process $\Xi$ by
\begin{equation}
\Xi(I) =\Theta(B_{\kappa_\Gamma^{-1} I,1,\infty})
\end{equation}
with $\kappa_\Gamma$ as in \eqref{def:eta}.
With this scaling,
  \begin{equation}
    \label{eq:Faverage_eta}
    \EE\Xi(I) =  \EE\Theta(B_{\kappa_\Gamma^{-1} I,1,\infty}) = \kappa_\Gamma\, \mathrm{vol}_{\mathbb{H}}(B_{\kappa_\Gamma^{-1} I,1,\infty})
    = \mathrm{length}(I), 
  \end{equation}
and hence the intensity measure of $\Xi$ is the normalised Lebesgue measure on $\RR$.
Furthermore, since the distribution of $\Theta$ is invariant under translations $z\mapsto z+t$ of $\HH$, the process $\Xi$ is a (translation-) stationary point process in $\RR$.

We have the following limit theorem.

\begin{theorem}\label{theorem:convergencealt}
For every Borel probability measure $\lambda$ on $\TT$ that is absolutely continuous with respect to the Lebesgue measure, we have convergence 
$\Xi_{N,\lambda}  \to \Xi$ in distribution as $N\to \infty$. 

Specifically, for all $k_1,\ldots,k_r\in\ZZ_{\geq 0}$, finite intervals $I_i$, we have that
\begin{equation}
\lim_{N\to \infty} \lambda\big(\big\{ x\in\TT : \scrN_{I_i}(x,N) = k_i \;\forall i \big\}\big) 
= \PP\big( \Xi(I_i)= k_i \; \forall i \big) .
\end{equation}
\end{theorem}

\begin{proof}
This follows directly from theorem \ref{theorem:convergence}. The only difference is the scaling by $N$ rather than $y^{-1}$. The equivalence of the convergence in the two scalings follows from the asymptotics $yN(y)\to \kappa_\Gamma$ from \eqref{eq:equidistribution} and the continuity of the limit distribution. 
\end{proof}

\section{Equidistribution along discrete orbits}
\label{sec:discrete}

Theorems \ref{theorem:convergence} and \ref{theorem:moments} already give the existence of limits for many interesting fine-scale distributions of the the real parts of $z_{\gamma \bm{c}_l}$ mod 1. There are standard combinatorial arguments that allow the extension of the results to other statistics, such as gap distribution and pair correlation. See for instance \cite{MarklofVinogradov2018}, where the pair correlation is computed from second mixed moments. Here we will use a geometric approach, the surface of section method, from which the gap distribution and pair correlation follow more directly.
This is similar to the approach in \cite{AthreyaCheung2014} for the statistics of Farey sequences, but with a different section, see section \ref{sec:entry} for details.
We refer the reader to \cite{Marklof2017} for more background on the connection between statistics of entry and return times for a given Poincar\'e section on one hand, and point processes and their Palm distributions on the other.

\subsection{A Poincar\'e section for the horocycle flow}
\label{sec:surface}

For $1 \leq \alpha < \beta \leq \infty$, define the two-dimensional section
\begin{equation}
  \label{eq:Sdef}
  S_{\alpha, \beta} = \{ a(s)k(-\tfrac{\pi}{2}) a(v^{-1})   : s > 0, \alpha \leq v < \beta \} \subset  G.
\end{equation}
We will show below that for $1\leq \alpha<\beta\leq\infty$ the sets
\begin{equation}
\widetilde S_{\alpha,\beta}^l =\Gamma \backslash \Gamma g_l S_{\alpha,\beta} = \Gamma \backslash \Gamma g_l \{ a(s)k(-\tfrac{\pi}{2}) a(v^{-1})   : s\in[1,\varepsilon_l^2), \alpha \leq v < \beta \}
\end{equation}
and
\begin{equation}
\widetilde S_{\alpha,\beta} = \bigcup_{l=1}^h \widetilde S_{\alpha,\beta}^l 
\end{equation}
are measurable subsets of $\Gamma \backslash G$. 
Define the finite Borel measures $\nu_{\alpha,\beta}$, $\nu_{\alpha,\beta}^l$ on $\Gamma \backslash G$ by
\begin{equation}
\int_{\Gamma \backslash G} f(g) \,\dd \nu_{\alpha, \beta}^l(g) =  \frac{1}{2\pi \mathrm{vol}_{\mathbb{H}}(\Gamma \backslash \mathbb{H})} 
\int_\alpha^\beta \int_1^{\varepsilon_l^2} f(\Gamma g_la(s)k(-\tfrac{\pi}{2})a(v^{-1})) \;\frac{\dd s}{s}\; \frac{\dd v}{v^2} ,
\end{equation}
and furthermore
\begin{equation}
\nu_{\alpha, \beta} =   \sum_{l=1}^h  \nu_{\alpha, \beta}^l,
\end{equation}
so that the support of $\nu_{\alpha,\beta}$ is $\widetilde S_{\alpha,\beta}$, and that of $\nu_{\alpha,\beta}^l$ is $\widetilde S_{\alpha,\beta}^l$.

Recall the definition of the multiset $X_{\alpha,\beta}(y)$ in \eqref{def:Xab}.

\begin{theorem}
\label{theorem:surfaceequidistribution0}
  For $1 \leq \alpha < \beta \leq \infty$ and $f : \mathbb{T} \times \Gamma \backslash G \to  \mathbb{C}$ bounded continuous, we have
  \begin{equation}
    \label{eq:surfaceequidistribution}
    \lim_{y\to 0} y 
    \sum_{\xi\in X_{\alpha,\beta}(y)} f(\xi, \Gamma n(\xi)a(y)) = \int_{\mathbb{T}} \int_{\Gamma \backslash G} f(x,g) \dd \nu_{\alpha, \beta}(g) \dd x .
  \end{equation}
\end{theorem}

The plan for the proof of this statement is to show that $\widetilde{S}_{\alpha, \beta}$ is a Poincar\'e section for the horocycle flow $\Gamma g\mapsto \Gamma g n(t)$ on $\Gamma \backslash G$ such that the return times for the periodic orbit $\{ \Gamma a(y) n(t): t\in\TT\}$ have the form $ y^{-1} \mathrm{Re}(z_{\gamma \bm{c}_l})$. We then introduce an $\epsilon$-thickening of $\widetilde{S}_{\alpha, \beta}$ in the direction of the horocycle flow, and theorem \ref{theorem:surfaceequidistribution0} then follows by applying the equidistribution of long horocycles in $\Gamma \backslash G$. 

Let us first investigate the structure of $\widetilde S_{\alpha,\beta}^l$.
Recall $\{\Gamma g_l a(s): s > 0\}$ is a closed geodesic in $\Gamma\backslash G$ for each $l$, and in particular the curves are not self-intersecting and are disjoint for different values of $l$.
Therefore $\{\Gamma g_l  a(s)   : s > 0\} k(-\tfrac{\pi}{2}) a(v^{-1})$ are also closed curves in $\Gamma\backslash G$, which are not self-intersecting and are disjoint for different values of $l$. In fact, they do not intersect even after a slight thickening, as the following lemma shows.

\begin{lemma}
  \label{lemma:Sdisjoint}
  For $1\leq \alpha < \beta < \infty$ and $\epsilon > 0$ define
  \begin{equation}
    \label{eq:Slepsilondef}
    S^{\epsilon}_{\alpha, \beta} = \{ a(s) k(-\tfrac{\pi}{2}) a(v^{-1}) n(u) : s> 0, \alpha \leq v < \beta, |u| <  \epsilon \} \subset  G.
  \end{equation}
  Then there is $\delta>1$ sufficiently small (depending only on $\Gamma$ and the $g_l$) and $\epsilon>0$ sufficiently small (depending only on $\delta$, $\Gamma$ and the $g_l$), such that for $\beta<\delta\alpha$  and  $\gamma \in \Gamma$ and $1\leq l_1, l_2 \leq h$, we have that
  \begin{equation}
    \label{eq:Sdisjoint}
    \gamma g_{l_1} S^{\epsilon}_{\alpha, \beta} \cap g_{l_2} S^{\epsilon}_{\alpha, \beta} = \emptyset 
  \end{equation}
  unless $l_1 = l_2$ and $\gamma \in \Gamma_{\bm{c}_{l_1}}$. 
\end{lemma}

\begin{proof}
  Suppose that $\gamma \in \Gamma $ is such that
  \begin{equation}
    \label{eq:Sintersect}
    \gamma g_{l_1} a(s_1) k(-\tfrac{\pi}{2}) a(v_1)^{-1} n(u_1) = g_{l_2} a(s_2) k(-\tfrac{\pi}{2}) a(v_2)^{-1} n(u_2),
  \end{equation}
 for some $s_1,s_2>0$, $\alpha \leq v_1, v_2 < \beta$, and $|u_1|, |u_2| < \epsilon$. Replacing $\gamma$ with $\gamma_2 \gamma \gamma_1$ for suitable $\gamma_i\in\Gamma_{\bm{c}_{i}}$,
 we may assume furthermore that $1\leq s_1  < \varepsilon_{l_1}^2$, $1\leq s_2 < \varepsilon_{l_2}^2$. 
  This implies that
  \begin{equation}
    \label{eq:gammacondition}
    g_{l_2}^{-1} \gamma g_{l_1} =
    \begin{pmatrix}
      s^{\frac{1}{2}} & 0 \\
      0 & s^{-\frac{1}{2}}
    \end{pmatrix}
    + O( \beta \epsilon + ( \alpha^{-\frac{1}{2}}\beta^{\frac{1}{2}} - \alpha^{\frac{1}{2}} \beta^{-\frac{1}{2}}))
  \end{equation}
  where $s = s_2/s_1$, and so $\varepsilon_{l_1}^{-2} < s < \varepsilon_{l_2}^2$.
  The condition that (\ref{eq:gammacondition}) holds for some $\varepsilon_{l_1}^{-2} < s < \varepsilon_{l_2}^2$  constrains $\gamma$ to a fixed compact subset of $G$.
  The discreteness of $\Gamma$ then implies that for $\beta/\alpha<\delta$ sufficiently close to $1$ and $\beta \epsilon$ sufficiently small, if (\ref{eq:gammacondition}), then in fact
  \begin{equation}
    \label{eq:gammacondition1}
    g_{l_2}^{-1} \gamma g_{l_1} =
    \begin{pmatrix}
      s^{\frac{1}{2}} & 0 \\
      0 & s^{-\frac{1}{2}}
    \end{pmatrix}
  \end{equation}
  for some $\varepsilon_{l_1}^{-2} \leq s \leq \varepsilon_{l_2}^2$. 
  Since the geodesics $\bm{c}_{l}$ are not $\Gamma$-equivalent, \eqref{eq:gammacondition1} implies $l_1 = l_2$, and so then $\gamma \in \Gamma_{\bm{c}_{l_1}}$.  
\end{proof}

The lemma implies that $\widetilde S_{\alpha,\beta}^l$ are measurable subsets and $\nu_{\alpha,\beta}$ well defined as Borel measures. This observation will be key in the application of the surface-of-section method below.

\begin{lemma}
  \label{lemma:identify}
 Let $1\leq \alpha < \beta < \delta\alpha< \infty$ with $\delta$ as in lemma \ref{lemma:Sdisjoint}. Then the map
 \begin{equation}\label{eq:identify}
X_{\alpha,\beta}(y)\to \widetilde S_{\alpha,\beta}, \qquad \xi\mapsto \Gamma n(\xi) a(y)
\end{equation}
is injective, and we have
\begin{equation}
X_{\alpha,\beta}(y) = \{ \xi\in\TT : \Gamma n(\xi) a(y) \in \widetilde S_{\alpha,\beta} \}.
\end{equation}
\end{lemma}

\begin{proof}
Recall that $z_{\gamma\bm{c}_l}$ is the point on $\gamma\bm{c}_l$ where the tangent points directly to the right. The Iwasawa decomposition \eqref{eq:Iwasawa} hence tells us that
\begin{equation}
  \label{eq:horizontal}
  n(\mathrm{Re}(z_{\gamma\bm{c}_l})) a(\mathrm{Im}(z_{\gamma\bm{c}_l})) k( \tfrac{\pi}{2}) = \gamma g_l a(s)
\end{equation}
for a unique $s > 0$.  
Therefore if $\alpha y \leq \mathrm{Im}(z_{\gamma \bm{c}_l}) < \beta y $, then
\begin{equation}
  \label{eq:Stops}
 n( \mathrm{Re}(z_{\gamma \bm{c}_l})) a(y) = \gamma g_l a(s) k(-\tfrac{\pi}{2}) a(v^{-1}) \in \gamma g_l S_{\alpha, \beta}
\end{equation}
with $v=y^{-1}  \mathrm{Im}(z_{\gamma \bm{c}_l})$,
and so 
\begin{equation}
  \label{eq:Stops222}
 \Gamma n( \mathrm{Re}(z_{\gamma \bm{c}_l})) a(y) \in \widetilde S_{\alpha,\beta}^l = \Gamma \backslash\Gamma g_l S_{\alpha, \beta} .
\end{equation}
This shows that the map \eqref{eq:identify} is well defined. In view of lemma \ref{lemma:Sdisjoint}, the value of $v\in[\alpha,\beta)$ is in fact unique. 
Together with the uniqueness of $s\in[1,\varepsilon_l^2)$ this implies the injectivity of \eqref{eq:identify}.
\end{proof}

Note that the previous results require $\beta<\infty$. The following result shows that restricting to finite intervals $[\alpha,\beta)$ only produces a small error.

\begin{lemma}
  \label{lemma:uppbd}
 For $\alpha>0$ and $f : \mathbb{T} \times \Gamma \backslash G \to  \mathbb{C}$ bounded, we have
  \begin{equation}
    \label{eq:surfaceequidistributionuppbd}
    \limsup_{y\to 0}  \bigg| y \sum_{\xi\in X_{\alpha,\infty}(y)} f(\xi, \Gamma n(\xi)a(y)) \bigg| \leq \kappa_\Gamma\, \alpha^{-1} \sup|f|  .
  \end{equation}
\end{lemma}

\begin{proof}
Apply corollary \ref{corollary:totalcount}, eq.~\eqref{eq:totalcount}.
\end{proof}

\begin{proof}[Proof of theorem \ref{theorem:surfaceequidistribution0}]
Since the sum over $l$ is finite, we may assume without loss of generality that we only have one term, i.e., $h=1$.
We write $[\alpha,\beta)$ as a disjoint union of intervals $[\alpha_j,\alpha_{j+1})$, $j=1,\ldots, J$, with $\alpha_1=\alpha$ and $\alpha_{J+1}=\beta$. If $\beta=\infty$, we can choose $\alpha_J$ sufficiently large to make the contribution of the interval $[\alpha_J,\infty)$ to both sides of \eqref{eq:surfaceequidistribution} as small as we wish; recall lemma \ref{lemma:uppbd}. We may therefore assume without loss of generality that $\beta<\infty$. Assume furthermore $\alpha_{j+1}<\delta \alpha_j$ with $\delta>1$ as in lemma \ref{lemma:Sdisjoint}. We will now calculate the contribution to \eqref{eq:surfaceequidistribution} of each finite interval $[\alpha_j,\alpha_{j+1})$. For simplicity of notation, let us rename $[\alpha_j,\alpha_{j+1})$ as $[\alpha,\beta)$. We have thus reduced the proof of the lemma to the statement for $1\leq\alpha< \beta<\delta\alpha<\infty$ with $\delta>1$ as in lemma \ref{lemma:Sdisjoint}.

What now remains to be shown is that for $f : \mathbb{T} \times \widetilde S_{\alpha,\beta}^l \to  \mathbb{C}$ bounded continuous, we have
  \begin{equation}
    \label{eq:surfaceequidistribution1234}
    \lim_{y\to 0} y 
    \sum_{\xi\in X_{\alpha,\beta}^l(y)} f(\xi, \Gamma n(\xi)a(y)) 
    = \int_{\mathbb{T}} \int_{\Gamma \backslash G} f(x,g) \dd \nu_{\alpha, \beta}^l(g) \dd x .
\end{equation}
To this end, we define for given $\epsilon>0$ and $f$ bounded continuous, the function $F_\epsilon: \mathbb{T} \times \Gamma \backslash G \to  \mathbb{C}$ by
\begin{equation}
F_\epsilon(x,\Gamma g) = 
\begin{cases}
f(x,\Gamma g_la(s)k(-\tfrac{\pi}{2})a(v^{-1})) & \text{if $g=g_la(s)k(-\tfrac{\pi}{2})a(v^{-1})n(u)\in \Gamma g_l S_{\alpha,\beta}^\epsilon$}\\
0 & \text{if $g\notin \Gamma g_l S_{\alpha,\beta}^\epsilon$.}
\end{cases}
\end{equation}
By lemma \ref{lemma:Sdisjoint} and the assumptions on $\alpha,\beta$, the function $F_\epsilon$ is well defined, is bounded and has compact support; it is continuous except at points with $u=\pm\epsilon$, which form a set of Haar measure zero, recall the parametrisation \eqref{eq:newHaar}.   
The key observation is now that in view of \eqref{eq:horizontal}
\begin{equation}\label{eq:rhs111}
\frac{1}{2\epsilon}\int_{\mathbb{T}} F_\epsilon(x, \Gamma n(x)a(y)) \dd x = y 
    \sum_{\xi\in X_{\alpha,\beta}^l(y)} f(\xi, \Gamma n(\xi)a(y)) .
\end{equation}
The equidistribution of closed horocycles (as stated in theorem 5.6 of \cite{MarklofStrombergsson2010}) implies that
$$
\lim_{y\to 0} \frac{1}{2\epsilon}\int_{\mathbb{T}} F_\epsilon(x, \Gamma n(x)a(y)) \dd x
= \frac{1}{2\epsilon} \int_{\TT\times\Gamma\backslash G} F_\epsilon(x, g) \,\dd x\, \dd\mu_\Gamma(g).
$$
Using the formula $\eqref{eq:newHaar}$ for $\mu_\Gamma$ with the correct normalisation factor, we work out that
\begin{equation} \begin{split}\label{eq:rhs111limit}
& \frac{1}{2\epsilon} \int_{\TT\times\Gamma\backslash G} F_\epsilon(x, g) \,\dd x\, \dd\mu_\Gamma(g) \\
& = \frac{1}{2 \pi \mathrm{vol}_{\mathbb{H}}(\Gamma \backslash \mathbb{H})}  \int_{\mathbb{T}}  
\int_\alpha^\beta \int_1^{\varepsilon_l^2} f(x,\Gamma g_la(s)k(-\tfrac{\pi}{2})a(v^{-1})) \;\frac{\dd s}{s}\; \frac{\dd v}{v^2} \; \dd x.
\end{split} \end{equation}
In conclusion, the right hand side of \eqref{eq:rhs111} converges to the right hand side of \eqref{eq:rhs111limit}, as required.
\end{proof}

\subsection{Equidistribution of intersection points on a closed geodesic}
\label{sec:intersection}

The following corollary of the discussion in the previous section may be of independent interest. It concerns the joint equidistribution of the points $\mathrm{Re}(z_{\gamma \bm{c}_l}) +\ZZ\in \mathbb{T}$ (with respect to Lebesgue measure) and the points $\Gamma z_{\gamma \bm{c}_l}$ on the closed geodesic $\Gamma \backslash \Gamma \bm{c}_l$ (with respect to the arc length measure). 
To be precise, for $\gamma \in \Gamma_\infty \backslash \Gamma / \Gamma_{\bm{c}_l}$, we define $t_l(\gamma)$ as the unique $t\in[0,2\log\varepsilon_l)$ such that
\begin{equation}
  \label{eq:slgammadef}
  \Gamma_{\bm{c}_l} \gamma^{-1}n(\mathrm{Re}(z_{\gamma \bm{c}_l})) a( \mathrm{Im}(z_{\gamma \bm{c}_l})) = \Gamma_{\bm{c}_l} g_l a(\e^t)k(-\tfrac{\pi}{2})
\end{equation}
holds; recall \eqref{eq:Stops} and $2\log\varepsilon_l$ is the total length of $\bm{c}_l$.
We remark that in the context of the roots $\mu^2 \equiv D \pmod m$, the natural harmonics on the closed geodesics are Hecke characters for $\mathbb{Q}(\sqrt{D})$. 

\begin{proposition}
  \label{corollary:curveequidistribution}
  For $1\leq \alpha < \beta \leq \infty$ and  $f : \TT \times [0, 2\log \varepsilon_l] \to \mathbb{C}$ bounded continuous, we have
  \begin{equation}
    \label{eq:curveequidistribution}
    \lim_{y\to \infty} y \sum_{\substack{\gamma \in \Gamma_{\infty} \backslash \Gamma / \Gamma_{\bm{c}_l} \\ \alpha y \leq \mathrm{Im}(z_{\gamma \bm{c}_l}) < \beta y}} f(\mathrm{Re}(z_{\gamma \bm{c}_l}), t_l(\gamma))
    = \frac{\alpha^{-1}-\beta^{-1}}{2 \pi \mathrm{vol}_{\mathbb{H}}(\Gamma \backslash \mathbb{H})}  
    \int_0^{2\log\varepsilon_l} \int_{\mathbb{T}}  f(x, t) \,\dd x\, \dd t.
  \end{equation}
\end{proposition}

\begin{proof}
  As above, we may assume without loss of generality that $1\leq\alpha< \beta<\delta\alpha<\infty$ with $\delta>1$ as in lemma \ref{lemma:Sdisjoint}.
  We define $\tilde{f}: \mathbb{T} \times \tilde{S}^l_{\alpha, \beta} \to \mathbb{C}$ by
  \begin{equation}
    \label{eq:tildefdef}
    \tilde{f}(x, \Gamma g_l a(s) k( - \frac{\pi}{2}) a(v^{-1})) = \chi_{[\alpha, \beta)}(v) f(x, \log s).
  \end{equation}
  This is well-defined by lemma \ref{lemma:Sdisjoint}, and we observe from (\ref{eq:slgammadef}) that
  \begin{equation}
    \label{eq:ftilderewrite}
    \lim_{y\to \infty} y \sum_{\substack{\gamma \in \Gamma_{\infty} \backslash \Gamma / \Gamma_{\bm{c}_l} \\ \alpha y \leq \mathrm{Im}(z_{\gamma \bm{c}_l}) < \beta y}} f(\mathrm{Re}(z_{\gamma \bm{c}_l}), t_l(\gamma)) = \lim_{y \to \infty} y \sum_{\substack{ \gamma \in \Gamma_\infty \backslash \Gamma / \Gamma_{\bm{c}_l} \\ \alpha y \leq \mathrm{Im}(z_{\gamma \bm{c}_l}) < \beta y}} \tilde{f}( \mathrm{Re}(z_{\gamma \bm{c}_l}), \Gamma n( \mathrm{Re}(z_{\bm{c}_l})) a(y)).
  \end{equation}
  Applying theorem \ref{theorem:surfaceequidistribution0}, the right side becomes
  \begin{equation}
    \label{eq:ftildelimit}
    \int_{\mathbb{T}} \int_{\Gamma \backslash G} \tilde{f}(x, g) \dd \nu_{\alpha, \beta}(g) \dd x = \frac{\alpha^{-1} - \beta^{-1}}{2\pi \mathrm{vol}_{\mathbb{H}}(\Gamma \backslash \mathbb{H})} \int_1^{\varepsilon_l^2} \int_{\mathbb{T}} f(x, \log s) \dd x \frac{\dd s}{s}.
  \end{equation}
  Equation (\ref{eq:curveequidistribution}) follows after changing variables $t = \log s$. 
\end{proof}

\section{Conditioned geodesic line processes}
\label{sec:conditioned}

In section \ref{sec:convergence} we considered the number of points $\scrN_{I,\alpha,\beta}(x,y)$ in small intervals \eqref{eq:NIdef}, where $x$ is distributed according to an absolutely continuous Borel probability measure on $\TT$. In order to capture gap and nearest neighbour statistics, as well as pair and higher correlation measures, we now evaluate $\scrN_{I,\alpha,\beta}(\eta,y)$ for $\eta$ randomly drawn from the sequence itself. That is, the random variable $\eta$ is uniformly from the finite multiset $X_{\alpha,\beta}(y)$ as defined in \eqref{def:Xab}, with fixed $1\leq \alpha<\beta\leq \infty$. 
As in section \ref{sec:convergence}, we define the corresponding geodesic line processes by
\begin{equation}
\Theta_{y,\alpha,\beta}^0 = \sum_{l=1}^h \sum_{\gamma\in\Gamma/\Gamma_{\bm{c}_l}} \delta_{z_{(n(\eta)a(y))^{-1}\gamma \bm{c}_l}} 
\end{equation}
and
\begin{equation}
\Theta_{\alpha,\beta}^0 = \sum_{l=1}^h \sum_{\gamma\in\Gamma/\Gamma_{\bm{c}_l}} \delta_{z_{g^{-1}\gamma \bm{c}_l}}
\end{equation}
with the random element $g$ distributed with respect to the Borel probability measure $\widehat\nu_{\alpha,\beta}$ on $\Gamma\backslash G$, which is $\nu_{\alpha,\beta}$ normalised as a probability measure. That is,
\begin{equation}
\widehat\nu_{\alpha,\beta} = \kappa_\Gamma^{-1} \bigg(\frac1\alpha-\frac1\beta\bigg)^{-1} \nu_{\alpha,\beta}.
\end{equation}
We may think of $\Theta_{y,\alpha,\beta}^0$ and $\Theta_{\alpha,\beta}^0$ as the processes $\Theta_{y,\lambda}$ resp.\ $\Theta$ (defined in section \ref{sec:convergence}) conditioned to have a point on the line $\{ z\in\HH : \Re z=0,\; \alpha\leq \Im z<\beta\}$.

\subsection{Convergence in distribution}
\label{sec:convergence2}

The key result of this section is the following analogue of theorem \ref{theorem:convergence}.

\begin{theorem}\label{theorem:convergence2}
For $1\leq\alpha<\beta\leq\infty$, $0\leq a<b\leq 1$, we have convergence 
$\Theta_{y,\alpha,\beta}^0 \to \Theta_{\alpha,\beta}^0$ in distribution as $y\to 0$. 

In particular, for all $k_1,\ldots,k_r\in\ZZ_{\geq 0}$, finite intervals $I_i$, $0\leq a<b\leq 1$,  and $1\leq\alpha_i<\beta_i\leq\infty$, we have that
\begin{equation} \begin{split}
&\lim_{y\to 0} y \#\big\{ \xi \in X_{\alpha,\beta}(y) \cap ([a,b)+\ZZ)  : 
\scrN_{I_i,\alpha_i,\beta_i}(\xi,y) = k_i \;\forall i  \big\} \\
&= \kappa_\Gamma (b-a) \bigg(\frac1\alpha-\frac1\beta\bigg) \PP\big( \Theta_{\alpha,\beta}^0(B_{I_i,\alpha_i,\beta_i})= k_i \; \forall i \big) 
\end{split} \end{equation}
and the limit is a continuous function of $\alpha_i$, $\beta_i$ and the endpoints of $I_i$.
\end{theorem}

In the special case of the void distribution (recall theorem \ref{theorem:convergence}) we have for instance that for any finite interval $I$ (or finite unions of finite intervals),
\begin{equation}
\lim_{y\to 0} y \#\big\{ \xi \in X_{\alpha,\beta}(y)   : 
\scrN_{I,\alpha_1,\beta_1}(\xi,y) = 0  \big\}
= \kappa_\Gamma \bigg(\frac1\alpha-\frac1\beta\bigg) \PP\big( \Theta_{\alpha,\beta}^0(B_{I,\alpha_1,\beta_1})= 0 \big) 
\end{equation}
with the formula
\begin{equation}
\begin{split}
& \PP\big( \Theta_{\alpha,\beta}^0(B)= 0 \big) 
= \int_{\Gamma\backslash G} \bigg( \prod_{l=1}^h \prod_{\gamma\in\Gamma/\Gamma_{\bm{c}_l}} \Big( 1- \chi_{B}( z_{g^{-1}\gamma \bm{c}_l} ) \Big)\bigg) \dd\widehat\nu_{\alpha,\beta}(g) \\
& = \kappa_\Gamma \bigg(\frac1\alpha-\frac1\beta\bigg) \sum_{l_0=1}^h \int_{\alpha}^{\beta}  \int_1^{\varepsilon_{l_0}^2}  \bigg( \prod_{l=1}^h \prod_{\gamma\in\Gamma/\Gamma_{\bm{c}_l}} \Big( 1- \chi_{S_B}(a(v)k(\tfrac{\pi}{2})a(s) g_{l_0}^{-1} \gamma g_{l}) \Big)\bigg)  \frac{\dd s}{s}\; \frac{\dd v}{v^2} .
\end{split}
\end{equation}
In particular, the choice $I=(0,t]$ yields the probability that the gap between consecutive elements is greater than $t$, and $I=[-t,0)\cup(0,t]$ yields the probability that the distance to the nearest neighbour is greater than $t$.

The proof of theorem \ref{theorem:convergence2} follows the same route as that of theorem \ref{theorem:convergence}, except that, rather than equidistribution of closed long horocycles, we now use the discrete averages established in theorem \ref{theorem:surfaceequidistribution0}. We will not repeat the proof, other than to establish the following lemma which guarantees the regularity of the limit distribution. 

\begin{lemma}
  \label{lemma:regu}  
For any finite interval $I$, $1\leq\alpha<\beta\leq\infty$, $1\leq\alpha_1<\beta_1\leq\infty$, we have
\begin{equation}\label{PP1}
\PP\big( \Theta_{\alpha,\beta}^0(B_{\partial I,\alpha_1,\beta_1})\geq 1 \big) = 0  \qquad \text{if $0\notin\partial I$}, 
\end{equation}
\begin{equation}\label{PP2}
\PP\big( \Theta_{\alpha,\beta}^0(B_{\partial I,\alpha_1,\beta_1})=1 \big) = 1  \qquad \text{if $0\in\partial I$}, 
\end{equation}
and
\begin{equation}\label{PP3}
\PP\big( \Theta_{\alpha,\beta}^0(B_{I,\alpha_1,\alpha_1})\geq 1 \big) = 0 .
\end{equation}
\end{lemma}

\begin{proof}
To prove \eqref{PP3}, it is sufficient to show that the set 
\begin{equation}\label{countun}
\bigcup_{1\leq l_0\leq h} \bigcup_{\gamma\in\Gamma/\Gamma_{\bm{c}_l}} \Big\{ (s,v) : \Im z_{a(v)k(\frac{\pi}{2})a(s) g_{l_0}^{-1} \gamma \bm{c}_l}  = \alpha_1 \Big\} \subset \RR_{>0}^2
\end{equation}
has measure zero with respect to $\frac{\dd s}{s}\,\frac{\dd v}{v^2}$. Note that
$z_{a(v)k(\frac{\pi}{2})a(s) g_{l_0}^{-1} \gamma \bm{c}_l} = v z_{k(\frac{\pi}{2})a(s) g_{l_0}^{-1} \gamma \bm{c}_l}$,
which implies that the set $\{ (s,v) : \Im z_{a(v)k(\frac{\pi}{2})a(s) g_{l_0}^{-1} \gamma \bm{c}_l}  = \alpha_1 \}$ has measure zero for every $\gamma$, $l_0$, and so does the countable union \eqref{countun}. This proves \eqref{PP3}. Analogously, to prove \eqref{PP1}, we need to show that
\begin{equation}\label{countun2}
\bigcup_{1\leq l_0\leq h} \bigcup_{\gamma\in\Gamma/\Gamma_{\bm{c}_l}} \Big\{ (s,v) : \Re z_{g(s, v) \gamma \bm{c}_l}  = \epsilon \Big\} \subset \RR_{>0}^2,
\end{equation}
where $g(s, v) = g(s, v, l_0) = a(v)k(\frac{\pi}{2}) a(s) g_{l_0}^{-1}$,
has measure zero for $\epsilon\neq 0$. This follows from the same argument as above. For \eqref{PP2}, we need to investigate
\eqref{countun2} in the case $\epsilon=0$. Here we have
\begin{equation}\label{setzero}
\Big\{ (s,v) : \Re z_{g(s, v) \gamma \bm{c}_l} = 0 \Big\}
= \Big\{ (s,v) : \Re z_{g(s, 1) \gamma \bm{c}_l} = 0 \Big\}.
\end{equation}
If $l\neq l_0$, or if $l=l_0$ and $\gamma\neq \Gamma_{\bm{c}_l}$, there is a unique $s>0$ such that $\Re z_{g(s,1) \gamma \bm{c}_l} = 0$, which shows \eqref{setzero} has measure zero in these cases. The remaining case is $l=l_0$, $\gamma= \Gamma_{\bm{c}_l}$, which yields $\Re z_{g(s,1) \gamma \bm{c}_l} = \Re z_{k(\frac{\pi}{2})\bm{c}_0} = 0$ for all $v,s$ and hence corresponds to the probability one event.
\end{proof}

\subsection{The intensity measure}
\label{section:intensity23}

We now turn to the intensity measure of the process $\Theta_{\alpha,\beta}^0$, defined by
\begin{equation}\label{def:intensity}
\EE\Theta_{\alpha,\beta}^0(B) = \int_{\Gamma \backslash G} \scrN_B(g) \dd\widehat\nu_{\alpha,\beta}(g),
\end{equation}
and prove an analogue of proposition \ref{lemma:Faverage}. We first prove the following lemma that shows \eqref{def:intensity} is finite (and which will be useful to prove the convergence of moments in the next section).

\begin{lemma}
  \label{lemma:measurebound4321}
  Let $\mathcal{I}$, $\bm{\alpha}$, $\bm{\beta}$ as in section \ref{sec:moments}, and $\bm{k} = (k_1, \dots, k_r)$ with $k_j \geq 0$ integers.
  Then there exists $c_1 > 0$ such that
  \begin{equation}
    \label{eq:measurebound000}
    \sup_{y>0}  y \#\big\{ \xi \in X_{\alpha_0,\beta_0}(y)  : 
\scrN_{I_i,\alpha_i,\beta_i}(\xi,y) \geq k_i \;\forall i  \big\} \ll \exp( - c_1 \max_{1\leq j\leq r} k_j)  
  \end{equation}
  and
 \begin{equation}
    \label{eq:measurebound1000}
    \PP\big( \Theta_{\alpha,\beta}^0(B_{I_i,\alpha_i,\beta_i})\geq k_i \; \forall i \big) 
 \ll \exp( -c_1 \max_{1\leq j \leq r} k_j ). 
  \end{equation}
\end{lemma}


\begin{proof}
We assume without loss of generality that $[\alpha_i,\beta_i)=[1,\infty)$ and label $j$ so that $k_j=\max_{1\leq i\leq r} k_i$.
Take $\lambda$ to be the uniform probability measure on $[0,1]$, and $I_\epsilon=[-\frac{\epsilon}{2},\frac{\epsilon}{2}]$ for some fixed $\epsilon>0$. Then
\begin{equation} \begin{split}
&\#\big\{ \xi \in X_{\alpha_0,\beta_0}(y)  : 
\scrN_{I_i,\alpha_i,\beta_i}(\xi,y) \geq k_i \;\forall i  \big\}  \\
&\leq \frac{1}{\epsilon}
\sum_{k=1}^\infty k\; \lambda\big(\big\{ x \in \TT  : \scrN_{I_\epsilon,1,\infty}(x,y) = k ,\; 
\scrN_{I_i,1,\infty}(x,y) \geq k_i \;\forall i  \big\}\big) .
\end{split} \end{equation}
Lemma \ref{lemma:measurebound} tells us that this is bounded above by
\begin{equation}
\ll \sum_{k=1}^\infty k \exp( - c_0 \max\{k,k_j\})  \ll \exp( - \tfrac12 c_0 k_j) . 
\end{equation}
This proves \eqref{eq:measurebound000}, which in turn, in view of theorem \ref{theorem:convergence2}, implies the bound \eqref{eq:measurebound1000}.
\end{proof}

We will now work out the expectation measure $\EE\Theta_{\alpha,\beta}^0$. By linearity, it is sufficient to consider the case $\beta=\infty$ and test sets of the form $B=B_{I,\alpha_2,\infty}$.

\begin{proposition}
  \label{lemma:Faverage2}
  Let $\alpha_1,\alpha_2\in[1,\infty)$ and $I$ a finite interval. Then
  \begin{equation}
    \label{eq:Faverage22}
    \EE\Theta_{\alpha_1,\infty}^0(B_{I,\alpha_2,\infty}) =  \min\bigg\{1,\frac{\alpha_1}{\alpha_2}\bigg\}\; \delta_0(I) + 
     \int_I W_{\alpha_1,\alpha_2}(v) \dd v 
  \end{equation}
where $\delta_0$ denotes the delta mass at the orgin and $W_{\alpha_1,\alpha_2}:\RR\to \RR_{\geq 0}$ is the even and continuous function given by
  \begin{equation}
    \label{eq:paircorrelation}
    W_{\alpha_1,\alpha_2}(v) = \frac{\alpha_1}{\ell v^2} \sum_{l_1,l_2=1}^h \sum_{\substack{\gamma \in \Gamma_{\bm{c}_{l_1}} \backslash \Gamma / \Gamma_{\bm{c}_{l_2}} \\ \gamma \bm{c}_{l_2} \neq \bm{c}_{l_1}, \overline{\bm{c}_{l_1}}}} 
    H_{\mathrm{sign}(g_{l_1}^{-1} \gamma g_{l_2}(0))}(q(\gamma, l_1, l_2), v \alpha_1^{-1}, v\alpha_2^{-1}),
  \end{equation}
  where $q(\gamma, l_1, l_2) = \frac{r + 1}{r -1}$ with $r$ the cross-ratio
  \begin{equation}
    \label{eq:rdef}
    r = \frac{((\gamma \bm{c}_{l_2})^+ - \bm{c}_{l_1}^-)((\gamma \bm{c}_{l_2})^- - \bm{c}_{l_1}^+)}{((\gamma \bm{c}_{l_2})^+ - \bm{c}_{l_1}^+)((\gamma \bm{c}_{l_2})^- - \bm{c}_{l_1}^-)},
  \end{equation}
  and
  \begin{multline}\label{def:H+}
    H_+(q,v_1,v_2) \\
    = \int_{-1}^1 \chi_{[1,\infty)}\left( 2v_1 \frac{s + q}{s^2 + 2qs + 1}\right) \chi_{[1, \infty)} \left( v_2  \frac{ 1 - s^2}{s^2 + 2qs + 1}\right)  \left| \frac{s^2 + 2qs + 1}{(s + q)(s^2 - 1)} \right| \dd s ,
  \end{multline}
  \begin{multline}\label{def:H-}
    H_-(q,v_1,v_2) = \left( \int_{-\infty}^{-1} + \int_1^\infty \right) \\
    \chi_{[1, \infty)}\left( 2v_1 \frac{s + q}{s^2 + 2qs + 1}\right)  \chi_{[1, \infty)} \left( v_2 \frac{ 1 - s^2}{s^2 + 2qs + 1}\right) \left| \frac{s^2 + 2qs + 1}{(s + q)(s^2 - 1)} \right| \dd s.
  \end{multline}
In the case $v_1=v_2=v$, we have explicitly
  \begin{equation}
    \label{eq:F+qvdef}
    H_+(q,v,v) =
    \begin{cases}
      0 & \mathrm{if\ } q < -1 \\
      0 & \mathrm{if\ } -1 < q < 1 \mathrm{\ and\ } v < \sqrt{2 - 2q} \\
      h_q(s_1(q, v)) - h_q(s_2(q, v)) & \mathrm{if\ } -1 < q < 1 \mathrm{\ and\ } v > \sqrt{2 - 2q} \\
      h_q(s_1(q, v)) - h_q(-q + \sqrt{q^2 - 1})  & \mathrm{if\ } q > 1, \\ 
    \end{cases}
  \end{equation}
  \begin{equation}
    \label{eq:F-qvdef}
    H_-(q,v,v) =
    \begin{cases}
      0 & \mathrm{if\ } q < -1 \mathrm{\ and\ } |v| < \sqrt{2 - 2q} \\
      h_q(s_1(q, v)) - h_q(s_2(q, v)) & \mathrm{if\ } q < -1 \mathrm{\ and\ } |v| > \sqrt{2 - 2q} \\
      h_q(s_1(q, v)) - h_q(s_2(q, v)) & \mathrm{if\ } -1 < q < 1 \mathrm{\ and\ } v < - \sqrt{2 - 2q} \\
      0 & \mathrm{if\ } - 1 < q < 1 \mathrm{\ and\ } v > - \sqrt{2 - 2q} \\
      h_q(- q - \sqrt{q^2 - 1}) - h_q(s_2(q, v))& \mathrm{if\ } q > 1,
    \end{cases}
  \end{equation}
  with
  \begin{equation}
    \label{eq:fdef1}
    h_q(s) = \log \frac{s + q}{1 - s^2},
  \end{equation}
  \begin{equation}
    \label{eq:s1def1}
    s_1(q, v) = \frac{ -q +\sqrt{v^2 + q^2 - 1}}{v + 1},
  \end{equation}
  and
  \begin{equation}
    \label{eq:s2def1}
    s_2(q, v) = v - q - \sqrt{v^2 + q^2 - 1}.
  \end{equation}
\end{proposition}

\begin{proof}
With the notation as in the proof of proposition \ref{lemma:Faverage}, and $B=B_{I,\alpha_2,\infty}$,
  \begin{equation} \begin{split}
    \label{eq:unfolded12}
     &\int_{\Gamma \backslash G} \scrN_{I,\alpha_2,\infty}(g) \dd \widehat\nu_{\alpha_1,\infty}(g) \\
    &= \frac{\alpha_1}{\ell}  \sum_{l_1,l_2=1}^h
\sum_{\gamma\in\Gamma/\Gamma_{\bm{c}_{l_2}}} \int_{\alpha_1}^{\infty} \int_1^{\varepsilon_{l_1}^2}   \chi_{S_B}(a(v)k(\tfrac{\pi}{2})a(s) g_{l_1}^{-1} \gamma g_{l_2}) \;\frac{\dd s}{s}\; \frac{\dd v}{v^2} .
\end{split} \end{equation}
Recall that $\chi_{S_B}$ is the characteristic function of the set $S_B = \{ g \in G : z_{g \bm{c}_0} \in B \}$.  In view of lemma \ref{lemma:measurebound4321}, the series is absolutely convergent.

The contribution to \eqref{eq:unfolded12} of terms with $l_1 = l_2 = l$ and $\gamma =\Gamma_{\bm{c}_{l_2}}$ is
 \begin{equation}
 \label{eq:diagonalcontribution}
\frac{\alpha_1}{\ell}  \sum_{l=1}^h 
\int_{\alpha_1}^{\infty} \int_1^{\varepsilon_l^2}   \chi_{S_B}( a(v)k(\tfrac{\pi}{2})a(s)) \;\frac{\dd s}{s}\; \frac{\dd v}{v^2} 
= \min\bigg\{1,\frac{\alpha_1}{\alpha_2}\bigg\}\; \delta_0(I),
\end{equation}
since $a(v)k(\tfrac{\pi}{2})a(s)\in S_B$ if and only if $0\in I$ and $v\geq \alpha_2$. This explains the first term on the right hand side of \eqref{eq:Faverage22}.

To evaluate the remaining terms, we break the sum over $\gamma \in \Gamma / \Gamma_{\bm{c}_{l_2}}$ ($\gamma \neq \Gamma_{\bm{c}_{l_2}}$) into a sum over $\Gamma_{\bm{c}_{l_1}}$ and a sum over double cosets $\Gamma_{\bm{c}_{l_1}} \backslash \Gamma / \Gamma_{\bm{c}_{l_2}}$.
  To this end, we first note that we have the partition
  \begin{equation}
    \label{eq:doublecosetsinglecoset}
    \Gamma_{\bm{c}_{l_1}} \gamma \Gamma_{\bm{c}_{l_2}} = \bigcup_{\gamma' \in \Gamma_{\bm{c}_{l_1}}} \gamma' \gamma \Gamma_{\bm{c}_{l_2}},
  \end{equation}
  with a disjoint union, if and only if $\gamma^{-1} \Gamma_{\bm{c}_{l_1}} \gamma \cap \Gamma_{\bm{c}_{l_2}}$ is trivial.
  Now if there were an element in this intersection, then it would fix the endpoints of $\bm{c}_{l_2}$ and $\gamma^{-1} \bm{c}_{l_1}$, and so if this element were nontrivial we would have $\gamma \bm{c}_{l_2} = \bm{c}_{l_1}$ or $\gamma \bm{c}_{l_2} = \overline{\bm{c}_{l_1}}$, where $\overline{\bm{c}_{l_1}}$ is the geodesic $\bm{c}_{l_1}$ with the opposite orientation.
  In the first case we would have $l_1 = l_2 = l$ and $\gamma \in \Gamma_{\bm{c}_l}$, which we have already considered. In the second case $\gamma \bm{c}_{l_{2}} = \overline{\bm{c}_{l_1}}$, the positive orientation condition implicit in the definition of the $z_{\bm{c}}$ implies that these terms give no contribution.

For all other terms we can use the double coset decomposition to unfold the integral over $s$, and their total contribution is
\begin{equation}
\begin{split}
\label{eq:mixedsecondmoment2}
&  \sum_{l_1,l_2=1}^h 
 \sum_{\substack{\gamma \in \Gamma / \Gamma_{\bm{c}_{l_2}}\\ \gamma\notin \Gamma_{\bm{c}_{l_2}}}} \int_{\alpha_1}^{\infty}  \int_1^{\varepsilon_{l_1}^2}  \chi_{S_B}(a(v)k(\tfrac{\pi}{2})a(s) g_{l_1}^{-1} \gamma g_{l_2}) \;\frac{\dd s}{s}\; \frac{\dd v}{v^2} \\
&    
=  \sum_{l_1,l_2=1}^h \sum_{\substack{\gamma \in \Gamma_{\bm{c}_{l_1}} \backslash \Gamma / \Gamma_{\bm{c}_{l_2}} \\ \gamma \bm{c}_{l_2} \neq \bm{c}_{l_1}, \overline{\bm{c}_{l_1}}}}
\int_{\alpha_1}^{\infty}  \int_0^\infty   \chi_{S_B}( a(v)k(\tfrac{\pi}{2})a(s) g_{l_1}^{-1} \gamma g_{l_2}) \;\frac{\dd s}{s}\; \frac{\dd v}{v^2}\\
& =  \sum_{l_1,l_2=1}^h \sum_{\substack{\gamma \in \Gamma_{\bm{c}_{l_1}} \backslash \Gamma / \Gamma_{\bm{c}_{l_2}} \\ \gamma \bm{c}_{l_2} \neq \bm{c}_{l_1}, \overline{\bm{c}_{l_1}}}} \int_0^\infty \int_{\alpha_1}^{\infty}  \chi_{I}( \mathrm{Re}(z_{g^{-1} \gamma \bm{c}_{l_2}})) \chi_{[\alpha_2,\infty)}(\mathrm{Im}(z_{g^{-1} \gamma \bm{c}_{l_2}})) \; \frac{\dd v}{v^2} \;\frac{\dd s}{s},
\end{split}
\end{equation}
where $g = g_{l_1} a(s) k( -\tfrac{\pi}{2}) a(v)^{-1}$. 
  We now replace $s \gets s s_\gamma^{-1}$ with $s_\gamma$ chosen so that the backward endpoint of the geodesic $a(s_\gamma) g_{l_1}^{-1} \gamma \bm{c}_{l_2}$ is $\pm 1$, the $\pm$ determined by whether the image of $0$ under $g_{l_1}^{-1} \gamma g_{l_2}$ is positive or negative. 
 The forward endpoint of the geodesic $a(s_\gamma) g_{l_1}^{-1} \gamma \bm{c}_{l_2}$ is then $\pm r$, where $r = r(\gamma, l_1, l_2)$ is the cross ratio of the points $\bm{c}_{l_1}^-, \bm{c}_{l_1}^+, (\gamma \bm{c}_{l_2})^+, (\gamma \bm{c}_{l_2})^-$, where we use $\bm{c}^{\pm}$ to denote the forward and backward endpoints of a geodesic $\bm{c}$. 
  We note that $r$ clearly does not depend on the choice of representative of the double coset $\Gamma_{\bm{c}_{l_1}} \gamma \Gamma_{\bm{c}_{l_2}}$, but that on the other hand $r(\gamma_1) = r(\gamma_2)$ only implies $G_{\bm{c}_{l_1}} \gamma_1 G_{\bm{c}_{l_2}} = G_{\bm{c}_{l_1}} \gamma_2 G_{\bm{c}_{l_2}}$, not necessarily $\Gamma_{\bm{c}_{l_1}} \gamma_1 \Gamma_{\bm{c}_{l_2}} = \Gamma_{\bm{c}_{l_1}} \gamma_{2} \Gamma_{\bm{c}_{l_2}}$.
  Here $G_{\bm{c}}$ is the stabilizer in $G$ of a geodesic $\bm{c}$. 

  With this change in $s$, we now have for $g = g_{l_1} a(s) k(-\tfrac{\pi}{2}) a(v)^{-1}$,
  \begin{equation}
    \label{eq:realz}
    \mathrm{Re}(z_{g^{-1} \bm{c}_{l_2}}) = \frac{v}{2} \left( \frac{-s \pm r}{s \pm r} - \frac{-s \pm 1}{s \pm 1}\right),
  \end{equation}
  and
  \begin{equation}
    \label{eq:imagz}
    \mathrm{Im}(z_{g^{-1} \bm{c}_{l_2}}) =  \frac{v}{2}  \left( \frac{-s \pm r}{s \pm r} + \frac{-s \pm 1}{s \pm 1} \right).
  \end{equation}
  We change variables
  \begin{equation}
    \label{eq:ttos}
    s \gets \pm \frac{-s + 1}{s + 1},\ \frac{\dd s}{s} \gets 2 \frac{\dd s}{ | s^2 - 1|},
  \end{equation}
  and the range of integration of $s$ is $\int_{-1}^1$ in the $+$ case and $\int_{-\infty}^{-1} + \int_{1}^\infty$ in the $-$ case.
  We have in both cases that
  \begin{align}
    \label{eq:RezImz}
    \mathrm{Re}(z_{g^{-1} \gamma \bm{c}_{l_2}}) & = \frac{v}{2}\frac{(r + 1)s + (r-1)}{(r-1)s + (r+1)} + s = \frac{v}{2} \frac{s^2 + 2qs +1}{s + q}, \\ \nonumber
    \mathrm{Im}(z_{g^{-1} \gamma \bm{c}_{l_2}}) & = \frac{v}{2} \frac{(r + 1)s + (r-1)}{(r-1)s + (r+1)} - s = \frac{v}{2} \frac{ 1 - s^2}{s + q},
  \end{align}
  where $ q = \frac{r + 1}{r -1}$.

  The integral in \eqref{eq:mixedsecondmoment2} becomes
  \begin{equation}
    \label{eq:Hgamma}
    2 \int_{-1}^1 \int_{\alpha_1}^{\infty}  \chi_{I} \left( \frac{v}{2} \frac{s^2 + 2qs + 1}{s + q}\right)\chi_{[\alpha_2, \infty)}\left( \frac{v}{2} \frac{ 1 - s^2}{s + q}\right) \frac{\dd v}{v^2}\frac{\dd s}{1 - s^2} 
  \end{equation}
  in the $+$ case and
  \begin{equation}
    \label{eq:Hgamma1}
    2 \left( \int_{-\infty}^{-1} + \int_1^\infty \right)\int_{\alpha_1}^{\infty}  \chi_{I}\left( \frac{v}{2} \frac{s^2 + 2qs + 1}{s + q}\right)\chi_{[\alpha_2, \infty)} \left( \frac{v}{2} \frac{ 1 - s^2}{s + q} \right) \frac{\dd v}{v^2} \frac{\dd s}{s^2 - 1}
  \end{equation}
  in the $-$ case.
  We change variables
  \begin{equation}
    \label{eq:vchange}
    v \gets 2 \frac{s + q}{s^2 + 2qs + 1} v,\ \frac{\dd v}{v^2} \gets  \frac{1}{2} \left| \frac{s^2 + 2qs + 1}{s + q} \right| \frac{\dd v}{v^2}
  \end{equation}
  to obtain
  \begin{equation}
    \label{eq:Hgamma2}
     \int_{I} H_{\pm}(q, v \alpha_1^{-1},v \alpha_2^{-1}) \frac{\dd v}{v^2} ,
  \end{equation}
  with $H_{\pm}(q, v_1,v_2)$ as in \eqref{def:H+} and \eqref{def:H-} as claimed. Note here that, as we have only positive terms in the above rearrangements, the series  
    \begin{equation}
    \label{eq:paircorrelation45600}
    \sum_{l_1,l_2=1}^h \sum_{\substack{\gamma \in \Gamma_{\bm{c}_{l_1}} \backslash \Gamma / \Gamma_{\bm{c}_{l_2}} \\ \gamma \bm{c}_{l_2} \neq \bm{c}_{l_1}, \overline{\bm{c}_{l_1}}}} 
    \int_I H_{\mathrm{sign}(g_{l_1}^{-1} \gamma g_{l_2}(0))}(q(\gamma, l_1, l_2), v\alpha_1^{-1},v \alpha_2^{-1}) \,\frac{\dd v}{v^2},
  \end{equation}
is still absolutely convergent. Thus for given $\alpha_1,\alpha_2$, the series 
 \begin{equation}
    \label{eq:paircorrelation456}
    \sum_{l_1,l_2=1}^h \sum_{\substack{\gamma \in \Gamma_{\bm{c}_{l_1}} \backslash \Gamma / \Gamma_{\bm{c}_{l_2}} \\ \gamma \bm{c}_{l_2} \neq \bm{c}_{l_1}, \overline{\bm{c}_{l_1}}}} 
    \frac1{v^2} \, H_{\mathrm{sign}(g_{l_1}^{-1} \gamma g_{l_2}(0))}(q(\gamma, l_1, l_2), v \alpha_1^{-1},v \alpha_2^{-1}) ,
  \end{equation}
  converges absolutely for almost every $v$. This implies, by the monotonicity 
  \begin{equation}
H_\pm(q,v' \alpha_1^{-1},v' \alpha_2^{-1})\leq H_\pm(q,v \alpha_1^{-1},v \alpha_2^{-1})\quad \text{if $0<|v'|\leq |v|$}, 
\end{equation}
that the series \eqref{eq:paircorrelation456}  converges for all $v'$ uniformly on compacta not containing the origin.
Lemmas \ref{lemma:doublecosetcount} and \ref{lemma:originbehaviour} below establish the uniform convergence of (\ref{eq:paircorrelation456}) for $v$ in a neighbourhood of $0$, and so (\ref{eq:paircorrelation456}) is continuous as a function of $v$ as each term in the series is continuous.
  
  We now specialise to the case $\alpha_1=\alpha_2=1$. 
  We note that in the range of integration of $H_+(q,v)$, $1 - s^2 \geq 0$, and so the condition
  \begin{equation}
    \label{eq:condition1}
    v\frac{1 - s^2}{s^2 + 2qs + 1}  \geq 1
  \end{equation}
  forces $v$ and $s^2 + 2qs + 1$ to have the same sign, and so the condition
  \begin{equation}
    \label{eq:condition2}
    2v \frac{s + q}{s^2 + 2qs + 1}  \geq 1
  \end{equation}
  forces $s + q \geq 0$.
  Similarly, since $1 - s^2 \leq 0$ in the range of integration of $H_-(q,v)$, $v$ and $s^2 + 2qs + 1$ have opposite signs, and so $s + q \leq 0$ in this case.

  We split into three cases depending on whether $1 < q$, $-1 < q < 1$, or $q < -1$ (note that $q$ is finite and not $\pm 1$ since $r$, being a cross-ratio, is none of $0$, $1$, $\infty$).
  In the case when $q > 1$, $s + q\geq 0$ does not affect the range of integration for $H_+$, $\int_{-1}^1$, but $s + q \leq 0$ restricts the range of integration for $H_-$ to $\int_{-\infty}^{-q}$.
  In the case when $-1 < q < 1$, the range of integration for $H_+$ is restricted to $\int_{-q}^1$, while that of $H_-$ is restricted to $\int_{-\infty}^{-1}$.
  Finally, in the case when $q < -1$, the condition $s + q \geq 0$ forces $H_+= 0$, while $s + q \leq 0$ restricts the range of integration for $H_-$ to $\int_{-\infty}^{-1} + \int_1^{-q}$.

  We note that if $|q| < 1$, then $s^2 + 2qs + 1 > 0$ for all $s$, while if $|q| > 1$, $s^2 + 2qs + 1$ has roots at $s = -q \pm \sqrt{q^2 - 1}$.
  We note that when $ q < -1$, $0 < -q - \sqrt{q^2 - 1} < 1 < -q + \sqrt{q^2 -1}$, while if $q > 1$, then $-q - \sqrt{q^2 - 1} < -1 < -q + \sqrt{q^2 - 1} < 0$.
  In the cases $ q > 1 $ and $q < -1$, we subdivide the range of integration to keep track of the sign of $s^2 + 2qs + 1$.
  In the case $q > 1$ we write $H_+ = I_1^+ + I_2^+$ with $I_1^+$, $I_2^+$ having ranges of integration $\int_{-1}^{-q + \sqrt{q^2 - 1}}$, $\int_{-q + \sqrt{q^2 -1}}^1$ respectively, and we write $H_- = I_1^- + I_2^- $ with $I_1^-$, $I_2^-$ having ranges of integration $\int_{-\infty}^{-q - \sqrt{q^2 - 1}}$, $\int_{-q - \sqrt{q^2 - 1}}^{-q}$ respectively.
  In the case $q < -1$, we write $H_- = I_3 + I_4$ with $I_3$, $I_4$ having ranges of integration $\int_{-\infty}^{-1}$, $\int_{1}^{-q}$ respectively.
  Writing $I_5^\pm$ for $H_\pm(q, v)$ in the case $-1 < q < 1$, we thus have $8$ integrals, $I_1^\pm, I_2^\pm, I_3, I_4, I_5^\pm$, to analyze.

  We start by considering the functions
  \begin{equation}
    \label{eq:v1def}
    v_1(s) = \frac{s^2 + 2qs + 1}{1 - s^2}
  \end{equation}
  and
  \begin{equation}
    \label{eq:v2def}
    v_2(s) = \frac{1}{2} \frac{s^2 + 2qs + 1}{s + q},
  \end{equation}
  which are clearly related to the conditions \eqref{eq:condition1}, \eqref{eq:condition2}.
  We note that $v_1(s)$ is increasing for all $s$ if $q > 1$ and decreasing for all $s$ if $q < -1$.
  When $0 < q < 1$, $v_1(s)$ is decreasing for $- \frac{1}{q} - \sqrt{\frac{1}{q^2} - 1} < s < - \frac{1}{q} + \sqrt{\frac{1}{q^2} - 1}$ and increasing otherwise, and when $ -1 < q < 0$, $v_1(s)$ is increasing for $- \frac{1}{q} - \sqrt{\frac{1}{q^2} - 1} < s < - \frac{1}{q} + \sqrt{\frac{1}{q^2} - 1}$ and decreasing otherwise.
  As for $v_2(s)$, it is increasing for all $s$ if $|q| > 1$, and when $|q| < 1$, $v_2(s)$ is decreasing for $-q - \sqrt{1 - q^2} < s < -q + \sqrt{1 - q^2}$ and is increasing otherwise.

  We also define the functions
  \begin{equation}
    \label{eq:s1def}
    s_1(v) = \frac{ -q +\sqrt{v^2 + q^2 - 1}}{v + 1}
  \end{equation}
  and
  \begin{equation}
    \label{eq:s2def}
    s_2(v) = v - q - \sqrt{v^2 + q^2 - 1}
  \end{equation}
  for all $v$ if $|q| > 1$ and $|v| > \sqrt{1 - q^2}$ if $|q| < 1$.
  We note that $s_2(v)$ is the smaller solution to $v = v_2(s)$ and that $s_1(v)$ is the smaller solution to $v = v_1(s)$ when $v < -1$ and the larger solution when $v > -1$.

  Finally, we define, for $s$ in the range of the integrals $I_1^\pm, I_2^\pm, I_3, I_4, I_5^\pm$,
  \begin{equation}
    \label{eq:fdef45}
    h_q(s) = \log \frac{s + q}{1 - s^2},
  \end{equation}
  so that
  \begin{equation}
    \label{eq:fprime}
    h_q'(s) = \frac{s^2 + 2qs + 1}{(s + q)(1 - s^2)}.
  \end{equation}
  We note that $|h_q'(s)| = h_q'(s)$ in the range of integration for $I_2^+, I_1^-, I_3, I_5^\pm$ and $|h_q'(s)| = -h_q'(s)$ in the range of integration for $I_1^+, I_2^-, I_4$.

  For $I_1^+$ we note that the integral is $0$ unless $v < 0$, and the conditions \eqref{eq:condition1}, \eqref{eq:condition2} become $ v \leq  v_1(s)$, $ v \leq v_2(s)$.
  We observe that in this case $v_1(s) \leq v_2(s)$ for all $s$ in the range of integration, and thus
  \begin{equation}
    \label{eq:I1+}
    I_1^+ =
    \begin{cases}
      h_q(s_1(v)) - h_q(-q + \sqrt{q^2 - 1}) & \mathrm{if\ } v < 0 \\
      0 & \mathrm{if\ } v > 0.
    \end{cases}
  \end{equation}
  Similarly, for $I_2^+$ we note that the integral is $0$ unless $v > 0$, and the conditions \eqref{eq:condition1}, \eqref{eq:condition2} become $v \geq  v_1(s)$, $v \geq  v_2(s)$.
  Here we find $v_1(s) \geq v_2(s)$ for all $s$ in the range of integration and hence 
 \begin{equation}
    \label{eq:I2+}
    I_2^+ =
    \begin{cases}
      0 & \mathrm{if\ } v < 0 \\
      h_q(s_1(v)) - h_q( - q + \sqrt{q^2 - 1}) & \mathrm{if\ } v > 0.
    \end{cases}
  \end{equation}

  For $I_1^-$, we note that the integral is $0$ unless $v < 0$, and the conditions \eqref{eq:condition1}, \eqref{eq:condition2} become $ v \leq v_1(s)$ and $ v \leq v_2(s)$.
  In this case we have $v_2(s) \leq v_1(s)$ for all $s$ in the range of integration, and we find
  \begin{equation}
    \label{eq:I1-}
    I_1^- =
    \begin{cases}
      h_q( - q - \sqrt{q^2 - 1}) - h_q(s_2(v)) & \mathrm{if\ } v < 0 \\
      0 & \mathrm{if\ } v > 0.
    \end{cases}
  \end{equation}
  Similarly, for $I_2^-$ we note that the integral is $0$ unless $ v > 0$, and the condition \eqref{eq:condition1}, \eqref{eq:condition2} become $v \geq v_1(s)$ and $ v \geq v_2(s)$.
  Here we have $v_2(s) \geq v_1(s)$ for all $s$ in the range of integration, and 
    \begin{equation}
    \label{eq:I2-}
    I_2^- =
    \begin{cases}
      0 & \mathrm{if\ } v < 0 \\
      h_q(- q - \sqrt{q^2 - 1}) - h_q(s_2(v)) & \mathrm{if\ } v > 0.
    \end{cases}
  \end{equation}

  For $I_3$ we note that the integral is $0$ unless $v < 0$, and the conditions \eqref{eq:condition1}, \eqref{eq:condition2} become $v \leq v_1(s)$, $v \leq v_2(s)$.
  In the range of integration, $s = -1 - \sqrt{2 - 2q}$ is the unique point at which $v_1(s) = v_2(s)$, and their value at this point is $ -\sqrt{2 - 2q}$.
  Since $v_1(s)$ is decreasing and $v_2(s)$ is increasing, we obtain
  \begin{equation}
    \label{eq:I3}
    I_3 =
    \begin{cases}
      h_q(s_1(v)) - h_q(s_2(v)) & \mathrm{if\ } v < - \sqrt{2 - 2q} \\
      0 & \mathrm{if\ }  v > - \sqrt{2 - 2q}.
    \end{cases}
  \end{equation}
  Similarly, for $I_4$ we note that the integral is $0$ unless $v > 0$ and the conditions \eqref{eq:condition1}, \eqref{eq:condition2} become $ v \geq v_1(s)$, $ v \geq v_2(s)$.
  In the range of integration, $s = - 1 + \sqrt{2 - 2q}$ is the unique point at which $v_1(s) = v_2(s)$, and their value at this point is $\sqrt{2 - 2q}$.
  Since $v_1(s)$ is decreasing and $v_2(s)$ is increasing, we obtain
  \begin{equation}
    \label{eq:I4}
    I_4 =
    \begin{cases}
      0 & \mathrm{if\ }  v < \sqrt{2 - 2q} \\
      h_q(s_1(v)) - h_q(s_2(v)) & \mathrm{if\ } v > \sqrt{2 - 2q}.
    \end{cases}
  \end{equation}

  For $I_5^+$ we note that the integral is $0$ unless $v > 0$ and the conditions \eqref{eq:condition1}, \eqref{eq:condition2} become $v \geq v_1(s)$, $ v \geq v_2(s)$.
  In the range of integration, $s = -1 + \sqrt{2 - 2q}$ is the unique point at which $v_1(s) = v_2(s)$, and their value at this point is $\sqrt{2 - 2q}$.
  We find that $v_1(s)$ is increasing on $ -1 + \sqrt{2 - 2q} < s < 1$ and that $v_2(s)$ is decreasing on $-q < s < -1 + \sqrt{2 - 2q}$, and so
  \begin{equation}
    \label{eq:I5+}
    I_5^+ =
    \begin{cases}
      0 & \mathrm{if\ }  v < \sqrt{2 - 2q} \\
      h_q(s_1(v)) - h_q(s_2(v)) & \mathrm{if\ } v > \sqrt{2 - 2q}.
    \end{cases}
  \end{equation}
  Similarly, for $I_5^-$ we note that the integral is $0$ unless $v < 0$ and the conditions \eqref{eq:condition1}, \eqref{eq:condition2} become $ v \leq v_1(s)$, $ v \leq v_2(s)$.
  In the range of integration, $s = - 1 - \sqrt{2 - 2q}$ is the unique point at which $v_1(s) = v_2(s)$, and their value at this point is $- \sqrt{2 - 2q}$.
  We find that $v_1(s)$ is decreasing on $- 1 - \sqrt{2 - 2q} < s < -1$ and that $v_2(s)$ is increasing on $-\infty < s < -1 - \sqrt{2 - 2q}$, and so
  \begin{equation}
    \label{eq:I5-}
    I_5^- =
    \begin{cases}
      h_q(s_1(v)) - h_q(s_2(v)) & \mathrm{if\ }  v < - \sqrt{2 - 2q} \\
      0 & \mathrm{if\ }  v > - \sqrt{2 - 2q}.
    \end{cases}
  \end{equation}
  Putting these calculations together, we obtain \eqref{eq:F+qvdef} and \eqref{eq:F-qvdef}.
\end{proof}

We remark that the cross-ratio $r$ in (\ref{eq:rdef}) is a classical quantity and $q = \frac{r + 1}{r -1}$ appears very naturally when counting geodesics, see for example the results by Good \cite{Good1983} proving asymptotic formula for the number of $\gamma \in \Gamma_{\bm{c}_{l_1}} \backslash \Gamma / \Gamma_{\bm{c}_{l_2}}$ with $q(\gamma, l_1, l_2) \leq Q$ as $Q \to \infty$.
In what follows we need an upper bound for this count of the correct order of magnitude, and in the interest of being self-contained, we give a proof of this bound despite \cite{Good1983} providing much stronger information.

\begin{lemma}
  \label{lemma:doublecosetcount}
  For $1 \leq l_1, l_2 \leq h$ and $Q > 1$, the number of $\gamma \in \Gamma_{\bm{c}_{l_1}} \backslash \Gamma / \Gamma_{\bm{c}_{l_2}}$ such that $| q(\gamma, l_1, l_2)| \leq Q$ is $O(Q)$ with implied constant depending on $\Gamma$ and the geodesics $\bm{c}_{l_1}$, $\bm{c}_{l_2}$. 
\end{lemma}

\begin{proof}
  We recall from the properties of the cross-ratio $r = r(\gamma, l_1, l_2)$, (\ref{eq:rdef}), that there is $g \in G$ such that $g \bm{c}_{l_1}$ is the geodesic from $0$ to $\infty$ and $g \gamma \bm{c}_{l_2}$ is the geodesic from $\pm 1$ to $\pm r$, the sign being the same as that of $g_{l_1}^{-1} \gamma g_{l_2} (0)$.
It follows that when $g_{l_1}^{-1} \gamma g_{l_2} (0) > 0$ there is $g \in G$ such that $ g\bm{c}_{l_1}$ is the geodesic from $-1$ to $1$ and $g \gamma \bm{c}_{l_2}$ is the geodesic from $\infty$ to $-q$, and when $g_{l_1}^{-1} \gamma g_{l_2} (0) < 0$ there is $g \in G$ such that $g\bm{c}_{l_1}$ is the geodesic from $1$ to $-1$ and $g \gamma \bm{c}_{l_2}$ is the geodesic from $\infty$ to $q$.
In either case we see that $\bm{c}_{l_1}$ intersects $\gamma \bm{c}_{l_2}$ if and only if $ -1 < q < 1$, and when this occurs $|q| = |\cos \theta|$ with $\theta$ the intersection angle between $\bm{c}_{l_1}$ and $\gamma \bm{c}_{l_2}$.  We note that there are only finitely many $q = q(\gamma, l_1, l_2)$ with $|q | < 1$.
  Indeed, these correspond to intersections between $\bm{c}_{l_1}$ and $\gamma \bm{c}_{l_2}$ in $\mathbb{H}$, or intersections between the geodesics $\bm{c}_{l_1}$ and $\bm{c}_{l_2}$ in $\Gamma \backslash \mathbb{H}$.
  As these distinct geodesics are closed, there can only be finitely many such intersections.

When $|q| > 1$ we find $g \in G$ such that $g \bm{c}_{l_1}$ is the geodesic from $-1$ to $1$ (or $1$ to $-1$) and $g \gamma \bm{c}_{l_2}$ is a geodesic connecting $\pm \log d$, $d$ consequently being the distance between the geodesics.
Expressing $d$ in terms of $q$, we find that $|q| = \cosh d$.

Let us fix points $z_1 \in \bm{c}_{l_1}$ and $z_2 \in \bm{c}_{l_2}$.
  Now let $z_1' \in \bm{c}_{l_1}$ and $z_2' \in \bm{c}_{l_2}$ be such that the distance between $\bm{c}_{l_1}$ and $\gamma \bm{c}_{l_2}$, namely $\cosh^{-1}(|q|)$, is the hyperbolic distance between $z_1'$ and $\gamma z_2'$.
  Let $\gamma_1\in \Gamma_{\bm{c}_{l_1}}$ and $\gamma_2 \in \Gamma_{\bm{c}_{l_2}}$ be such that the distance between $\gamma_1 z_1$ and $z_1'$ is at most $\log \varepsilon_{l_1}$ and the distance between $\gamma_2 z_2$ and $z_2'$ is at most $\log \varepsilon_{l_2}$.
  Allowing the implied constant to depend on $\bm{c}_{l_1}, \bm{c}_{l_2}$, it follows that the distance between $z_1$ and $\gamma_1^{-1} \gamma \gamma_2 z_2 $ is $\cosh^{-1}(|q|) + O(1)$.

  Since $\gamma_1^{-1} \gamma \gamma_2 \in \Gamma_{\bm{c}_{l_1}} \gamma \Gamma_{\bm{c}_{l_2}}$, the above shows that
  \begin{equation}
    \label{eq:doublecosetbound}
    \# \{ \gamma \in \Gamma_{\bm{c}_{1}} \backslash \Gamma / \Gamma_{\bm{c}_{l_2}} : 1 < |q(\gamma, l_1, l_2)| \leq Q \} \leq \# \{ \gamma \in \Gamma : \mathrm{d}_{\mathbb{H}}(z_1, \gamma z_2) \leq \cosh^{-1}(Q) + O(1)\}
  \end{equation}
  where $\mathrm{d}_{\mathbb{H}}$ is the hyperbolic metric.
  As the set of $z\in \mathbb{H}$ such that $\mathrm{d}_{\mathbb{H}}(z_1, z) \leq \cosh^{-1}(Q) + O(1)$ has area $O(Q)$, the lemma follows. 
\end{proof}

\begin{lemma}
  \label{lemma:originbehaviour}
  For $q = q(\gamma, l_1, l_2)$ and $\frac{v}{\alpha_1}$, $\frac{v}{\alpha_2}$ sufficiently small, we have
  \begin{equation}
    \label{eq:originbound}
    H_{\pm}(q, \alpha_1^{-1}v, \alpha_2^{-1} v) =
    \begin{cases}
      0 & \mathrm{if\ } q < 1 \\
      O(\frac{v^2}{q^2}) & \mathrm{if\ } q > 1,
    \end{cases}
  \end{equation}
  where the implied constants depend only on the group $\Gamma$ and the geodesics $\bm{c}_l$.
  In particular (\ref{eq:paircorrelation456}) converges absolutely and uniformly for $v$ near $0$.

  Moreover, we have
  \begin{equation}
    \label{eq:originasymptotic}
    H_{\pm}(q, v, v) =
    \begin{cases}
      0 & \mathrm{if\ } q < 1 \\
      \frac{q - \sqrt{q^2 -1}}{2\sqrt{q^2 -1}}v^2 + O(\frac{v^3}{q^2})& \mathrm{if\ } q > 1,
    \end{cases}
  \end{equation}
  and thus
  \begin{equation}
    \label{eq:densityat0}
    W_{1,1}(v) = \frac{1}{\ell} \sum_{l_1, l_1 = 1}^h \sum_{\substack{q = q(\gamma, l_1, l_2) > 1\\ \gamma \in \Gamma_{\bm{c}_{l_1}} \backslash \Gamma / \Gamma_{\bm{c}_{l_2}} \\ \gamma \bm{c}_{l_2} \neq \bm{c}_{l_1}, \overline{\bm{c}_{l_1}}}} \frac{q - \sqrt{q^2 -1}}{2\sqrt{q^2 -1}} + O(v). 
  \end{equation}
\end{lemma}

\begin{proof}
  For the proof of this lemma we break into the same cases as in the proof of proposition \ref{lemma:Faverage2}, namely the eight integrals $I_1^\pm, I_2^\pm, I_3, I_4, I_5^\pm$.
  We recall that
  \begin{align}
    \label{eq:Irecall}
    H_+ &= I_1^+ + I_2^+ = \int_{-1}^{-q + \sqrt{q^2 -1}}  + \int_{-q + \sqrt{q^2 -1}}^1  \quad \mathrm{when\ } q > 1, \\\nonumber
    H_- & = I_1^- + I_2^- = \int_{-\infty}^{-q - \sqrt{q^2 -1}} + \int_{-q -\sqrt{q^2 -1}}^{-q}  \quad \mathrm{when\ } q > 1, \\ \nonumber
    H_+ & = 0 \quad \mathrm{when\ } q < -1, \\\nonumber
    H_- & = I_3 + I_4 = \int_{-\infty}^{-1} + \int_{1}^{-q} \quad \mathrm{when\ } q < -1, \\ \nonumber
    H_+ & = I_5^+ = \int_{-q}^1  \quad \mathrm{when\ } -1 < q < 1, \\\nonumber
    H_- & = I_5^+ = \int_{-\infty}^{-1} \quad \mathrm{when\ } -1 < q < 1.
  \end{align}
  We recall that the integrand of all the above integrals is $\chi_1(s,v) \chi_2(s,v)$ where $\chi_1(s,v)$ is the indicator function of the condition
  \begin{equation}
    \label{eq:alpha1bound}
    2v \frac{s + q}{s^2 + 2qs + 1} \geq \alpha_1
  \end{equation}
  and $\chi_2(s, v)$ is the indicator function of the condition
  \begin{equation}
    \label{eq:alpha2bound}
    v \frac{1 - s^2}{s^2 + 2qs + 1} \geq \alpha_2.
  \end{equation}
  We first show that for $v$ sufficiently small, $I_3 = I_4 = I_5^\pm = 0$, thus verifying the first case in (\ref{eq:originbound}). 

  We have that $I_3 = 0$ if $v \geq 0$ and the conditions (\ref{eq:alpha1bound}), (\ref{eq:alpha2bound}) become $v \leq \alpha_1 v_2(s)$ and $v \leq \alpha_2 v_1(s)$.
  We note however that $v_1(s) \leq 1$ for all $s < -1$, and so we conclude that $I_3 = 0$ for all $\frac{v}{\alpha_2}$ sufficiently small.
  Similarly, we have that $I_4 = 0$ if $v \leq 0$ and the conditions (\ref{eq:alpha1bound}), (\ref{eq:alpha2bound}) become $v \geq \alpha_1 v_2(s)$ and $v \geq \alpha_2 v_1(s)$.
  We note however that $v_2(s) \geq 1$ for all $1 < s < -q$, and so $I_4 = 0$ for all $\frac{v}{\alpha_1}$ sufficiently small.

  We have that $I_5^+ = 0$ if $v \leq 0$ and the conditions (\ref{eq:alpha1bound}), (\ref{eq:alpha2bound}) become $v \geq v_2(s) \alpha_1$ and $v \geq \alpha_2 v_1(s)$.
  We note that in this case we have $v_2(s) \geq \sqrt{1 - q^2}$, and since by lemma \ref{lemma:doublecosetcount} there are only finitely many $-1 < q < 1$, it follows that $I_5^+ = 0$ for all $\frac{v}{\alpha_1}$ sufficiently small.
  Similarly, we have that $I_5^- = 0$ if $v \geq 0$ and the conditions (\ref{eq:alpha1bound}), (\ref{eq:alpha2bound}) become $v \leq v_2(s) \alpha_1$ and $v \leq v_1(s) \alpha_2$.
  We note that $v_2(s) \leq -\sqrt{1 - q^2}$, so as before we have $I_5^+ = 0$ for $\frac{v}{\alpha_1}$ sufficiently small.

  It remains to consider the cases $I_1^{\pm}$ and $I_2^{\pm}$, i.e. when $q > 1$.
  We have that $I_1^+ = 0$ if $v\geq 0$ and the conditions (\ref{eq:alpha1bound}), (\ref{eq:alpha2bound}) become $v \leq \alpha_1 v_2(s)$ and $v \leq \alpha_2 v_1(s)$.
  To upper bound $I_1^+$ we only consider the condition $v \leq \alpha_2 v_1(s)$ and we obtain
  \begin{equation}
    \label{eq:I1plusbound}
    I_1^+ \leq h_q(s_1(\alpha_2^{-1}v)) - h_q( - q + \sqrt{q^2 -1}). 
  \end{equation}
  Similarly, we have that $I_2^+ = 0$ if $v \leq 0$ and the conditions (\ref{eq:alpha1bound}), (\ref{eq:alpha2bound}) become $v \leq \alpha_1 v_2(s)$ and $v \leq \alpha_2 v_1(s)$.
  By dropping the first condition we obtain the same bound as before,
  \begin{equation}
    \label{eq:I2plusbound}
    I_2^+ \leq h_q( s_1( \alpha_2^{-1} v)) - h_q( -q + \sqrt{ q^2 - 1}).
  \end{equation}

  We have that $I_1^- = 0$ if $v \geq 0$ and the conditions (\ref{eq:alpha1bound}), (\ref{eq:alpha2bound}) become $v \leq \alpha_1 v_2(s)$ and $v \leq \alpha_2 v_1(s)$.
  Dropping the second condition we obtain the upper bound
  \begin{equation}
    \label{eq:I1minusbound}
    I_1^- \leq h_q( - q - \sqrt{q^2 -1}) - h_q( s_2( \alpha_1^{-1} v)).
  \end{equation}
  Similarly, we have that $I_2^- = 0$ if $v \leq 0$ and the conditions (\ref{eq:alpha1bound}), (\ref{eq:alpha2bound}) become $v \geq \alpha_1 s_2(v)$ and $v \geq \alpha_2 s_1(v)$.
  Dropping the second condition we obtain the same bound
  \begin{equation}
    \label{eq:I2minusbound}
    I_2^- \leq h_q( - q - \sqrt{q^2 -1}) - h_q( s_2( \alpha_1^{-1} v)).
  \end{equation}
  We remark that, by analysing when $v_1(s) \leq v_2(s)$, $v_1(s) \geq v_2(s)$ for all $s$ in the range of integration of $I_1^\pm, I_2^\pm$, these upper bounds are in fact equalities for $q > \frac{\alpha_1^2 + \alpha_2^2}{2\alpha_1\alpha_2}$.
  By lemma \ref{lemma:doublecosetcount} this condition is satisfied for all but finitely many double cosets.

  We now compute the upper bounds (\ref{eq:I1plusbound}), (\ref{eq:I2plusbound}) and (\ref{eq:I1minusbound}), (\ref{eq:I2minusbound}) asympotically as $v \to 0$.
  This will not only verify (\ref{eq:originbound}) but also (\ref{eq:originasymptotic}) since these bounds are equalities for $\alpha_1 = \alpha_2 =1$.
  We have
  \begin{multline}
    \label{eq:Hplusasymptotic}
    h_q(s_1(v)) - h_q( - q + \sqrt{q^2 - 1}) \\
    = \log (v +1) + \log \frac{qv + \sqrt{v^2 + q^2 - 1}}{\sqrt{q^2 -1}} - \log \frac{v + q\sqrt{v^2 + q^2 -1} - q^2 + 1}{q\sqrt{q^2 -1} - q^2 + 1}
  \end{multline}
  for $q > 1$ and $v$ sufficiently small.
  Using that $q \geq 1 + \eta$ for some $\eta > 0$ depending only on $\Gamma$ and the $\bm{c}_{l}$ by lemma (\ref{lemma:doublecosetcount}), we expand
  \begin{equation}
    \label{eq:sqrtexpand}
    \sqrt{v^2 + q^2  - 1} = \sqrt{q^2 -1} \left( 1 + \frac{v^2}{2(q^2 -1)} + O( \frac{v^4}{q^4})\right).
  \end{equation}
  We find that the first and second term in (\ref{eq:Hplusasymptotic}) is
  \begin{multline}
    \label{eq:secondterm}
    \log ( v + 1) + \log( 1 + \frac{qv}{\sqrt{q^2 -1}} + \frac{v^2}{2(q^2 -1)} + O(\frac{v^4}{q^4}) ) \\
    =  \log( 1 + \bigg(1 + \frac{q}{\sqrt{q^2 -1}}\bigg)v + \bigg(\frac{1}{2(q^2-1)} + \frac{q}{\sqrt{q^2-1}}\bigg)v^2) + O(\frac{v^3}{q^2}). 
  \end{multline}
  We also find that the third term is
  \begin{multline}
    \label{eq:thirdterm}
    \log ( 1 + \frac{v}{q\sqrt{q^2 - 1} - q^2 + 1} + \frac{qv^2}{2\sqrt{q^2-1}(q\sqrt{q^2 -1} - q^2 + 1)} + O( \frac{v^4}{q^2})) \\
    = \log( 1 + \bigg(1 + \frac{q}{\sqrt{q^2 -1}}\bigg) v + \frac{1}{2}\bigg( \frac{q}{\sqrt{q^2-1}} + \frac{q^2}{q^2 -1}\bigg) v^2 ) + O(\frac{v^4}{q^2}). 
  \end{multline}
  We now have
  \begin{equation}
    \label{eq:Hplusasymptotic2}
    h_q(s_1(v)) - h_q( -q + \sqrt{q^2-1})
    = \log \frac{1 + \bigg(1 + \frac{q}{\sqrt{q^2 -1}}\bigg)v + \bigg(\frac{1}{2(q^2-1)} + \frac{q}{\sqrt{q^2-1}}\bigg)v^2}{1 + \bigg(1 + \frac{q}{\sqrt{q^2 -1}}\bigg) v + \frac{1}{2}\bigg( \frac{q}{\sqrt{q^2-1}} + \frac{q^2}{q^2 -1}\bigg) v^2} + O(\frac{v^3}{q^2}).
  \end{equation}
  Writing
  \begin{multline}
    \label{eq:Hplusasymptotic3}
    \frac{1 + \bigg(1 + \frac{q}{\sqrt{q^2 -1}}\bigg)v + \bigg(\frac{1}{2(q^2-1)} + \frac{q}{\sqrt{q^2-1}}\bigg)v^2}{1 + \bigg(1 + \frac{q}{\sqrt{q^2 -1}}\bigg) v + \frac{1}{2}\bigg( \frac{q}{\sqrt{q^2-1}} + \frac{q^2}{q^2 -1}\bigg) v^2} \\
    = 1 + \frac{\frac{q - \sqrt{q^2 -1}}{2\sqrt{q^2-1}}v^2}{1 + \bigg(1 + \frac{q}{\sqrt{q^2 -1}}\bigg) v + \frac{1}{2}\bigg( \frac{q}{\sqrt{q^2-1}} + \frac{q^2}{q^2 -1}\bigg) v^2} = 1 + \frac{q - \sqrt{q^2 -1}}{2\sqrt{q^2-1}} v^2 + O(\frac{v^3}{q^2}),
  \end{multline}
  we obtain
  \begin{equation}
    \label{eq:Hplusasymptotic4}
    h_q(s_1(v)) - h_q( -q + \sqrt{q^2-1}) = \frac{q - \sqrt{q^2 -1}}{2\sqrt{q^2 -1}} v^2 + O( \frac{v^3}{q^2}).
  \end{equation}
  This establishes (\ref{eq:originbound}) and (\ref{eq:originasymptotic}) in the $+$ case.
  
  It remains to analyze
  \begin{multline}
    \label{eq:Hminusasymptotic}
    h_q(-q - \sqrt{q^2 -1}) - h_q(s_2(v)) \\
    = -\log \frac{-v + \sqrt{v^2 + q^2 -1}}{ \sqrt{q^2 -1}} + \log \frac{q^2 -1 -qv + v^2 + (q -v)\sqrt{v^2 + q^2 - 1}}{q^2 -1 + q\sqrt{q^2 -1}}.
  \end{multline}
  Using (\ref{eq:sqrtexpand}), the first term here is
  \begin{equation}
    \label{eq:firstterm}
    \log( 1 - \frac{v}{\sqrt{q^2 -1}} + \frac{v^2}{2(q^2-1)}) + O(\frac{v^4}{q^4}),
  \end{equation}
  and the second term is
  \begin{equation}
    \label{eq:nextterm}
    \log( 1 - \frac{v}{\sqrt{q^2 -1}} + ( \frac{q - \sqrt{q^2 -1}}{\sqrt{q^2 -1}} + \frac{ q^2 - q\sqrt{q^2 -1}}{2(q^2 -1)} ) v^2 ) + O(\frac{v^3}{q^2}).
  \end{equation}
  We now have
  \begin{equation}
    \label{eq:Hminusasymptotic3}
    h_q(-q - \sqrt{q^2 -1}) - h_q(s_1(v)) = \log \frac{1 - \frac{v}{\sqrt{q^2 -1}} + ( \frac{q - \sqrt{q^2 -1}}{\sqrt{q^2 -1}} + \frac{ q^2 - q\sqrt{q^2 -1}}{2(q^2 -1)} ) v^2}{1 - \frac{v}{\sqrt{q^2 -1}} + \frac{v^2}{2(q^2-1)}} + O(\frac{v^3}{q^2}). 
  \end{equation}
  Writing
  \begin{multline}
    \label{eq:Hminusasymptotic4}
    \frac{1 - \frac{v}{\sqrt{q^2 -1}} + ( \frac{q - \sqrt{q^2 -1}}{\sqrt{q^2 -1}} + \frac{ q^2 - q\sqrt{q^2 -1}}{2(q^2 -1)} ) v^2}{1 - \frac{v}{\sqrt{q^2 -1}} + \frac{v^2}{2(q^2-1)}} \\
    = 1 + \frac{ \frac{q - \sqrt{q^2 -1}}{2\sqrt{q^2 -1}} v^2 }{ 1 - \frac{v}{\sqrt{q^2 -1}} + \frac{v^2}{2(q^2-1)}} = 1 + \frac{q - \sqrt{q^2 -1 }}{2\sqrt{q^2 -1}} v^2 + O( \frac{v^3}{q^2}), 
  \end{multline}
  it follows that
  \begin{equation}
    \label{eq:Hminusasymptotic2}
    h_q(-q - \sqrt{q^2 -1}) - h_q(s_1(v)) = \frac{q - \sqrt{q^2 -1}}{2\sqrt{q^2 -1}} v^2  + O(\frac{v^3}{q^2}). 
  \end{equation}
  This is enough to establish (\ref{eq:originbound}) and (\ref{eq:originasymptotic}). 
\end{proof}

 \subsection{Moments}\label{sec:moments2}

Let us return to the moments of the counting functions $\scrN_{I, \alpha, \beta}(\xi,y)$, now averaging $\xi\in X_{\alpha_0,\beta_0}(y)$ over the sequence itself rather than the absolutely continuous measure $\lambda$.
For $y > 0$, $\bm{s} = (s_1, \dots, s_r) \in \mathbb{C}^r$, $\mathcal{I} = I_1 \times \cdots \times I_r$, $I_j \subset \mathbb{R}$ a finite interval so that $0\notin\partial I_j$, and $\bm{\alpha} = (\alpha_0, \dots, \alpha_r)$, $\bm{\beta} = ( \beta_0, \dots, \beta_r)$ satisfying $1\leq \alpha_j < \beta_j \leq \infty$, 
define the moments
\begin{equation}
  \label{eq:Momentsdef0}
  M_{\mathcal{I}, \bm{\beta}, \bm{\alpha}}^{a,b}(\bm{\kappa}, y) =y
    \sum_{\xi\in X_{\alpha_0,\beta_0}(y)\cap ([a,b)+\ZZ)}  \prod_{1\leq j \leq r} \scrN_{I_j, \alpha_j, \beta_j}(\xi, y)^{\kappa_j} ,
\end{equation}
\begin{equation}
  \label{eq:Mdef10}
  M_{\mathcal{I}, \bm{\alpha}, \bm{\beta}}^0(\bm{\kappa}) = \int_{\Gamma \backslash G} \prod_{1\leq j \leq r} \scrN_{I_j, \alpha_j, \beta_j}(g)^{\kappa_j} \dd \widehat\nu_{\alpha_0,\beta_0}(g) ; 
\end{equation}
and the moment generating functions 
\begin{equation}
  \label{eq:Gdef0}
  G_{\mathcal{I}, \bm{\alpha}, \bm{\beta}}^{a,b}(\bm{s}, y)  = y
    \sum_{\xi\in X_{\alpha_0,\beta_0}(y)\cap ([a,b)+\ZZ)} \exp( \sum_{1\leq j \leq r} s_j \scrN_{I_j, \alpha_j, \beta_j}(x, y))  ,
\end{equation}
\begin{equation}
  \label{eq:Gdef10}
    G_{\mathcal{I}, \bm{\alpha}, \bm{\beta}}^0(\bm{s}) = \int_{\Gamma \backslash G} \exp( \sum_{1\leq j\leq r} s_j \scrN_{I_j, \alpha_j, \beta_j}(g)) \dd \widehat\nu_{\alpha_0,\beta_0}(g) .
  \end{equation}

\begin{theorem}
  \label{theorem:moments000}
  Let $0\leq a<b\leq 1$ and $\mathcal{I}$, $\bm{\alpha}$, $\bm{\beta}$ as above.
  Then there is a constant $c_0 > 0$ such that for $\bm{s} = (s_1, \dots, s_r) \in \mathbb{C}^r$ satisfying $ \sum_j \max\{ \mathrm{Re}(s_j), 0 \} < c_0 $ the function $G_{\mathcal{I}, \bm{\alpha}, \bm{\beta}}^0(\bm{s})$ is analytic and we have
  \begin{equation}
    \label{eq:momentconvergence000}
    \lim_{y \to 0} G^{a,b}_{\mathcal{I}, \bm{\alpha}, \bm{\beta}}(\bm{s}, y) = \kappa_\Gamma (b-a) \bigg(\frac1\alpha-\frac1\beta\bigg) G_{\mathcal{I}, \bm{\alpha}, \bm{\beta}}^0(\bm{s}).
  \end{equation}
\end{theorem}

As in section \ref{sec:moments}, theorem \ref{theorem:moments000} implies the convergence as $y \to 0$ of the mixed moments.

\begin{corollary}
  \label{corollary:moments000}
If $\bm{\kappa} = (\kappa_1, \dots, \kappa_r) \in \mathbb{R}^r$ with $\kappa_j \geq 0$, then $M_{\mathcal{I}, \bm{\alpha}, \bm{\beta}}^0(\bm{\kappa})$ is finite and
  \begin{equation}
    \label{eq:momentconvergence1000}
    \lim_{y \to 0} M^{a,b}_{\mathcal{I}, \bm{\alpha}, \bm{\beta}}(\bm{\kappa}, y) = \kappa_\Gamma (b-a) \bigg(\frac1\alpha-\frac1\beta\bigg) M_{\mathcal{I}, \bm{\alpha}, \bm{\beta}}^0(\bm{\kappa}). 
  \end{equation}
\end{corollary}

We highlight the special case of the first moment $r = 1$ and $\kappa = 1$.

\begin{corollary}
  \label{corollary:totalcount000}
  For any finite interval $I$ such that $0\notin\partial I$, $1\leq \alpha < \beta \leq \infty$ and $0\leq a < b \leq 1$, we have
  \begin{equation}
    \label{eq:totalcount000}
    \lim_{y\to 0} y
    \sum_{\xi\in X_{\alpha_0,\beta_0}(y)\cap ([a,b)+\ZZ)} \scrN_{I, \alpha, \beta} (\xi, y)  = \kappa_\Gamma (b-a) \bigg(\frac1\alpha-\frac1\beta\bigg) \int_{\Gamma \backslash G} \scrN_{I,\alpha,\beta}(g) \dd\widehat\nu_{1,\infty}(g)   .
    \end{equation}
 \end{corollary}
 
 The proof of theorem \ref{theorem:moments000} follows from the same argument as in the continuous case, with lemma \ref{lemma:measurebound} replaced by lemma \ref{lemma:measurebound4321}.

\subsection{Palm distribution and pair correlation}
\label{sec:palm}

We now return to the setting of a one-dimensional point process discussed in section \ref{sec:1dpp}.
Instead of
\begin{equation}
\Xi_{N,\lambda} = \sum_{j=1}^N \sum_{k\in\ZZ} \delta_{N(\xi_j-\xi+k)} 
\end{equation}
with $\lambda$ distributed at random with respect to $\lambda$, we consider the point process
\begin{equation}
\Xi_{N}^0 = \sum_{j=1}^N \sum_{k\in\ZZ} \delta_{N(\xi_j-\xi_p+k)} 
\end{equation}
where $p$ is a random integer drawn from $\{1,2,\ldots,N\}$ with uniform probability. Note that the expectation measure is given by
\begin{equation}
\EE \Xi_{N}^0 = \frac1N \sum_{j,p=1}^N \sum_{k\in\ZZ} \delta_{N(\xi_j-\xi_p+k)} 
= \delta_0 + R_{2,N} ,
\end{equation}
with the pair correlation measure
\begin{equation}
R_{2,N} = \frac1N \sum_{\substack{ j,p=1\\ j\neq p}}^N \sum_{k\in\ZZ} \delta_{N(\xi_j-\xi_p+k)} .
\end{equation}

We furthermore define the point process process $\Xi^0$ by
\begin{equation}
\Xi^0(I) =\Theta_{1,\infty}^0(B_{\kappa_\Gamma^{-1} I,1,\infty})
\end{equation}
with $\kappa_\Gamma$ as in \eqref{def:eta}. The process $\Xi^0$ is in fact distributed according to the Palm measure of the  process $\Xi$; see section 31 of \cite{Kallenberg2021} for a general discussion of Palm measures, or \cite{Marklof2017} for a context more closely analogous to ours. 
The intensity measure of the Palm measure is in view of proposition \ref{lemma:Faverage2} given by
  \begin{equation}
    \label{eq:Faverage_eta12}
    \EE\Xi^0(I) =  \EE\Theta_{1,\infty}^0(B_{\kappa_\Gamma^{-1} I,1,\infty})     = \delta_0(I) + \int_I  w(v) \dd v ,
  \end{equation}
where  $w(v)=\kappa_\Gamma^{-1} W(\kappa_\Gamma^{-1} v)$, and so
  \begin{equation}
    \label{eq:paircorrelation2}
    w(v) = \frac{1}{v^2 \mathrm{vol}_\HH(\Gamma\backslash\HH)} \sum_{l_1,l_2=1}^h \sum_{\substack{\gamma \in \Gamma_{\bm{c}_{l_1}} \backslash \Gamma / \Gamma_{\bm{c}_{l_2}} \\ \gamma \bm{c}_{l_2} \neq \bm{c}_{l_1}, \overline{\bm{c}_{l_1}}}} 
    H_{\mathrm{sign}(g_{l_1}^{-1} \gamma g_{l_2}(0))}(q(\gamma, l_1, l_2), \kappa_\Gamma^{-1} v).
  \end{equation}
  
This shows that the following statements are special cases of theorem \ref{theorem:convergence2} (convergence in distribution) and  corollary \ref{corollary:totalcount000} (pair correlation). As in the proof of theorem \ref{theorem:convergencealt}, the key observation is that we can move from the scaling by $y^{-1}$ to the scaling by $N$ by the regularity of the limiting distributions.

\begin{theorem}\label{theorem:convergencealt2}
The point processes $\Xi_N^0  \to \Xi^0$ converge in distribution, together with all moments, as $N\to \infty$. 

Specifically, for all $k_1,\ldots,k_r\in\ZZ_{\geq 0}$ and finite intervals $I_i$, we have that
\begin{equation}
\lim_{N\to \infty} \frac1N \#\big(\big\{ p\leq N  : \scrN_{I_i}(\xi_p,N) = k_i \;\forall i \big\}\big) 
= \PP\big( \Xi^0(I_i)= k_i \; \forall i \big) ,
\end{equation}
and for every finite interval $I$
\begin{equation}
\lim_{N\to \infty}  R_{2,N}(I) = \int_I  w(v) \dd v .
\end{equation}
\end{theorem}

The translation-stationarity of $\Xi$ implies cycle-stationarity of the Palm distributed $\Xi^0$. Furthermore the fact that $\Xi(\{0\}) = 1$ almost surely (this follows from \eqref{eq:Faverage_eta12}) implies that $\Xi^0$ and hence $\Xi$ are simple. For further details on this connection see \cite{Marklof2017} and references therein. We will explore a dynamical interpretation of this in the next section.

\subsection{Entry and return times}
\label{sec:entry}

Athreya and Cheung \cite{AthreyaCheung2014} constructed a Poincar\'e section (different from $\widetilde S_{\alpha,\beta}$) for the horocycle flow on $\mathrm{SL}(2,\mathbb{Z}) \backslash G$, so that the return times yield the gaps in the Farey sequence. 
In fact, the return times for the horocycle flow for our section $\widetilde S_{\alpha,\beta}$ generate the gaps between the roots of quadratic congruences when $D>0$. 
To see this, define for  initial condition $g\in G$ the {\em first entry time} for the section $\widetilde S_{\alpha,\beta}$ by the $\Gamma$-invariant function
\begin{equation}
\tau_{\alpha,\beta}(g) = \inf\{ u>0 : \Gamma g n(u) \in\widetilde S_{\alpha,\beta} \} .
\end{equation}
For $\Gamma g\in \widetilde S_{\alpha,\beta}$, this is also called the {\em first return time}.
In view of the discussion in section \ref{sec:surface}, $\tau_{\alpha,\beta}$ is a measurable function. It is a general fact, see \cite{Marklof2017}, that the ergodicity of the horocycle flow $\Gamma g \mapsto \Gamma g n(u)$ implies that
\begin{equation}
\int_{\Gamma \backslash G}
f(g) \dd\mu_\Gamma(g)=
\frac{1}{\overline\tau_{\alpha,\beta}}\int_{\Gamma \backslash G}
\bigg( \int_0^{\tau_{\alpha,\beta}(g)}  f(g n(u)) \,  \dd u\bigg) \dd \nu_{\alpha, \beta}(g) ,
\end{equation}
with the average return time (mean free path length)
\begin{equation}
\overline\tau_{\alpha,\beta} = \int_{\Gamma \backslash G} \tau_{\alpha,\beta}(g) \dd \nu_{\alpha, \beta}(g) .
\end{equation}
Let us define the $j$th entry time recursively by $\tau_{\alpha,\beta}^{(j)}(g)=\tau_{\alpha,\beta}(g n(\tau_{\alpha,\beta}^{(j-1)}(g)))$ with $\tau_{\alpha,\beta}^{(0)}(g)=0$, and $j\in\ZZ$ (i.e., backward and forward in time). For $g$ random, the point process
\begin{equation}
\sum_{j\in\ZZ} \delta_{\tau_j}, \qquad \tau_j:=\frac{\tau_{\alpha,\beta}^{(j)}(g)}{\overline\tau_{\alpha,\beta}},
\end{equation}
then has, by construction, the same distribution as the process $\Xi$ or $\Xi^0$, if $g$ is a random element in $\Gamma\backslash G$ distributed according to $\mu_\Gamma$ or $\widehat\nu_{\alpha,\beta}$, respectively. In summary, the random processes generated by the real parts of the tops of our geodesic line processes are the same as the entry times of the horocycle flow to the Poincar\'e section $\widetilde S_{\alpha,\beta}$.

\section{Geodesics and roots of quadratic congruences}
\label{sec:roots}

The aim of this section is to express the roots $\frac{\mu}{m}$ in terms of the tops of certain geodesics in the Poincar\'e upper half-plane $\mathbb{H}$. The key result is the following theorem \ref{theorem:topsandroots}. It shows that the results of sections \ref{sec:geodesic} and \ref{sec:conditioned} yield the convergence of all fine-scale statistics for the roots of quadratic congruences, including moments, gap distributions and pair correlation measures. In particular, the limit theorems \ref{theorem:convergencealt} and \ref{theorem:convergencealt2} for the one-dimensional point processes, turn into our main theorems \ref{thm1} and \ref{thm2} in the introduction for the special choice $n=1$.

\begin{theorem}
  \label{theorem:topsandroots}
  Let  $D > 0$ be square-free and $D\not\equiv 1 \bmod 4$. Fix $n > 0$ and $\nu \pmod n$ such that $\nu^2 \equiv D \pmod n$.
  Then there exists a finite set of geodesics $\{\bm{c}_1, \dots, \bm{c}_h\}$, all of which are closed and have the same length in $\Gamma_0(n) \backslash \mathbb{H}$, with the following properties:
  \begin{enumerate}[{\rm (i)}]
  \item For any integer $m > 0$ and $\mu \pmod m$ satisfying $\mu^2 \equiv D \pmod m$, $m \equiv 0 \pmod n$, and $\mu \equiv \nu \pmod n$, there is a unique $l$ and double coset $\Gamma_\infty \gamma \Gamma_{\bm{c}_l} \in \Gamma_\infty \backslash \Gamma_0(n) / \Gamma_{\bm{c}_l}$ such that
    \begin{equation}
      \label{eq:topsandroots}
      z_{\gamma \bm{c}_l} \equiv \frac{\mu}{m} + \ii \frac{\sqrt{D}}{m} \pmod {\Gamma_{\infty}}.
    \end{equation}
  \item Conversely, given $l$ and double coset $\Gamma_\infty \gamma \Gamma_{\bm{c}_l} \in \Gamma_\infty \backslash \Gamma_0(n) / \Gamma_{\bm{c}_l}$ with $\gamma \bm{c}_l$ positively oriented, there exist unique $m > 0$ and $\mu \pmod m$ satisfying $\mu^2 \equiv D \pmod m$, $m \equiv 0\pmod n$, and $\mu \equiv \nu \pmod n$ such that \eqref{eq:topsandroots} holds. 
  \end{enumerate}
\end{theorem}

We will see below that the number $h$ of the closed geodesics is equal to the narrow class number $h^+(D)$ of the quadratic number field $\mathbb{Q}(\sqrt{D})$. The narrow class group is defined as the quotient of the group of fractional ideals in $\mathbb{Q}(\sqrt{D})$ modulo the group of fractional ideals of the form $\xi\mathbb{Z}[\sqrt{D}]$ with $\xi\in\mathbb{Q}(\sqrt{D})$ totally positive; here $\mathbb{Z}[\sqrt{D}]$ is a maximal order by our assumption on $D$. Moreover, the stabiliser subgroups $\Gamma_{\bm{c}_l}$ are all conjugate in $G$ to the (projectively) cyclic subgroup generated by $\pm\begin{pmatrix} \varepsilon_0 & 0 \\ 0 & \varepsilon_0^{-1} \end{pmatrix}$, where $\varepsilon_0$ is the generator of the totally positive units of $\mathbb{Z}[\sqrt{D}]$, chosen so that $\varepsilon_0 > 1$.

The remainder of this section is dedicated to the proof of theorem \ref{theorem:topsandroots}. It is based on the amalgamation of two well-known correspondences.
The first is that given an integral binary quadratic form $F(X,Y)$ with positive discriminant, the geodesic in $\mathbb{H}$ connecting the two roots in $\mathbb{R}$ of $F(X,1)$ projects to a closed geodesic in $\mathrm{SL}(2,\mathbb{Z})\backslash \mathbb{H}$.  The second is that if $F(X,Y) = mX^2 - 2\mu XY + cY^2$ has discriminant $4D$, then $\mu^2 \equiv D \pmod m$, so then we see that the sum of the roots of $F(X,1)$ is $2\frac{\mu}{m}$, the difference of the roots is $2 \frac{\sqrt{D}}{m}$, and thus the top of the geodesic corresponding to $F(X,Y)$ is $\frac{\mu}{m} + \ii \frac{\sqrt{D}}{m}$.
In the remainder of this section we re-establish these connections in our own notation, but using ideals in $\mathbb{Z}[\sqrt{D}]$ instead of binary quadratic forms. We then examine them further in the case when $\mathrm{SL}(2,\mathbb{Z})$ is replaced by the congruence subgroup $\Gamma_0(n)$.

\subsection{Roots, ideals and binary quadratic forms}
\label{sec:roots2}

We start with the following proposition, connecting the roots of the congruence $\mu^2 \equiv D \pmod m$ with ideals in $\mathbb{Z}[\sqrt D]$.

\begin{proposition}
  \label{proposition:rootsideals}
  Let $D \in \mathbb{Z}$, $D \neq 1$, be square-free and $I \subset \mathbb{Z}[\sqrt{D}]$ be an ideal not divisible by any rational integers greater than $1$.
  Then there exists a $\mathbb{Z}$-basis of $I$, $\{\beta_1, \beta_2\}$, unique modulo the action of $\Gamma_\infty$, with the form
  \begin{equation}
    \label{eq:rootbasis}
    \begin{pmatrix}
      \beta_1 \\
      \beta_2
    \end{pmatrix}
    =
    \begin{pmatrix}
      1 & \mu \\
      0 & m
    \end{pmatrix}
    \begin{pmatrix}
      \sqrt{D} \\
      1
    \end{pmatrix}
    ,
  \end{equation}
  where $m > 0$ and $\mu^2 \equiv D \pmod {m}$.
  Conversely, given $m$ and $\mu \pmod m$ satisfying $\mu^2 \equiv D \pmod{m}$, the sublattice of $\mathbb{Z}[\sqrt{D}]$ with $\mathbb{Z}$-basis $\{\beta_1, \beta_2\}$ given by \eqref{eq:rootbasis} is an ideal of $\mathbb{Z}[\sqrt{D}]$.
\end{proposition}

\begin{proof}
  Reducing to Hermite normal form, any (rank $2$) lattice $I \subset \mathbb{Z}[\sqrt{D}]$ has a unique $\mathbb{Z}$-basis $\{\beta_1, \beta_2\}$ with the form
  \begin{equation}
    \label{eq:HNFbasis}
    \begin{pmatrix}
      \beta_1 \\
      \beta_2
    \end{pmatrix}
    =
    \begin{pmatrix}
      b_1 & b_2 \\
      0 & b_3
    \end{pmatrix}
    \begin{pmatrix}
      \sqrt{D} \\
      1
    \end{pmatrix}
    ,
  \end{equation}
  where $0 < b_1, b_3$ and $0 \leq b_2 < b_3$.
  We remark that the action of $\Gamma_\infty$ on the left of this basis is to simply add multiples of $b_3$ to $b_2$, and so one can consider only $b_2 \pmod {b_3}$ if one works modulo $\Gamma_\infty$.

  Since $\sqrt{D}$ generates the ring $\mathbb{Z}[\sqrt{D}]$, the sublattice $I$ being an ideal is equivalent to $\sqrt{D}I \subset I$, and since $\sqrt{D}I$ has basis
  \begin{equation}
    \label{eq:squarerootdbasis}
    \begin{pmatrix}
      \sqrt{D}\beta_1 \\
      \sqrt{D}\beta_2
    \end{pmatrix}
    =
    \begin{pmatrix}
      b_1 & b_2 \\
      0 & b_3
    \end{pmatrix}
    \begin{pmatrix}
      0 & D \\
      1 & 0 
    \end{pmatrix}
    \begin{pmatrix}
      b_1 & b_2 \\
      0 & b_3
    \end{pmatrix}
    ^{-1}
    \begin{pmatrix}
      \beta_1 \\
      \beta_2
    \end{pmatrix}
    ,
  \end{equation}
  $I$ being an ideal is equivalent to
  \begin{equation}
    \label{eq:integralmatrix}
    \frac{1}{b_1b_3}
    \begin{pmatrix}
      b_1 & b_2 \\
      0 & b_3
    \end{pmatrix}
    \begin{pmatrix}
      0 & D \\
      1 & 0 
    \end{pmatrix}
    \begin{pmatrix}
      b_3 & -b_2 \\
      0 & b_1
    \end{pmatrix}
    = \frac{1}{b_1b_3}
    \begin{pmatrix}
      b_2b_3 & -b_2^2 + Db_1^2 \\
      b_3^2 & -b_2b_3
    \end{pmatrix}
  \end{equation}
  having integer entries.
  From the first column and last row, we see that for $I$ to be an ideal, it is necessary that $b_1\mid b_2$ and $b_1 \mid b_3$.
  Writing $b_2 = \mu b_1$ and $b_3 = m b_1$, we see from the top right entry of \eqref{eq:integralmatrix} that $\mu^2 \equiv D \pmod {m}$ is necessary for $I$ to be an ideal.
  We note that, if $I$ is an ideal not divisible by rational integers greater than $1$, then $b_1 = 1$, and we obtain the first part of proposition \ref{proposition:rootsideals}.
  For the converse part, we simply note that if $\mu^2 \equiv D \pmod m$, then the above shows that the sublattice $I$ with basis given by \eqref{eq:rootbasis} satisfies $\sqrt{D}I \subset I$ as required. 
\end{proof}

We now produce different bases for ideals $I \subset \mathbb{Z}[\sqrt{D}]$ than those in \eqref{eq:rootbasis}. Any new basis is related to \eqref{eq:rootbasis} by an element of $\mathrm{SL}(2,\mathbb{Z})$.
We fix representatives $I_l$, $1\leq l \leq h^+(D)$, of the narrow class group of $\mathbb{Q}(\sqrt{D})$.
This implies that for any ideal $I \subset \mathbb{Z}[\sqrt{D}]$ there is a totally positive $\xi \in I_l^{-1}$ such that $I = \xi I_l$.
To obtain a basis for $I$ in this way, we first embed $\mathbb{Q}(\sqrt{D})$ in $\mathbb{R}^2$ via $a + b \sqrt{D} \mapsto (a + b\sqrt{D}, a - b\sqrt{D})$, so that $\mathbb{Z}[\sqrt{D}]$ is a lattice in $\mathbb{R}^2$.
(Here we have fixed $\sqrt{D}$ to be the positive square root of $D$.)
We adopt the notation that for $\xi \in \mathbb{Z}[\sqrt{D}]$ we use $\xi^{(1)}$ and $\xi^{(2)}$ to denote these coordinates.
With this notation, in \eqref{eq:rootbasis} we replace
\begin{equation}
  \label{eq:betareplace}
  \begin{pmatrix}
    \beta_1 \\
    \beta_2 
  \end{pmatrix}
  \gets
  \begin{pmatrix}
    \beta_1^{(1)} & \beta_1^{(2)} \\
    \beta_2^{(1)} & \beta_{2}^{(2)}
  \end{pmatrix}
  , \quad
  \begin{pmatrix}
    \sqrt{D} \\
    1
  \end{pmatrix}
  \gets
  \begin{pmatrix}
    \sqrt{D} & - \sqrt{D} \\
    1 & 1
  \end{pmatrix}
  .
\end{equation}
We fix $\mathbb{Z}$-bases $\{\beta_{l1}, \beta_{l2}\}$ for each of the $I_l$ so that
\begin{equation}
  \label{eq:signdetcondition}
  \mathfrak{B}_l = 
  \begin{pmatrix}
    \beta_{l1}^{(1)} & \beta_{l1}^{(2)} \\
    \beta_{l2}^{(1)} & \beta_{l2}^{(2)}
  \end{pmatrix}
\end{equation}
has positive determinant.
Then $\{\xi \beta_{l1}, \xi \beta_{l2}\}$ is a basis for $\xi I_l$, and the corresponding basis matrix is
\begin{equation}
  \label{eq:basismatrix}
  \begin{pmatrix}
    \xi^{(1)} \beta_{l1}^{(1)} & \xi^{(2)}\beta_{l1}^{(2)} \\
    \xi^{(1)} \beta_{l2}^{(1)} & \xi^{(2)} \beta_{l2}^{(2)}
  \end{pmatrix}
  = \mathfrak{B}_l
  \begin{pmatrix}
    \xi^{(1)} & 0 \\
    0 & \xi^{(2)}
  \end{pmatrix}
  .
\end{equation}

For $m > 0$ and $\mu \pmod m$ satisfying $\mu^2 \equiv D \pmod m$, we let $I \subset \mathbb{Z}[\sqrt{D}]$ be the corresponding ideal via proposition \ref{proposition:rootsideals}.
There is some $I_l$ and totally positive $\xi$ such that $I = \xi I_l$, and hence there is a $\gamma \in \mathrm{SL}(2,\mathbb{Z})$ such that
\begin{equation}
  \label{eq:basischange}
  \gamma \mathfrak{B}_l
  \begin{pmatrix}
    \xi^{(1)} & 0 \\
    0 & \xi^{(2)}
  \end{pmatrix}
  =
  \begin{pmatrix}
    1 & \mu \\
    0 & m
  \end{pmatrix}
  \begin{pmatrix}
    \sqrt{D} & - \sqrt{D} \\
    1 & 1
  \end{pmatrix}
  .
\end{equation}
We observe that considering $\mu$ as a residue class modulo $m$ is equivalent to considering the coset $\Gamma_\infty \gamma$.
More significantly, since $\xi$ is only determined up to a multiple of a totally positive unit, we also need only consider the coset $\gamma \Gamma_l$, where
\begin{equation}
  \label{eq:Gammaldef}
  \Gamma_l = \left\{ \pm \mathfrak{B}_l
    \begin{pmatrix}
      \varepsilon_0^k & 0 \\
      0 & \varepsilon_0^{-k}
    \end{pmatrix}
    \mathfrak{B}_l^{-1} : k \in \mathbb{Z} \right\}
\end{equation}
with $\varepsilon_0$, chosen to be greater than $1$, generating the group of totally positive units in $\mathbb{Z}[\sqrt{D}]$.
Thus equation \eqref{eq:basischange} defines a unique double coset $\Gamma_{\infty} \gamma \Gamma_{l}$. Furthermore $\Gamma_l$ satisfies
\begin{equation}
  \label{eq:Gammaldef1}
  \Gamma_l = \mathrm{SL}(2,\mathbb{Z}) \cap \left\{ \pm \mathfrak{B}_l
  \begin{pmatrix}
    t & 0 \\
    0 & t^{-1}
  \end{pmatrix}
  \mathfrak{B}_l^{-1} : t > 0 \right\}
\end{equation}
exactly because $\varepsilon I_l = I_l$ if and only if $\varepsilon$ is a unit in $\mathbb{Z}[\sqrt{D}]$.

Conversely, suppose that there is a totally positive $\xi$ (necessarily in $\mathbb{Q}(\sqrt{D})$), and $\gamma \in \mathrm{SL}(2,\mathbb{Z})$ such that
\begin{equation}
  \label{eq:basischangeconverse}
  \gamma \mathfrak{B}_l
  \begin{pmatrix}
    \xi^{(1)} & 0 \\
    0 & \xi^{(2)}
  \end{pmatrix}
  =
  \begin{pmatrix}
    1 & * \\
    0 & * 
  \end{pmatrix}
  \begin{pmatrix}
    \sqrt{D} & - \sqrt{D} \\
    1 & 1
  \end{pmatrix}
  .
\end{equation}
We claim that then we in fact have
\begin{equation}
  \label{eq:basischangeconverse2}
  \gamma \mathfrak{B}_l
  \begin{pmatrix}
    \xi^{(1)} & 0 \\
    0 & \xi^{(2)}
  \end{pmatrix}
  =
  \begin{pmatrix}
    1 & \mu \\
    0 & m
  \end{pmatrix}
  \begin{pmatrix}
    \sqrt{D} & - \sqrt{D} \\
    1 & 1
  \end{pmatrix}
  ,
\end{equation}
where $m > 0$ and $\mu^2 \equiv D \pmod m$.

To verify this, we first note that if $\xi \in I_l^{-1}$, then the left side of \eqref{eq:basischangeconverse} is the basis for an ideal $I \subset \mathbb{Z}[\sqrt{D}]$, so then \eqref{eq:basischangeconverse2} follows by an application of proposition \ref{proposition:rootsideals}.
Now to show $\xi \in I_l^{-1}$ follows from \eqref{eq:basischangeconverse}, we fix a basis $\{\overline{\beta}_{l1}, \overline{\beta}_{l2} \}$ of $I_l^{-1}$ and write $\xi = c_1 \overline{\beta}_{l1} + c_2 \overline{\beta}_{l2}$.
Our goal then is to show that $c_1$ and $c_2$ are in fact integers.
Since $I_l^{-1}I_l = \mathbb{Z}[\sqrt{D}]$, we can define integers $b_{lijk}$ by
\begin{equation}
  \label{eq:blijkdef}
  \overline{\beta}_{li} \beta_{lj} = b_{lij1} \sqrt{D} + b_{lij2},
\end{equation}
and we let
\begin{equation}
  \label{eq:Blidef}
  B_{li} =
  \begin{pmatrix}
    b_{li11} & b_{li12} \\
    b_{li21} & b_{li22}
  \end{pmatrix}
  .
\end{equation}
It follows that
\begin{equation}
  \label{eq:Bi}
  \mathfrak{B}_l
  \begin{pmatrix}
    \xi^{(1)} & 0 \\
    0 & \xi^{(2)}
  \end{pmatrix}
  = (c_1B_{l1} + c_2B_{l2})
  \begin{pmatrix}
    \sqrt{D} & - \sqrt{D} \\
    1 & 1
  \end{pmatrix}
  ,
\end{equation}
and so, under the assumption \eqref{eq:basischangeconverse}, if $\gamma = \begin{pmatrix} * & * \\ c & d \end{pmatrix}$, then
\begin{equation}
  \label{eq:Bleq}
  B_l
  \begin{pmatrix}
    c_1 \\
    c_2 
  \end{pmatrix}
  =
  \begin{pmatrix}
    d \\
    -c
  \end{pmatrix}
  ,
\end{equation}
where
\begin{equation}
  \label{eq:Bldef}
  B_l =
  \begin{pmatrix}
    b_{l111} & b_{l211} \\
    b_{l121} & b_{l221}
  \end{pmatrix}
  .
\end{equation}
The integrality of $c_1$ and $c_2$ is now an immediate consequence of the following lemma.

\begin{lemma}
  \label{lemma:Bl}
  For all $l$, we have $B_l \in \mathrm{GL}(2, \mathbb{Z})$. 
\end{lemma}

\begin{proof}
  We first observe that if $c_1', c_2'$ were such that
  \begin{equation}
    \label{eq:Blnull}
    B_l
    \begin{pmatrix}
      c_1' \\
      c_2'
    \end{pmatrix}
    =
    \begin{pmatrix}
      0 \\
      0
    \end{pmatrix}
    ,
  \end{equation}
  then, setting $\xi'= c_1'\overline{\beta}_{l1} + c_2'\overline{\beta}_{l2}$, we would have from \eqref{eq:Bi} that $\xi I_l \subset \mathbb{Z}$.
  This is a contradiction if $\xi' \neq 0$, and so we must have $\det B_l \neq 0$.

  Now suppose that $\xi'' = c_1'' \overline{\beta}_{l1} + c_2''\overline{\beta}_{l2}$ is a primitive vector in $I_l^{-1}$, meaning $a^{-1} \xi'' \notin I_l^{-1}$ for all rational integers $a > 1$, or equivalently $\gcd(c_1'', c_2'') = 1$.
  Then the ideal $\xi'' I_l \subset \mathbb{Z}[\sqrt{D}]$ is not divisible by any rational integer greater than $1$.
  Now by the construction of the $B_{li}$, see (\ref{eq:Bi}), the rows of
  \begin{equation}
    \label{eq:Blbasis}
    ( c_1''B_{l1} + c_2'' B_{l2})
    \begin{pmatrix}
      \sqrt{D} & - \sqrt{D} \\
      1 & 1
    \end{pmatrix}
  \end{equation}
  form a basis of $\xi'' I_l$.
  Putting $c_1'' B_{l1} + c_2'' B_{l2}$ into Hermite normal form, proposition \ref{proposition:rootsideals} implies that the entries in the first column of $c_1''B_{l1} + c_2'' B_{l2}$, i.e. the entries of $B_l \begin{pmatrix} c_1'' \\ c_2'' \end{pmatrix}$, are coprime.
  Hence $B_l$ maps pairs of coprime integers to coprime integers, and thus $B_l \in \mathrm{GL}(2, \mathbb{Z})$. 
\end{proof}

We record the above observations in the following proposition.

\begin{proposition}
  \label{proposition:rootsgeodesics}
  Let $m$ be a positive integer and $\mu \pmod m$ satisfy $\mu^2 \equiv D \pmod m$.
  Then there exists a unique $l$, $1\leq l \leq h^+(D)$, and a unique double coset $\Gamma_\infty \gamma \Gamma_{l} \in \Gamma_\infty \backslash \mathrm{SL}(2,\mathbb{Z}) / \Gamma_l$ such that
  \begin{equation}
    \label{eq:rootparametrization}
    \begin{pmatrix}
      1 & \frac{\mu}{m} \\
      0 & 1
    \end{pmatrix}
    \begin{pmatrix}
      1 & 0 \\
      0 & m
    \end{pmatrix}
    \begin{pmatrix}
      \sqrt{D} & -\sqrt{D} \\
      1 & 1
    \end{pmatrix}
    = \gamma \mathfrak{B}_l
    \begin{pmatrix}
      \xi^{(1)} & 0 \\
      0 & \xi^{(2)}
    \end{pmatrix}
  \end{equation}
  for some totally positive $\xi$.

  Conversely, if for $\Gamma_\infty \gamma \Gamma_{l} \in \Gamma_\infty \backslash \mathrm{SL}(2,\mathbb{Z}) / \Gamma_{l}$ there exists a totally positive $\xi$ such that
  \begin{equation}
    \label{eq:rootparametrization2}
    \gamma \mathfrak{B}_l
    \begin{pmatrix}
      \xi^{(1)} & 0 \\
      0 & \xi^{(2)}
    \end{pmatrix}
    =
    \begin{pmatrix}
      1 & * \\
      0 & *
    \end{pmatrix}
    \begin{pmatrix}
      \sqrt{D} & -\sqrt{D} \\
      1 & 1
    \end{pmatrix}
    ,
  \end{equation}
  then in fact
  \begin{equation}
    \label{eq:rootparametrization3}
    \gamma \mathfrak{B}_l
    \begin{pmatrix}
      \xi^{(1)} & 0 \\
      0 & \xi^{(2)}
    \end{pmatrix}
    =
    \begin{pmatrix}
      1 & \frac{\mu}{m} \\
      0 & 1
    \end{pmatrix}
    \begin{pmatrix}
      1 & 0 \\
      0 & m
    \end{pmatrix}    
    \begin{pmatrix}
      \sqrt{D} & -\sqrt{D} \\
      1 & 1
    \end{pmatrix}
  \end{equation}
  with $m$ a positive integer and $\mu \pmod m$ satisfying $\mu^2 \equiv D \pmod m$.
  Moreover such $m$ and $\mu \pmod m$, if they exist, are unique for a given $\Gamma_\infty \gamma \Gamma_l \in \Gamma_\infty \backslash \mathrm{SL}(2,\mathbb{Z}) / \Gamma_{l}$.   
\end{proposition}

Using the identification of $G / \{\pm I\}$ with the unit tangent bundle of $\mathbb{H}$, we interpret proposition \ref{proposition:rootsgeodesics} geometrically in proposition \ref{proposition:geometricinterpretation} below.
We define oriented geodesics $\bm{c}_l$, $1\leq l \leq h^+(D)$ by
\begin{equation}
  \label{eq:cldef}
  \bm{c}_l = \left\{ (\det \mathfrak{B}_l)^{-\frac{1}{2}}\mathfrak{B}_l
    \begin{pmatrix}
      t & 0 \\
      0 & t^{-1}
    \end{pmatrix}
    : t > 0 \right\}. 
\end{equation}
In this case $\bm{c}_l$ defines a closed geodesic in $\mathrm{SL}(2,\mathbb{Z}) \backslash \mathbb{H}$ and $\Gamma_l$ is the stabilizer in $\mathrm{SL}(2,\mathbb{Z})$ of the geodesic $\bm{c}_l$, i.e. $\Gamma_l = \Gamma_{\bm{c}_l}$ in the general notation introduced in section \ref{sec:geodesic}. 

\begin{proposition}
  \label{proposition:geometricinterpretation}
  Let $m > 0$ and $\mu\pmod m$ satisfy $\mu^2 \equiv D \pmod m$.
  Then there is a unique $l$ and double coset $\Gamma_\infty \gamma \Gamma_{l} \in \Gamma_\infty \backslash \mathrm{SL}(2,\mathbb{Z}) / \Gamma_l$ such that
  \begin{equation}
    \label{eq:geometricinterpretation}
    z_{\gamma \bm{c}_l} \equiv \frac{\mu}{m} + \ii \frac{\sqrt{D}}{m}
  \end{equation}
  modulo $\Gamma_\infty$.
   
  Conversely, given $l$ and a double coset $\Gamma_\infty \gamma \Gamma_l$ such that $\gamma \bm{c}_l$ is positively oriented, there exist unique $m$ and $\mu \pmod m$ satisfying the above.
\end{proposition}

\begin{proof}
  We observe that in view of \eqref{eq:Iwasawa}, the matrices
  \begin{equation}
    \label{eq:Harrows}
     \begin{pmatrix}
      * & * \\
      0 & * 
    \end{pmatrix}
    \begin{pmatrix}
      \cos \frac{\pi}{4} & -\sin \frac{\pi}{4} \\
      \sin \frac{\pi}{4} & \cos \frac{\pi}{4}
    \end{pmatrix}
  \end{equation}
  are identified in the unit tangent bundle of $\mathbb{H}$ with points having tangent vectors pointing directly to the right, and so the intersection of the set of these matrices with the set of matrices having the form
  \begin{equation}
    \label{eq:geodesic}
    \gamma (\det \mathfrak{B}_l)^{-\frac{1}{2}} \mathfrak{B}_l
    \begin{pmatrix}
      * & 0 \\
      0 & *
    \end{pmatrix}
  \end{equation}
  is exactly the point on the geodesic $\gamma \bm{c}_l$ having horizontal tangent vector pointing to the right, if such a point exists, and otherwise this intersection is empty.
  We note that such a point on $\gamma \bm{c}_l$ exists exactly when $\gamma\bm{c}_l$ is positively oriented and in this case gives the point on $\gamma \bm{c}_l$ with largest imaginary part.
  Scaling and rewriting the right side of \eqref{eq:rootparametrization} as
  \begin{flalign}
    \label{eq:rootparametrizationrewrite}
    & (2\sqrt{D})^{-\frac{1}{2}}
    \begin{pmatrix}
      1 & \frac{\mu}{m} \\
      0 & 1
    \end{pmatrix}
    \begin{pmatrix}
      m^{-\frac{1}{2}} & 0 \\
      0 & m^{\frac{1}{2}}
    \end{pmatrix}
    \begin{pmatrix}
      \sqrt{D} & -\sqrt{D} \\
      1 & 1
    \end{pmatrix}
    \\ \nonumber
    & \quad =
    \begin{pmatrix}
      1 & \frac{\mu}{m} \\
      0 & 1
    \end{pmatrix}
    \begin{pmatrix}
      \left(\frac{\sqrt{D}}{m} \right)^{\frac{1}{2}} & 0 \\
      0 & \left(\frac{m}{\sqrt{D}}\right)^{\frac{1}{2}}
    \end{pmatrix}
    \begin{pmatrix}
      \cos \frac{\pi}{4} & - \sin \frac{\pi}{4} \\
      \sin \frac{\pi}{4} & \cos \frac{\pi}{4}
    \end{pmatrix}
    ,
  \end{flalign}
  we finish the proof by appealing to proposition \ref{proposition:rootsgeodesics}.
\end{proof}

\subsection{Congruence subgroups}
\label{sec:congruence}

So far we have only considered the orbit of geodesics under the action of $\mathrm{SL}(2,\mathbb{Z})$, however, as reflected in theorem \ref{theorem:topsandroots}, the orbits under $\Gamma_0(n)$ also give interesting arithmetic information.
Indeed, by picking cosets $\Gamma_0(n) \gamma_l \in \Gamma_0(n) \backslash \mathrm{SL}(2,\mathbb{Z})$ correctly, one can obtain all $m$ and roots $\mu\pmod m$ satisfying $m \equiv 0 \pmod n$ and $\mu \equiv \nu \pmod n$, for a fixed root $\nu \pmod n$, by considering the $\Gamma_0(n)$-orbits of the geodesics $\gamma_l \bm{c}_l$.

To establish this, we first show that if a root $\mu_1 \pmod {m_1}$ satisfies $m_1 \equiv 0 \pmod n$ and $\mu_1 \equiv \nu \pmod n$, and if $\mu_2 \pmod {m_2}$ is another root in the same $\Gamma_0(n)$-orbit as $\mu_1 \pmod{m_1}$, then $\mu_2 \pmod {m_2}$ also satisfies $m_2 \equiv 0 \pmod n$, $\mu_2 \equiv \nu \pmod n$.
We then show that if roots $\mu_1 \pmod {m_1}$ and $\mu_2 \pmod {m_2}$ in the same $\mathrm{SL}(2,\mathbb{Z})$-orbit satisfy $m_1 \equiv m_2 \equiv 0 \pmod n$ and $\mu_1 \equiv \mu_2 \equiv \nu \pmod n$, then in fact the roots are in the same $\Gamma_0(n)$-orbit.
Finally, for sake of completeness, we show that for every $1\leq l \leq h^+(D)$ there is $\gamma_l = \gamma_l(n, \nu)$ such that the top of $\gamma_l \bm{c}_l$ satisfies the congruence conditions. 

Although it may not be the quickest route to demonstrating the first point, we proceed by recording the following corollary of proposition \ref{proposition:rootsgeodesics}.

\begin{corollary}
  \label{corollary:rootsbqfs}
  As $\gamma = \begin{pmatrix} a & b \\ c & d \end{pmatrix}$ runs through representatives of the double cosets in $\Gamma_\infty \backslash \mathrm{SL}(2,\mathbb{Z}) / \Gamma_{l}$ such that $\gamma \bm{c}_l$ is positively oriented,
  \begin{equation}
    \label{eq:mmubqf}
    \begin{pmatrix}
      \mu \\
      m
    \end{pmatrix}
    = (\det B_l)
    \begin{pmatrix}
      a & b \\
      c & d
    \end{pmatrix}
    \begin{pmatrix}
      b_{l112}b_{l211} - b_{l212}b_{l111} & b_{l112}b_{l221} - b_{l212}b_{l121} \\
      b_{l122}b_{l211} - b_{l111}b_{l222} & b_{l122}b_{l221} - b_{l121}b_{l222}
    \end{pmatrix}
    \begin{pmatrix}
      c \\
      d
    \end{pmatrix}
    ,
  \end{equation}
  parametrise all $\mu \pmod {m}$ satisfying $m > 0$ and $\mu^2 \equiv D \pmod m$.
\end{corollary}

This corollary recovers the classic parametrisation of roots of $\mu^2 \equiv D \pmod {m}$ in terms of binary quadratic forms with discriminant $4D$.
Indeed, \eqref{eq:mmubqf} gives $m$ as a binary quadratic form evaluated at $c, d$, and one can check that this form has discriminant $4D$.
This binary quadratic form automatically comes with a Bhargava cube, see \cite{Bhargava2004}, given by the matrices $B_{l1}, B_{l2}$, which is associated with the ideals $\mathbb{Z}[\sqrt{D}], I_l, I_l^{-1}$.
We also note that \eqref{eq:mmubqf} gives a criterion on $\gamma$ for the geodesic $\gamma \bm{c}_l$ to be positively oriented, namely that the $m$ given by \eqref{eq:mmubqf} is in fact positive.

\begin{proof}
  We recall that for $\xi = c_1 \overline{\beta}_{l1} + c_2 \overline{\beta}_{l2}$, 
  \begin{equation}
    \label{eq:Bi1}
    \gamma \mathfrak{B}_l
    \begin{pmatrix}
      \xi^{(1)} & 0 \\
      0 & \xi^{(2)}
    \end{pmatrix}
    = \gamma (c_1 B_{l1} + c_2 B_{l2})
    \begin{pmatrix}
      \sqrt{D} & - \sqrt{D} \\
      1 & 1
    \end{pmatrix}
    ,
  \end{equation}
  and so \eqref{eq:rootparametrization}, \eqref{eq:rootparametrization3} give
  \begin{equation}
    \label{eq:explicitroot}
    \begin{pmatrix}
      1 & \mu \\
      0 & m
    \end{pmatrix}
    = \gamma (c_1 B_{l1} + c_2 B_{l2}).
  \end{equation}
  Equating the first columns, we see that
  \begin{equation}
    \label{eq:c1c2}
    \begin{pmatrix}
      c_1 \\
      c_2
    \end{pmatrix}
    = B_{l}^{-1}
    \begin{pmatrix}
      d \\
      -c
    \end{pmatrix}
    ,
  \end{equation}
  from which \eqref{eq:mmubqf} follows from the second column of \eqref{eq:explicitroot}.
\end{proof}

Having \eqref{eq:mmubqf}, we immediately see that if $\gamma$ is replaced by $\gamma_1 \gamma$ for some $\gamma_1 \in \Gamma_0(n)$, $m \pmod n$ is replaced by a multiple, and, if $m \equiv 0 \pmod n$, $\mu \pmod n$ is unchanged.
Hence the conditions $m \equiv 0 \pmod n$ and $\mu \equiv \nu \pmod n$ are invariant under $\Gamma_0(n)$, thus finishing the first part of our examination of the $\Gamma_0(n)$ orbits.
The second part is completed by the following lemma.

\begin{lemma}
  \label{lemma:congruenceorbit}
  Let $m_1$, $m_2$ be positive integers and $\mu_1 \pmod {m_1}$, $\mu_2\pmod {m_2}$ satisfy $\mu_1^2 \equiv D \pmod {m_1}$, $\mu_2^2 \equiv D \pmod {m_2}$.
  Suppose that $m_1 \equiv m_2 \equiv 0 \pmod n$ and that $\mu_1 \equiv \mu_2 \equiv \nu \pmod n$ where $\nu\pmod n$ is a fixed root of $\nu^2 \equiv D \pmod n$.
  Then, if there exists $\gamma \in \mathrm{SL}(2,\mathbb{Z})$ such that
  \begin{equation}
    \label{eq:congruenceorbit}
    \begin{pmatrix}
      1 & \mu_1 \\
      0 & m_1
    \end{pmatrix}
    \begin{pmatrix}
      \sqrt{D} & -\sqrt{D} \\
      1 & 1 
    \end{pmatrix}
    \begin{pmatrix}
      \xi^{(1)} & 0 \\
      0 & \xi^{(2)}
    \end{pmatrix}
    = \gamma
    \begin{pmatrix}
      1 & \mu_2 \\
      0 & m_2
    \end{pmatrix}
    \begin{pmatrix}
      \sqrt{D} & -\sqrt{D} \\
      1 & 1 
    \end{pmatrix}
  \end{equation}
  for some totally positive $\xi$, then in fact $\gamma \in \Gamma_0(n)$.
\end{lemma}

\begin{proof}
  Let $I_1$, $I_2$ be the ideals corresponding to $\mu_1 \pmod {m_1}$, $\mu_2 \pmod {m_2}$.
  Then since
  \begin{equation}
    \label{eq:idealcontain}
    \begin{pmatrix}
      1 & \mu_j \\
      0 & m_j
    \end{pmatrix}
    =
    \begin{pmatrix}
      1 & k_j \\
      0 & \frac{m_j}{n}
    \end{pmatrix}
    \begin{pmatrix}
      1 & \nu \\
      0 & n
    \end{pmatrix}
  \end{equation}
  for some integers $k_j$, we have $I_1, I_2 \subset I$, where $I$ is the ideal corresponding to $\nu \pmod {n}$.
  Similarly, we have
  \begin{equation}
    \label{eq:idealcontain1}
    \begin{pmatrix}
      1 & \mu_1 \\
      0 & m_1
    \end{pmatrix}
    =
    \begin{pmatrix}
      1 & k \\
      0 & n
    \end{pmatrix}
    \begin{pmatrix}
      1 & \nu_1 \\
      0 & n_1
    \end{pmatrix}
    ,
  \end{equation}
  where $n_1 = \frac{m_1}{n}$ and $\nu_1 \pmod {n_1}$ is the root corresponding to the ideal $I_1 I^{-1}$.
  Now \eqref{eq:congruenceorbit} implies that $\xi I_1 = I_2$, so $\xi I_1 I^{-1} = I_2 I^{-1} \subset \mathbb{Z}[\sqrt{D}]$, whence
  \begin{equation}
    \label{eq:Adef}
    \begin{pmatrix}
      1 & \nu_1 \\
      0 & n_1
    \end{pmatrix}
    \begin{pmatrix}
      \sqrt{D} & - \sqrt{D} \\
      1 & 1
    \end{pmatrix}
    \begin{pmatrix}
      \xi^{(1)} & 0 \\
      0 & \xi^{(2)}
    \end{pmatrix}
    = A
    \begin{pmatrix}
      \sqrt{D} & -\sqrt{D} \\
      1 & 1
    \end{pmatrix}
  \end{equation}
  for some integer matrix $A$.
  Rewriting the left side of \eqref{eq:Adef} as
  \begin{flalign}
    \label{eq:Adefrewrite}
    &
    \begin{pmatrix}
      1 & k \\
      0 & n
    \end{pmatrix}
    ^{-1}
    \begin{pmatrix}
      1 & \mu_1 \\
      0 & m_1
    \end{pmatrix}
    \begin{pmatrix}
      \sqrt{D} & - \sqrt{D} \\
      1 & 1
    \end{pmatrix}
    \begin{pmatrix}
      \xi^{(1)} & 0 \\
      0 & \xi^{(2)}
    \end{pmatrix}
    \\ \nonumber
    & \quad = 
    \begin{pmatrix}
      1 & k \\
      0 & n
    \end{pmatrix}
    ^{-1} \gamma
    \begin{pmatrix}
      1 & \mu_2 \\
      0 & m_2
    \end{pmatrix}
    \begin{pmatrix}
      \sqrt{D} & - \sqrt{D} \\
      1 & 1
    \end{pmatrix}
    ,
  \end{flalign}
  we obtain
  \begin{equation}
    \label{eq:gammaeq}
    \begin{pmatrix}
      1 & k \\
      0 & n
    \end{pmatrix}
    ^{-1} \gamma
    \begin{pmatrix}
      1 & \mu_2 \\
      0 & m_2
    \end{pmatrix}
    = A.
  \end{equation}
  Since $A$ has integer entries, an inspection of the $(2,1)$ entry of \eqref{eq:gammaeq} shows that $\gamma \in \Gamma_0(n)$. 
\end{proof}

\begin{lemma}
  \label{lemma:orbitexist}
  For each $1\leq l \leq h^+(D)$, there exists $\gamma_l = \gamma_l(n, \nu)$ such that
  \begin{equation}
    \label{eq:orbitexist}
    z_{\gamma_l \bm{c}_l} = \frac{\mu}{m} + \mathrm{i} \frac{\sqrt{D}}{m}
  \end{equation}
  with $m \equiv 0 \pmod n$ and $\mu \equiv \nu \pmod n$. 
\end{lemma}

\begin{proof}
  By proposition \ref{proposition:rootsideals}, there exists an ideal $I_0 \subset \mathbb{Z}[\sqrt{D}]$ corresponding to the root $\nu \pmod n$.
  As $I_1 \subset \mathbb{Z}[\sqrt{D}]$ ranges over the narrow ideal class of $I_lI_0^{-1}$, the integers $N(I_1)$ are the values of an integer binary quadratic form, see corollary \ref{corollary:rootsbqfs}, and so we may choose $I_1$ with $\gcd(n, N(I_1)) = 1$ and without rational integer divisors.
  We let $\mu_1 \mod m_1$ denote the root corresponding to $I_1$ via proposition \ref{proposition:rootsideals}.

  Letting $m_1 \overline{m}_1 + n \overline{n_1} = 1$ and setting $\mu = \nu m_1\overline{m}_1 + \mu_1 n \overline{n}$, we compute
  \begin{multline}
    \label{eq:I0I1}
    I_0 I_1 = (n, \nu + \sqrt{D})(m_1, \mu_1 + \sqrt{D}) = (m_1 n, m_1 \nu + m_1 \sqrt{D}, n \mu_1 + n \sqrt{D}, \mu_1 \nu + D + (\mu_1 + \nu)\sqrt{D}) \\
    = (m_1 n, \mu + \sqrt{D}, \mu_1 \nu + D + (\mu_1 + \nu) \sqrt{D}) = (m_1 n, \mu + \sqrt{D}),
  \end{multline}
  where the last equality follows from $\mu_1 \nu + D - (\mu_1 + \nu)\mu \equiv D - \mu^2 \equiv 0$ modulo both $n$ and $m_1$.
  The ideal $I = I_0 I_1$ therefore corresponds to the root $\mu \pmod m_1 n$ and is in the narrow ideal class of $I_l$.
  Equation (\ref{eq:orbitexist}) then follows by the machinery of section \ref{sec:roots2}. 
\end{proof}

Corollary \ref{corollary:rootsbqfs}, lemma \ref{lemma:congruenceorbit}, and lemma \ref{lemma:orbitexist} show that for a fixed $\nu \pmod n$ and each $l$, $1\leq l \leq h^+(D)$, there exists a unique coset $\Gamma_0(n) \gamma_l \in \Gamma_0(n) \backslash \mathrm{SL}(2,\mathbb{Z})$ such that roots $\mu \pmod m$ with $m \equiv 0 \pmod n$ and $\mu \equiv \nu \pmod n$ are found in the positively oriented tops of the geodesics $\gamma \gamma_l \bm{c}_l$ as $\gamma$ ranges over $\Gamma_0(n)$ as in proposition \ref{proposition:geometricinterpretation}.
This is the bulk of theorem \ref{theorem:topsandroots}; it only remains to show that the closed geodesics obtained from projecting $\gamma_l \bm{c}_l$ to $\Gamma_0(n) \backslash \mathbb{H}$ all have the same length.
In fact we show that the closed geodesic in $\Gamma_0(n) \backslash \mathbb{H}$ coming from $\gamma_l\bm{c}_l$ has the same length as the closed geodesic in $\mathrm{SL}(2,\mathbb{Z}) \backslash \mathbb{H}$ coming from $\bm{c}_l$.
This is not a trivial observation; in general the splitting behaviour of geodesics in finite covering surfaces is quite complicated, being analogous to the splitting of primes in number fields, see \cite{Sarnak1980}.

Understanding the length of $\gamma_l \bm{c}_l$ in $\Gamma_0(n) \backslash \mathbb{H}$ is equivalent to understanding the stabilizer of $\gamma_l\bm{c}_l$ in $\Gamma_0(n)$.
It is clear that there is some positive integer $k$ such that
\begin{equation}
  \label{eq:congruencestabilizer}
  \gamma_l \mathfrak{B}_l
  \begin{pmatrix}
    \varepsilon_0^{k} & 0 \\
    0 & \varepsilon_0^{-k}
  \end{pmatrix}
  \mathfrak{B}_l^{-1} \gamma_l^{-1} \in \Gamma_0(n),
\end{equation}
and the smallest such $k$ gives the generator of the stabilizer.
Moreover, the length of $\gamma \bm{c}_l$ in $\Gamma_0(n) \backslash \mathbb{H}$ has length this $k$ times the length of $\bm{c}_l$ in $\mathrm{SL}(2,\mathbb{Z}) \backslash \mathbb{H}$.
The following lemma, which is essentially a corollary of lemma \ref{lemma:congruenceorbit}, shows that in our setting we have in fact $k = 1$.

\begin{lemma}
  \label{lemma:lift}
  For all $l$, we have
  \begin{equation}
    \label{eq:lift}
    \gamma_l \mathfrak{B}_l
    \begin{pmatrix}
      \varepsilon_0 & 0 \\
      0 & \varepsilon_0^{-1} 
    \end{pmatrix}
    \mathfrak{B}_l^{-1} \gamma_l^{-1} \in \Gamma_0(n).
  \end{equation}
\end{lemma}

\begin{proof}
  We let $\gamma_0$ be the matrix on the left side of \eqref{eq:lift}, so $\gamma_0$ satisfies
  \begin{equation}
    \label{eq:gamma0def}
    \gamma_0 \gamma_l \mathfrak{B}_l = \gamma_l \mathfrak{B}_l
    \begin{pmatrix}
      \varepsilon_0 & 0 \\
      0 & \varepsilon_0^{-1}
    \end{pmatrix}
    .
  \end{equation}
  From the way $\gamma_l$ was chosen, for some totally positive $\xi$ we have
  \begin{equation}
    \label{eq:gammaldef}
    \gamma_l \mathfrak{B}_l
    \begin{pmatrix}
      \xi^{(1)} & 0 \\
      0 & \xi^{(2)}
    \end{pmatrix}
    =
    \begin{pmatrix}
      1 & \mu \\
      0 & m
    \end{pmatrix}
    \begin{pmatrix}
      \sqrt{D} & -\sqrt{D} \\
      1 & 1
    \end{pmatrix}
  \end{equation}
  with $m \equiv 0 \pmod n$.
  Since the diagonal matrices commute, we put \eqref{eq:gammaldef} into \eqref{eq:gamma0def} to obtain
  \begin{equation}
    \label{eq:lift1}
    \gamma_0
    \begin{pmatrix}
      1 & \mu \\
      0 & m 
    \end{pmatrix}
    \begin{pmatrix}
      \sqrt{D} & -\sqrt{D} \\
      1 & 1
    \end{pmatrix}
    =
    \begin{pmatrix}
      1 & \mu \\
      0 & m
    \end{pmatrix}
    \begin{pmatrix}
      \sqrt{D} & -\sqrt{D} \\
      1 & 1
    \end{pmatrix}
    \begin{pmatrix}
      \varepsilon_0 & 0 \\
      0 & \varepsilon_0^{-1}
    \end{pmatrix}
    ,
  \end{equation}
  so lemma \ref{lemma:congruenceorbit} in the special case $m_1 = m_2 = m$ and $\mu_1 = \mu_2 = \mu$ gives $\gamma_0 \in \Gamma_0(n)$. 
\end{proof}

This completes the proof of theorem \ref{theorem:topsandroots}.

\subsection{Negative $D$}\label{sec:negative}

For completeness, let us now briefly discuss the case of negative $D$. Here the situation is similar as for positive $D$; in fact simpler, since the tops of geodesics are replaced by $\Gamma$-orbits of points in $\HH$, whose statistical distribution is well understood. The results mirroring those in section \ref{sec:geodesic} are proved in \cite{MarklofVinogradov2018}, and the pair correlation density is calculated in \cite{KelmerKontorovich2015,MarklofVinogradov2018}. For further related studies see also  \cite{BocaPasolPopaZaharescu2014,BocaPopaZaharescu2014,Lutsko2020,RisagerRudnick2009,RisagerSodergren2017} and references therein.

Following \cite{Bykovskii1984,DukeFriedlanderIwaniec1995}, the key observation is the following. Denote by $\Gamma_z$ the stabiliser in $\Gamma_0(n)$ of a point $z\in\HH$.

\begin{theorem}
  \label{theorem:pointsandroots}
  Let  $D < 0$ square-free and $D\not\equiv 1 \bmod 4$. Fix $n > 0$ and $\nu \pmod n$ such that $\nu^2 \equiv D \pmod n$.
  Then there exist $z_1,\ldots,z_h\in\HH$ with the following properties:
  \begin{enumerate}[{\rm (i)}]
  \item For any integer $m > 0$ and $\mu \pmod m$ satisfying $\mu^2 \equiv D \pmod m$, $m \equiv 0 \pmod n$, and $\mu \equiv \nu \pmod n$, there is a unique $l$ and double coset $\Gamma_\infty \gamma \Gamma_{z_l} \in \Gamma_\infty \backslash \Gamma_0(n) / \Gamma_{z_l}$ such that
    \begin{equation}
      \label{eq:pointssandroots}
      \gamma z_l \equiv \frac{\mu}{m} + \ii \frac{\sqrt{-D}}{m} \pmod {\Gamma_{\infty}}.
    \end{equation}
  \item Conversely, given $l$ and double coset $\Gamma_\infty \gamma \Gamma_{z_l} \in \Gamma_\infty \backslash \Gamma_0(n) / \Gamma_{z_l}$ , there exist unique $m > 0$ and $\mu \pmod m$ satisfying $\mu^2 \equiv D \pmod m$, $m \equiv 0\pmod n$, and $\mu \equiv \nu \pmod n$ such that \eqref{eq:topsandroots} holds. 
  \end{enumerate}
\end{theorem}

Here $h$ is the class number $h(D)$ of the imaginary number field $\mathbb{Q}(\sqrt{D})$, and $\mathbb{Z}[\sqrt{D}]$ is its maximal order.
We also note that $\Gamma_{z_l} = \{\pm I\}$ unless $D = -1$, in which case $h(-1) = 1$ and when $n=1$ we may choose $z_1 = \ii$, so
\begin{equation}
  \label{eq:istabilizer}
  \Gamma_{\ii} = \left\{
    \begin{pmatrix}
      1 & 0 \\
      0 & 1
    \end{pmatrix}
    ,
    \begin{pmatrix}
      0 & -1 \\
      1 & 0
    \end{pmatrix}
    ,
    \begin{pmatrix}
      -1 & 0 \\
      0 & -1 
    \end{pmatrix}
    ,
    \begin{pmatrix}
      0 & 1 \\
      -1 & 0
    \end{pmatrix}
    \right\}.
\end{equation}

The proof of the above statement is analogous to that of theorem \ref{theorem:topsandroots}. The algebraic manipulations leading to \eqref{eq:basischange} are identical, although now over the complex numbers. In order to turn \eqref{eq:basischange} into a real equation, we multiply it from the right by the matrix $\begin{pmatrix}
    -\frac{\ii}{2} & \frac{1}{2} \\
    \frac{\ii}{2} &  \frac{1}{2}
  \end{pmatrix}$
to obtain, writing $\sqrt D=\ii\sqrt{-D}$ and $\xi^{(1)} = a + \ii b$ and $\xi^{(2)} = a - \ii b$,
\begin{equation}
  \label{eq:basischangeCC}
  \gamma \widetilde{\mathfrak{B}}_l
  \begin{pmatrix}
    a & -b \\
    b & a
  \end{pmatrix}
  =
  \begin{pmatrix}
    1 & \mu \\
    0 & m
  \end{pmatrix}
  \begin{pmatrix}
    \sqrt{-D} & 0 \\
    0 & 1
  \end{pmatrix}
\end{equation}
and
\begin{equation}
\widetilde{\mathfrak{B}}_l = \mathfrak{B}_l \begin{pmatrix}
    -\frac{\ii}{2} & \frac{1}{2} \\
    \frac{\ii}{2} &  \frac{1}{2}
  \end{pmatrix}
  = \begin{pmatrix}
    \beta_{l1}^{(1)} & \beta_{l1}^{(2)} \\
    \beta_{l2}^{(1)} & \beta_{l2}^{(2)}
  \end{pmatrix}
\begin{pmatrix}
    -\frac{\ii}{2} & \frac{1}{2} \\
    \frac{\ii}{2} &  \frac{1}{2}
  \end{pmatrix}
= \begin{pmatrix}
    \Im \beta_{l1}^{(1)} & \Re \beta_{l1}^{(1)} \\
    \Im \beta_{l2}^{(1)} & \Re \beta_{l2}^{(1)}
  \end{pmatrix} .
\end{equation}
Applying both sides of \eqref{eq:basischangeCC} to $\ii$ (as fractional linear transformations) yields 
\begin{equation}
  \label{eq:basischangeCC2}
  \gamma z_l  =
  \frac{\mu}{m} +\ii \frac{\sqrt{-D}}{m} , \qquad \text{with}\; z_l=\frac{\beta_{l1}^{(1)}}{\beta_{l2}^{(1)}},
\end{equation}
as required.

Theorem \ref{theorem:pointsandroots} shows that when $D < 0$, the normalised roots $\frac{\mu}{m}$, with the congruence conditions $m \equiv 0 \pmod m$ and $\mu \equiv \nu\pmod n$ appear as the real parts of points in a finite number of $\Gamma_0(n)$-orbits.
The results of \cite{MarklofVinogradov2018} in the `cuspidal observer' setting then immediately apply to give the analogues of theorems \ref{theorem:convergence} and \ref{theorem:moments}.
As pointed out in \cite{MarklofVinogradov2018}, these results are closely related to previous work on the distribution of angles in hyperbolic lattices, see \cite{BocaPasolPopaZaharescu2014,BocaPopaZaharescu2014,Lutsko2020,RisagerRudnick2009,RisagerSodergren2017} and the references therein.
The calculation of the pair correlation density in this setting is carried out in \cite{KelmerKontorovich2015} and \cite{MarklofVinogradov2018}, and thus give the analogue of our theorem \ref{thm1} when $D < 0$ with the expression for $w_D$ having a similar form to our (\ref{eq:paircorrelation2}), (\ref{eq:F+qvdef}) and (\ref{eq:F-qvdef}).

\bibliographystyle{plain}
\bibliography{references}

\end{document}